\documentclass[letterpaper, reqno,11pt]{article}
\usepackage[margin=1.0in]{geometry}
\usepackage{color,latexsym,amsmath,amssymb, amsthm, graphicx, bbm}
\usepackage{tikz}
\usetikzlibrary{arrows}

\numberwithin{equation}{section}

\newcommand{\RR}{\mathbb{R}}

\newcommand{\ZZ}{\mathbb{Z}}

\newcommand{\eps}{\varepsilon}

\newcommand{\vol}{\lambda}

\newtheorem{thm}{Theorem}[section]
\newtheorem{lem}[thm]{Lemma}

\newtheorem{prop}[thm]{Proposition}
\newtheorem{cor}[thm]{Corollary}

\newtheorem*{main-discV2Thm}{Theorem \ref{main-discV2}}
\newtheorem*{curvatureProjectionsThm}{Theorem \ref{curvatureProjections}$^*$}
\newtheorem*{convexProjectionsThm}{Theorem \ref{convexProjections}$^*$}
\newtheorem*{dimension-expanderThm}{Theorem \ref{dimension-expander}}
\newtheorem*{main-entropy-growthThm}{Theorem \ref{main-entropy-growth}}
\newtheorem*{main-energy-dispersionThm}{Theorem \ref{main-energy-dispersion}}

\theoremstyle{remark}
\newtheorem{rem}[thm]{Remark}
\newtheorem{example}{Example}
\newtheorem{defn}[thm]{Definition}

\begin{document}
\title{On the dimension of exceptional parameters for nonlinear projections, and the discretized Elekes-R\'onyai theorem}
\author{Orit E. Raz\thanks{Einstein Institute of Mathematics, The Hebrew University of Jerusalem, Jerusalem, Israel. oritraz@mail.huji.ac.il.} 
\and Joshua Zahl\thanks{Department of Mathematics, University of British Columbia, Vancouver, BC. jzahl@math.ubc.ca.}}
\maketitle

\begin{abstract}
We consider four related problems. (1) Obtaining dimension estimates for the set of exceptional vantage points for the pinned Falconer distance problem. (2) Nonlinear projection theorems, in the spirit of Kaufman, Bourgain, and Shmerkin. (3) The parallelizability of planar $d$-webs. (4) The Elekes-R\'onyai theorem on expanding polynomials. 

Given a Borel set $A$ in the plane, we study the set of exceptional vantage points, for which the pinned distance $\Delta_p(A)$ has small dimension, that is, close to $(\dim A)/2$. We show that if this set has positive dimension, then it must have very special structure. This result follows from a more general single-scale nonlinear projection theorem, which says that if $\phi_1,\phi_2,\phi_3$ are three smooth functions whose associated 3-web has non-vanishing Blaschke curvature, and if $A$ is a $(\delta,\alpha)_2$-set in the sense of Katz and Tao, then at least one of the images $\phi_i(A)$ must have measure much larger than $|A|^{1/2}$, where $|A|$ stands for the measure of $A$. We prove analogous results for $d$ smooth functions $\phi_1,\ldots,\phi_d$, whose associated $d$-web is not parallelizable. 

We use similar tools to characterize when bivariate real analytic functions are ``dimension expanding'' when applied to a Cartesian product: if $P$ is a bivariate real analytic function, then $P$ is either locally of the form $h(a(x) + b(y))$, or $P(A,B)$ has dimension at least $\alpha+c$ whenever $A$ and $B$ are Borel sets with Hausdorff dimension $\alpha$. Again, this follows from a single-scale estimate, which is an analogue of the Elekes-R\'onyai theorem in the setting of the Katz-Tao discretized ring conjecture. 
\end{abstract}

\section{Introduction}
We consider four related problems. (1) Obtaining dimension estimates for the set of exceptional vantage points for the pinned Falconer distance problem. (2) Nonlinear projection theorems, in the spirit of Kaufman, Bourgain, and Shmerkin. (3) The parallelizability of planar $d$-webs. (4) The Elekes-R\'onyai theorem on expanding polynomials. We will briefly discuss each of these items.

We will begin with projection theory, and specifically the study of exceptional parameters for nonlinear analogues of the linear projection function $\pi_e\colon\RR^2\to\RR$, which is the orthogonal projection onto $\operatorname{span}(e)$. A direction $e\in S^1$ is called exceptional for a planar set $A$ if the dimension of the projection $\pi_e(A)$ is smaller than the dimension of the projection in a ``typical'' direction. Results of Kaufman and Bourgain quantify the size of the set of exceptional directions. There are several results concerning nonlinear generalizations of the projection function $\pi_e$. An important example of a nonlinear projection is the pinned distance function $d_p(q) = |p-q|$. The theme of many of these results is that under suitable constraints, nonlinear projections are as well-behaved as linear projections, with respect to the size of the set of exceptional parameters. We discover a new and unexpected phenomenon, which says that often nonlinear projections are much better behaved than linear ones. 

Our results can also be expressed in the discretized setting of $(\delta,\alpha)_2$ sets, which were introduced by Katz and Tao \cite{KT} as a single-scale model for planar Borel sets that have Hausdorff dimension $\alpha$. In brief, a $(\delta,\alpha)_2$ set is a union of about $\delta^{-\alpha}$ interior-disjoint squares of side-length $\delta$, which satisfy a Frostman-type non-concentration condition. It is straightforward to construct an example of a $(\delta,\alpha)_2$ set $A$ that has projection of size about $|A|^{1/2}$ in many different directions, or equivalently, a set that has small image under may different linear functions $\pi_e$. We prove that morally speaking, this phenomenon can only occur for linear functions. 

The fibers of each function $\pi_{e}$ are parallel lines; more generally, we say a set of smooth functions $\phi_1,\ldots,\phi_d$ determine a {\emph parallelizable (planar) $d$-web} if after a change of coordinates, the fibers of each function $\phi_i$ are parallel lines. We prove that if a set of smooth functions $\phi_1,\ldots,\phi_d$ do not determine a parallelizable planar $d$-web (in a certain quantitative sense), and if $A$ is a $(\delta,\alpha)_2$ set, then at least one of the images $\phi_i(A)$ must be much larger than $|A|^{1/2}$. 

Finally, we shall discuss the connection with the Elekes-R\'onyai theorem on expanding polynomials. If $\phi_1,\phi_2,\phi_3$ are smooth functions, then after a change of coordinates we can locally write $\phi_1(x,y)=x$, $\phi_2(x,y)=y$, and $\phi_3(x,y)=P(x,y)$ for some smooth function $P$. Suppose that $A$ is a $(\delta,\alpha)_2$ set, and that $\phi_1(A)$ and $\phi_2(A)$ each have size about $|A|^{1/2}$. If we write $A_1=\phi_1(A)$ and $A_2=\phi_2(A)$, then $A\subset A_1\times A_2$, and the sets $A$ and $A_1\times A_2$ have comparable size. Following the discussion above, we would like to show that $|P(A)|$ is much larger than $|A|^{1/2}$, or almost equivalently, that $P(A_1\times A_2)$ is much larger than $|A|^{1/2}$. If $P$ was a polynomial and if $A_1$ and $A_2$ were discrete sets of cardinality $n$, then Elekes and R\'onyai \cite{ER} proved that either $P(A_1\times A_2)$ is much larger than $n$, or $P$ has a very specific form (for example $P(x,y) = x+y$). We will reduce the results discussed above about non-parallelizable $d$-webs to a discretized analogue of the Elekes-R\'onyai theorem for smooth functions.

\subsection{Exceptional directions for nonlinear projections}\label{exceptionalDirectionsIntroSec}
In \cite{Ka}, Kaufman proved that if $A \subset \RR^2$ is a Borel set and if $\pi_e\colon\RR^2\to\RR$ is the orthogonal projection onto $\operatorname{span}(e)$, then $\pi_e(A)$ can only have small dimension for a small exceptional set of directions $e \in \mathbb{S}^1$. 
\begin{thm}[Kaufman]\label{KaufmanThm}
Let $A\subset\RR^2$ be a Borel set and let $s\leq \dim A$. Then 
\begin{equation}\label{KaufmanBd}
\dim \{ e \in \mathbb{S}^1 : \dim \pi_e(A) < s \} \leq s.
\end{equation}
\end{thm}
Here and throughout the paper, $\dim$ refers to Hausdorff dimension. Recently, Orponen and Shmerkin~\cite{OS} obtained a quantitatively small (but qualitatively very significant) improvement to the RHS of \eqref{KaufmanBd} whenever $0<s<\dim A$ and $s<1$. See also \cite{OS2} for related results on this theme. In a slightly different direction, Bourgain~\cite{B10} (see also \cite{Ob12}) obtained a large improvement over \eqref{KaufmanBd} when $s$ is close to $(\dim A)/2$. As a special case of his result, we have the following.
\begin{thm}[Bourgain]\label{bourgainThm}
Let $A\subset\RR^2$ be a Borel set. Then 
\begin{equation}\label{dimA2}
\dim\{e\in \mathbb{S}^1\colon \dim \pi_e(A) \leq (\dim A)/2 \}=0.
\end{equation}
\end{thm}
An example by Kaufman and Mattila \cite{KM} shows that \eqref{dimA2} cannot hold if $(\dim A)/2$ is replaced by $(\dim A)/2+c$ for any $c>0$. 


Estimates similar to those in Theorem \ref{KaufmanThm} hold for a broad class of generalized projections (see \cite{PS} and the references therein), and recently Shmerkin \cite{Shm} proved an analogue of Theorem \ref{bourgainThm} where the set of orthogonal projections $\{\pi_e\}_{e\in S^1}$ is replaced by a family $\{F_{\lambda}\}_{\lambda\in\Lambda}$ of nonsingular $C^2$ functions $F_\lambda\colon\RR^2\to\RR$, referred to as \emph{nonlinear projections}, that satisfy a natural transversality condition.

A particularly interesting class of nonlinear projections are the pinned distance functions $d_p(q) = |p-q|$. In \cite{Or}, Orponen proved a packing dimension version of the Falconer distance conjecture in the plane for Ahlfors-regular sets by showing that the nonlinear projection $d_p(q)$ could be analyzed using multi-scale averages of suitably chosen linear projections, and thus the problem of computing the size of the pinned distance set $\Delta_p(E)=\{d_p(q)\colon q\in E\}$ was closely related to the problem of computing the size of certain orthogonal projections of $E$ at various scales and locations. This powerful new idea has been extended in subsequent works \cite{KS, Shm2, Shm}. In particular, Shmerkin \cite{Shm} proved an analogue of Theorem \ref{bourgainThm} for pinned distances.


\begin{thm}[Shmerkin]\label{shmerkinPinnedDistThm}
Let $A\subset\RR^2$ be a Borel set. Then 
\begin{equation}\label{dimA20}
\dim\{p\in \RR^2\colon \dim \Delta_p(A) \leq (\dim A)/2 \}=0.
\end{equation}
\end{thm}

The following example due to Elekes \cite{El} shows that the quantity $(\dim A)/2$ on the LHS of \eqref{dimA20} cannot be replaced by $(\dim A)/2+c$ for any $c>0$. 
\begin{example}\label{elekesExample}
Let $0<\alpha<2$ and let $c>0$.  Let $X,Y,S\subset [0,1]$ be Borel sets (see e.g.~\cite{KM}) so that $\dim(X\times Y)=\alpha$, $\dim(S)>0$, and $\dim(X-sY)<\alpha/2+c$ for each $s\in S$.
For each $(a,b)\in X\times Y$, define $q_{a,b}=(a/2, \sqrt{1+b-a^2/4})$, and define $A = \{q_{a,b}\colon a\in X,\ b\in Y\}$. We have $\dim A = \alpha$. Define $K = S\times \{0\}$. 
Then for each $(x,y)=q_{a,b}\in A$ and $s\in S$, we have $|(s,0) - q_{a,b}|^2 =(s^2 + 1) + (b-sa)$. Thus for each $p=(s,0)\in K$, we have $\dim(\Delta_p(A))=\dim(B-sA)<\alpha/2+c$. 
We conclude that
\begin{equation}\label{exceptionalPinnedDistanceDirections}
\{p\in \RR^2\colon \dim \Delta_p(A) \leq (\dim A)/2+c \}\supset K,
\end{equation}
so in particular the set on the LHS of \eqref{exceptionalPinnedDistanceDirections} has positive dimension.  
\end{example}

In Example \ref{elekesExample}, the set $K$ is contained in a line. It is possible to construct slightly more complicated multi-scale examples where the set \eqref{exceptionalPinnedDistanceDirections} is not contained in one line, but can be contained in a union of countably many lines. 
The next theorem says that if $c>0$ is small, then this is essentially the only way that the set \eqref{exceptionalPinnedDistanceDirections} can have positive dimension. To make this statement precise, we will need the following definition.

\begin{defn}\label{nonConcentratedOnLineS}
We say a Borel set $K \subset \RR^2$ is \emph{curved} if there exist $\beta>0$ and $C>0$, and a Borel probability measure $\mu$ supported on $K$ so that for all $r>0$,
\begin{equation}\label{ballEstimateForTriples}
\mu^3(\{(p_1,p_2,p_3)\in K^3\colon |T(p_1,p_2,p_3)|<r\})\leq Cr^\beta,
\end{equation}
where $|T(p_1,p_2,p_3)|$ is the area of the triangle spanned by $p_1,p_2,p_3$. If $K$ is not curved, we say it is \emph{flat}. In particular, if $\dim(K)=0$ then $K$ is flat. 
\end{defn}
For example, if $K\subset\RR^2$ is contained in a countable union of lines, then it is flat. If $K$ has positive dimension and is contained in a smooth plane curve that has nonvanishing curvature (or more generally, whose curvature never vanishes to infinite order), then $K$ is curved. If $K=K_1\times K_2$, where $K_1,K_2\subset\RR$ are Borel sets of positive dimension, then $K$ is curved. If $K'\subset K$ has positive dimension and $K'$ is curved, then $K$ is curved. 

\begin{thm}\label{unifLowerBoundPinnedDistances}
For each $0<\alpha<2$ there exists $c(\alpha)>0$ so that the following holds. Let $A\subset\RR^2$ be a Borel set, with $\dim(A)=\alpha.$ Then the set 
$\{p\in \RR^2\colon \dim\Delta_p(A) \leq \dim(A)/2 + c\}$ is flat. 
\end{thm}

\begin{cor}\label{pinnedDistancesOnGamma}
For each $0<\alpha<2$ there exists $c(\alpha)>0$ so that the following holds. Let $A\subset\RR^2$ be a Borel set, with $\dim(A)=\alpha.$ Let $\gamma\subset\RR^2$ be a smooth curve whose curvature never vanishes to infinite order. Then
\[
 \dim\{p\in \gamma \colon \dim\Delta_p(A) \leq \dim(A)/2 + c\}=0.
 \]
\end{cor}



Theorem \ref{unifLowerBoundPinnedDistances} is an example of a general phenomenon, which says that at a single scale, it is impossible for a set to have small image simultaneously under three nonlinear projections, provided these projections satisfy a certain non-degeneracy condition. This condition will be described below. 


\subsection{A single-scale estimate: the role of Blaschke curvature}
\begin{defn}
Let $U\subset\RR^2$ be a connected open set, and let $\phi_1,\phi_2,\phi_3\colon U\to\RR$ be smooth functions whose gradients are non-vanishing and pairwise linearly independent at each point of $U$. Blaschke~\cite{Bl} defined the curvature form

\begin{equation}\label{BlaschkeCurvatureForm}
2 \frac{\partial}{\partial \phi_1}\frac{\partial}{\partial \phi_2} \log \frac{\partial \phi_3 / \partial\phi_1}{\partial \phi_3 / \partial\phi_2}d\phi_1\wedge d\phi_2.
\end{equation}
\end{defn}
This curvature form is invariant under change of coordinates, and it captures important information about the 3-web given by the level-sets of the functions $\phi_1,\phi_2,\phi_3$ on $U$. Indeed, Blaschke~\cite[Section 9]{Bl} (see also \cite{Izo} for the exact formulation given above) showed that the curvature form \eqref{BlaschkeCurvatureForm} depends only on the 3-web (i.e. the triple of foliations defined by the level sets $\{\phi_i=c\}_{c\in\RR}$), and not on the choice of $\phi_1,\phi_2,\phi_3$.

If each of $\phi_1,\phi_2,\phi_3$ is linear, then the curvature form \eqref{BlaschkeCurvatureForm} vanishes identically. Conversely, if the curvature form \eqref{BlaschkeCurvatureForm} vanishes identically on $U$, then around each point $p\in U$ there is a local change of coordinates so that 
\begin{equation}\label{specialFormOfPhi}
\phi_1(x,y) = x;\quad \phi_2(x,y) = y;\quad \phi_3(x,y)=x+y.
\end{equation}
This observation is important in the context of projection theory for the following reason: If $\phi_1,\phi_2,\phi_3$ are of the form \eqref{specialFormOfPhi}, and if $E = A\times A$, where $A\subset\RR$ is an arithmetic progression, then $\#\phi_i(E)\leq 2(\#E)^{1/2}$ for each index $i=1,2,3$. 
Thus if the curvature form \eqref{BlaschkeCurvatureForm} vanishes identically, then for each integer $N\geq 1$ we can construct a set $E\subset U$ with $\#E=N$ so that $\#\phi_i(E)\leq 2N^{1/2}$ for $i=1,2,3$. Similarly, it is possible to construct discretized sets $E\subset U$ satisfying a  Frostman type non-concentration condition so that $\phi_1(E),\phi_2(E)$, and $\phi_3(E)$ each have size roughly $|E|^{1/2}$. In this section we will describe several converses to this statement, which are manifestations of the following principle: If the curvature form \eqref{BlaschkeCurvatureForm} does not vanish identically, then one of $\phi_1(E),\phi_2(E)$, or $\phi_3(E)$ must be substantially larger than $|E|^{1/2}$. 

Returning to the pinned Falconer distance problem, let $p_1,p_2,p_3$ be points in $\RR^2$ and let $\phi_i(q) = |p_i-q|^2$. As we will show in Section \ref{pinnedDistancesSection}, the curvature form \eqref{BlaschkeCurvatureForm} vanishes identically if and only if $p_1,p_2,p_3$ are collinear. If $p_1,p_2,p_3$ are not collinear, then the gradients of $\phi_i$ and $\phi_j$ are linearly dependent precisely on $\ell_{i,j}$, the line spanned by $p_i$ and $p_j$. We can now state the following single-scale version of Theorem  \ref{unifLowerBoundPinnedDistances}. 

\begin{thm}\label{pinnedDistanceCor}
For each $0<\alpha<2$, there exists $c = c(\alpha)>0$ such that the following holds for all $\delta>0$ sufficiently small. Let $p_1,p_2,p_3\subset [0,1]^2$ be three points that span a triangle with area at least $\delta^c$. Let $X\subset[0,1]^2$ be a union of $\delta$-squares with $|X|\geq\delta^{2-\alpha+c}$, and suppose that for all balls $B$ of diameter $r\geq\delta$, $X$ satisfies the Frostman-type non-concentration condition
\begin{equation}\label{nonConcentrationCondDist}
\begin{split}
|X \cap B| &\le  r^\alpha\delta^{2-\alpha-c}.
\end{split}
\end{equation}

Then for at least one index $i,$ we have
\begin{equation}\label{atLeastOneBigPinnedDistance}
|\Delta_{p_i}(X)|\geq \delta^{1-\alpha/2-c}.
\end{equation}
\end{thm}
\begin{rem}
Theorem \ref{pinnedDistanceCor} fails if we remove the condition that the three points span a triangle with large area. Indeed, both Example \ref{elekesExample} and the ``railroad tracks'' example from \cite[Figure 1]{KT} shows that if the three points $p_1,p_2,p_3$ are collinear (or almost collinear), then \eqref{atLeastOneBigPinnedDistance} can fail for all three points, or indeed, any number of points\footnote{more precisely, for any integer $N$, we can construct an example where an estimate of the form \eqref{atLeastOneBigPinnedDistance} fails for $N$ points.}. 
\end{rem}
\begin{rem}
Theorem \ref{pinnedDistanceCor} is a discretized version of \cite[Theorem 32]{ES} (see also \cite{ShaSoly}), which bounds the number of ``triple points'' determined by three families of circles. See \cite{El, RSS2} for further discussion of this problem.
\end{rem}

The next result is a version of Theorem \ref{pinnedDistanceCor} for arbitrary 3-webs.

\begin{thm}\label{curvatureProjections}
For each $0< \alpha< 2,$ there exists $c=c(\alpha)>0$ so that the following holds. Let $K\subset U\subset\RR^2$, where $K$ is compact and $U$ is open and connected. Let $\phi_1,\phi_2,\phi_3\colon U\to\RR$ be smooth functions whose gradients are pairwise linearly independent at each point of $U$, and suppose the Blaschke curvature form \eqref{BlaschkeCurvatureForm} is non-vanishing on $U$. Then the following is true for all $\delta>0$ sufficiently small.

Let $X\subset K$ be a union of $\delta$-squares, with $|X|\geq \delta^{2-\alpha+c}$, and suppose that for all balls $B$ of diameter $r\geq\delta$, $X$ satisfies the non-concentration condition
\begin{equation}\label{nonConcentrationCondX}
\begin{split}
|X \cap B| &\le  r^\alpha\delta^{2-\alpha-c}.
\end{split}
\end{equation}
Then for at least one index $i$ we have
\begin{equation}\label{atLeastOneBigProjection}
|\phi_i(X)|\geq \delta^{1-\alpha/2-c}.
\end{equation}
\end{thm}

Theorem \ref{curvatureProjections} is a discretized projection theorem in the spirit of \cite{B10, Shm}. In \cite{B10}, Bourgain proved that if $X\subset[0,1]^2$ is a set satisfying \eqref{nonConcentrationCondX}, and if $\Theta\subset S^1$ is a large (and non-concentrated) set of directions, then there must exist a direction $e\in\Theta$ so that the associated orthogonal projection $\pi_{e}(q)$ satisfies \eqref{atLeastOneBigProjection}. The set $\Theta$ must be large, since if we define $X=[0,1]^2 \cap (\delta^{\alpha}\ZZ)^2$, then whenever $s$ and $t$ are integers with $|s|$ and $|t|$ small, the projection $(x,y)\cdot (s,t)$ will have small cardinality and thus fail to satisfy \eqref{atLeastOneBigProjection}. Thus it is new, and rather surprising, that Theorem \ref{curvatureProjections} only requires three projections.

\subsection{A single-scale estimate when Blaschke curvature vanishes: the role of convexity}
Let us now consider a quadruple 
%
of functions $\phi_1,\ldots,\phi_4$ where each triple has vanishing Blaschke curvature. To simplify our discussion, we will work in coordinates where $\phi_1(x,y)=x,$ $\phi_2(x,y) = y$, $\phi_3(x,y)=x+y$, and $\phi_4(x,y) = u(x)+v(y)$. In this setting, we can still expect expansion provided that at least one of the functions $u$ or $v$ is strictly convex.

\begin{thm}\label{convexProjections}
For each $0< \alpha< 2,$ there exists $c=c(\alpha)>0$ so that the following holds. Let $u,v\colon [0,1]\to\RR$ be smooth. Suppose $u', v'$, and $u''$ are nonzero on $[0,1]$. Let $\phi_1(x,y)=x,$ $\phi_2(x,y)=y,$  $\phi_3(x,y)=x+y$, and $\phi_4(x,y)=u(x)+v(y)$. Then the following is true for all $\delta>0$ sufficiently small.

Let $X\subset [0,1]^2$ be a union of $\delta$-squares, with $|X|\geq \delta^{2-\alpha+c}$, and suppose that for all balls $B$ of diameter $r\geq\delta$, $X$ satisfies the non-concentration condition
\begin{equation}\label{nonConcentrationCondXConvex}
\begin{split}
|X \cap B| &\le  r^\alpha\delta^{2-\alpha-c}.
\end{split}
\end{equation}
Then for at least one index $i$ we have
\begin{equation}\label{atLeastOneBigProjectionConvex}
|\phi_i(X)|\geq \delta^{1-\alpha/2-c}.
\end{equation}
\end{thm}

Theorem~\ref{convexProjections} can be viewed as a continuous analogue of \cite[Theorem 1.2]{JRT}, and indeed our proof is inspired by certain ideas from \cite{JRT}. 

Finally, we remark that Theorem~\ref{convexProjections} allows us to prove similar expansion results for $d$-tuples of functions $\phi_1,\ldots,\phi_d$ for $d>4$. In brief, suppose that $\phi_1,\ldots,\phi_d$ are smooth functions, and that the Blaschke curvature of each triple $\phi_i,\phi_j,\phi_k$ is identically 0. Then after a change of coordinates we can write $\phi_1(x,y)=x,$ $\phi_2(x,y)=y,$  $\phi_3(x,y)=x+y$, and $\phi_i(x,y)=u_i(x)+v_i(y)$ for each $i\geq 4$. We can then apply Theorem~\ref{convexProjections} for each $i\geq 4$. If $u_i''$ and $v_i''$ are identically 0 for each index $i\geq 4$, then the $d$-web $(\phi_1,\ldots,\phi_d)$ is parallelizable. Thus Theorem~\ref{convexProjections} is a quantitative version of the principle that if $d$ functions $\phi_1,\ldots,\phi_d$ determine a non-parallelizable $d$-web and if $X$ is a $(\delta,\alpha)_2$ set, then $|\phi_i(X)|$ must be substantially larger than $|X|^{1/2}$ for at least one index $i$.

\subsection{The discretized Elekes-R\'onyai theorem}
Theorems \ref{unifLowerBoundPinnedDistances} and \ref{curvatureProjections} are closely related to the Elekes-R\'onyai theorem, which we will now briefly describe. In \cite{ErdSze}, Erd\H{o}s and Szemer\'edi proved that if $A\subset\RR$ is a finite set, then either the sum set $A+A=\{a+a^\prime\colon a,a^\prime\in A\}$ or the product set $A.A=\{aa^\prime\colon a,a^\prime\in A\}$ must have cardinality much larger than that of $A$. 
\begin{thm}[Erd\H{o}s-Szemer\'edi]\label{ESThm}
There exists $c>0$ so that for all finite sets $A\subset\RR$, we have
\begin{equation}\label{sumProdIneq}
\#(A+A) + \#(A.A) \gtrsim (\#A)^{1+c}.
\end{equation}
\end{thm}

The Erd\H{o}s-Szemer\'edi Theorem quantifies the principle that a subset of $\RR$ cannot be approximately closed under both addition and multiplication. In \cite{ER}, Elekes and R\'onyai developed this idea in a slightly different direction, and proved the following:
\begin{thm}[Elekes-R\'onyai]\label{ElekesRonyai}
Let $P$ be a bivariate real polynomial. Then either $P$ is one of the special forms $P(x,y) = h(a(x)+b(y))$ or $P(x,y) = h(a(x)b(y))$, where $h,a,b$ are univariate real polynomials, or 
that for all finite sets of real numbers $A,B$ of cardinality $N$, we have 
\begin{equation}\label{growthPoly}
\#P(A,B) =\omega(N).
\end{equation}
\end{thm}
If $P(x,y) = h(a(x)+b(y))$ or $P(x,y) = h(a(x)b(y))$, where $h,a,b$ are univariate real polynomials, then we call $P$ a (polynomial) special form. If $P$ is not a (polynomial) special form then we call it an expanding polynomial. Theorem \ref{ElekesRonyai} has since been generalized in several directions. See \cite{ES, RSS, RSZ, W, MRSW, RS, BB}, and the references therein. 

Similar questions can be asked for metric entropy. In this direction, Katz and Tao proposed the discretized ring conjecture \cite{KT}, which was solved by Bourgain \cite{B03}. The discretized ring conjecture (now a theorem) is similar to the Erd\H{o}s-Szemer\'edi theorem, except cardinality has been replaced by metric entropy. It was an important ingredient in the proof of Theorem \ref{bourgainThm}
\begin{thm}[Bourgain]\label{discRingThm}
For each $0<\alpha<1$, there is a number $c = c(\alpha)>0$ and $s = s(\alpha)>0$ so that the following holds for all $\delta>0$ sufficiently small. Let $A\subset [1,2]$ be a union of $\delta$-intervals, with $|A| = \delta^{1-\alpha}$. Suppose that $A$ satisfies the following non-concentration condition for each interval $J$
\begin{equation}\label{nonConcentrationDiscRingConj}
|A \cap J| \leq |J|^{\alpha}\delta^{1-\alpha-s}.
\end{equation}
Then
\begin{equation*}
|A+A| + |A.A| \geq \delta^{1-\alpha-c}.
\end{equation*}
 
\end{thm}
The non-concentration condition \eqref{nonConcentrationDiscRingConj} arises naturally when discretizing fractal sets; see \cite{KT} or Section \ref{dimExpansionSec} for details. In particular, in \cite{B10} Bourgain used a variant of Theorem~\ref{discRingThm} to prove the following ``dimension expansion'' result for subsets of $\RR$:
\begin{thm}[Bourgain]\label{growthInRings}
For each $0<\alpha<1$ there is a number $c=c(\alpha)>0$ so that the following holds. Let $A\subset\RR$ be a Borel set with $\dim A = \alpha$. Then there exists $\lambda \in A$ so that
\begin{equation*}
\dim (A+\lambda A)\geq\alpha+c.
\end{equation*}
\end{thm}
Theorem~\ref{growthInRings} implies the Erd\H{o}s-Volkmann ring conjecture \cite{EV} (proved by Edgar and Miller \cite{EM} slightly earlier, and also addressed directly by Bourgain in \cite{B03}), which asserts that there does not exist a proper Borel subring of the reals with positive Hausdorff dimension. 

Theorem \ref{discRingThm} has seen a number of extensions and generalizations. In \cite{BG}, Bourgain and Gamburd proved a variant of \ref{discRingThm} with a less restrictive version of the non-concentration condition \eqref{nonConcentrationDiscRingConj}, while in \cite{GKZ} Guth, Katz, and the second author found a simple new proof of Theorem \ref{discRingThm} that yields an explicit lower bound on the exponent $\eps$. Theorem \ref{discRingThm} has also been generalized to $SU(2)$ \cite{BG} as well as other settings \cite{BG2, H, HS, S}.

We prove a discretized version of the Elekes-R\'onyai theorem for analytic functions, in the spirit of Theorem \ref{discRingThm}. Before stating our result, we will define what it means for an analytic function to be a special form (cf. \cite[Lemma 10]{RaSh}).
\begin{defn}\label{defnAnalyticSpecialForm}
Let $P(x,y)$ be analytic on a connected open set $U\subset\RR^2$. We say $P$ is an (analytic) special form if on each connected region of $U \backslash \big(\big\{ \frac{d}{dx}P = 0\big\}\cup \big\{ \frac{d}{dy} = 0\big\}\big)$, there are univariate real analytic functions $h$, $a$, and $b$ so that $P(x,y) = h(a(x)+b(y))$.
\end{defn}
Note that every polynomial special form is also an analytic special form. 
We are now ready to state our main result on expanding functions. 

\begin{thm}\label{main-entropy-growth}
Let $0< \kappa\leq \alpha< 1$, let $U\subset\RR^2$ be a connected open set that contains $[0,1]^2$, and let $P\colon U \to\RR$ be analytic (resp.~polynomial). Then either $P$ is an analytic (resp.~polynomial) special form, or there exists $\eps = \eps(\alpha,\kappa)>0$ and $\eta = \eta(\alpha,\kappa,P)>0$ so that the following is true for all $\delta>0$ sufficiently small.

Let $A,B\subset[0,1]$ be unions of $\delta$ intervals, and suppose that for all intervals $J$, $A$ and $B$ satisfy the non-concentration conditions
\begin{equation}\label{nonConcentrationCond2}
\begin{split}
|A \cap J| & \le  |J|^\kappa\delta^{1-\alpha-\eta},\\
|B \cap J| & \le  |J|^\kappa\delta^{1-\alpha-\eta}.
\end{split}
\end{equation}
Let $E\subset A\times B$ be a union of $\delta$-squares, with $|E|\geq \delta^{2-2\alpha+\eta}$. Then we have the growth estimate
\begin{equation}\label{entropyGrowth}
|P(E)|\geq \delta^{1-\alpha-\eps}.
\end{equation}
%
\end{thm}

\begin{rem}
Note that the growth exponent $\eps$ in Theorem \ref{main-entropy-growth} is independent of $P$, and in particular independent of the degree of $P$ when $P$ is a polynomial. This is interesting for the following reasons. First, if $P$ is a polynomial then it might be possible (though likely difficult) to prove an analogue of Theorem \ref{main-entropy-growth} using Bourgain's discretized sum-product theorem (Theorem \ref{discRingThm}) and tools from Ruzsa calculus. For example, these ideas were used in \cite{BKT} to prove expansion over $\mathbb{F}_p$ for a certain explicitly specified polynomial.  However, the entropy growth exponent $\eps$ arising from such a strategy would necessarily depend on the polynomial $P$ (the degree of $P$, the number of monomial terms, etc).

Second, there is a general intuition in additive combinatorics that states that one should expect a quantity such as $P(A\times B)$ to be large when there are few solutions to $P(x,y) = P(x^\prime,y^\prime)$, and conversely. Concretely, if $A$ and $B$ have cardinality $\delta^{-\alpha}$, then the statements ``there are at about $\delta^{-3\alpha+\eps}$ $\delta$-separated solutions to $P(x,y) = P(x^\prime,y^\prime)$'' and ``$P(A,B)$ has measure about $\delta^{1-\alpha-\eps}$'' are often morally equivalent. In the present setting, however, this intuition is not entirely correct---there exists a sequence of polynomials $P_1,P_2,\ldots$ of increasing degree so that there are $\delta^{-3\alpha+o(1)}$ solutions to $P_j(x,y) = P_j(x^\prime,y^\prime)$, and yet the measure of $P_j(A \times B)$ remains small.
\end{rem}

We will also prove the following ``dimension expansion'' version of Theorem \ref{main-entropy-growth}, which is an analogue of Theorem \ref{growthInRings}.  
\begin{thm}\label{dimension-expander}
For every $0< \alpha< 1$, there exists $c=c(\alpha)>0$ so that the following holds. Let $U\subset\RR^2$  be a connected open set that contains $[0,1]^2$ and let $P\colon U \to\RR$ be analytic (resp.~polynomial). Then either $P$ is an analytic (resp.~polynomial) special form, or for every pair of Borel sets $A,B\subset [0,1]$ of dimension at least $\alpha$, we have 
\[
\dim P(A \times B) \ge \alpha+c.
\] 
\end{thm}

\subsection{Proof ideas}
If $P(x,y)$ is a smooth function, we will define an auxiliary function $K_P(x,y)$ that measures the extent to which $P$ ``looks like'' a special form (in the sense of Definition \ref{defnAnalyticSpecialForm}) near the point $(x,y)$. $K_P(x,y)$ is closely related to the Blaschke curvature of the 3-web defined by the functions $\phi_1(x,y)=x,$ $\phi_2(x,y)=y$, and $\phi_3(x,y)=P(x,y)$ at the point $(x,y)$. We will show that if $P$ is analytic and $K_P$ vanishes identically, then $P$ is an analytic special form. 

At the other extreme, if $K_P$ is bounded away from 0, then we will prove that the set of solutions to $P(a,b)=P(a',b')$, where $a,a'\in A$ and $b,b'\in B$ has small $\delta$-covering number. We will call this an ``energy dispersion'' estimate, since it is an upper bound on the $L^2$ norm of the multiplicity function $m(z)=\mathcal{E}_{\delta}(P^{-1}(z)\cap (A\times B))$, where $\mathcal{E}_{\delta}(\cdot)$ denotes the $\delta$-covering number. This energy dispersion estimate will be stated precisely in Proposition \ref{countingNumberQuadruplesNablePWedgeP}, and it is one of the main technical results of this paper. It is  closely related to the $L^2$ flattening lemma that was used by Bourgain and Gamburd \cite{BG2, BG} to establish expansion in the Cayley graph of $SL_2(\mathbb{F}_p)$ and the existence of a spectral gap in certain free subgroups of $SU(2)$.

Theorem \ref{curvatureProjections} follows from Proposition \ref{countingNumberQuadruplesNablePWedgeP} and a formula that relates the auxiliary function $K_P$ to Blaschke curvature. Theorem \ref{pinnedDistanceCor} and Theorem \ref{unifLowerBoundPinnedDistances} are proved by combining Proposition \ref{countingNumberQuadruplesNablePWedgeP} with tools from real algebraic geometry, while Theorem \ref{main-entropy-growth} is proved by combining Proposition \ref{countingNumberQuadruplesNablePWedgeP} with tools from semi-analytic geometry.

Proposition \ref{countingNumberQuadruplesNablePWedgeP} is proved using ideas from additive combinatorics. We will give a brief sketch of the proof, which is inspired by the proof of the Elekes-R\'onyai theorem in \cite{RSZ}. The goal is to prove an estimate of the form $\mathcal{E}_{\delta}(\{P(a,b)=P(a',b')\})\leq\delta^{-3\alpha+\eps}$ for some $\eps=\eps(\alpha)>0$. Here and for the rest of this proof sketch, we will always suppose that $a,a'\in A$ and $b,b'\in B$. Suppose instead that $\mathcal{E}_{\delta}(\{P(a,b)=P(a',b')\})\geq \delta^{-3\alpha+\eps}$; if $\eps>0$ is sufficiently small we will obtain a contradiction.

For the purpose of this sketch, we will suppose that for each $x,x'$ and $y$, there is a unique $g(x,x',y)$ so that $P(x,y)=P(x',g(x,x',y)).$ In particular, if $P(a,b)=P(a',b'),$ then $g(a,a',b)\in B$. This means there are many triples $(a,a',b)\in A\times A\times B$  for which $g(a,a',b)\in B$. Concretely, we have an estimate of the form
\begin{equation}\label{manyTriplesInB}
|\{(a,a',b)\colon g(a,a',b)\in B\}|\geq\delta^{3-3\alpha+O(\eps)}.
\end{equation}
Next, we will write $g(a,a',b)=g_{b}(a,a')$; we will think of $\{g_b\}_{b\in B}$ as a family of (nonlinear) projections from $\RR^2\to\RR$. The bound \eqref{manyTriplesInB} implies that for many $b\in B$, there is a set $E_b\subset A\times A$ with $|E_b|\geq\delta^{O(\eps)}|A\times A|$ so that $g_b(E_b)\subset B$, and in particular, $|g_b(E_b)|\leq\delta^{-O(\eps)}|E_b|^{1/2}$. This is precisely the setting where Shmerkin's projection theorem (discussed in Section \ref{exceptionalDirectionsIntroSec}) can be applied. Shmerkin's projection theorem says that if the family of projection maps $\{g_b\}$ satisfies a certain non-degeneracy condition (which is indeed satisfied if the auxiliary function $K_P$ is bounded away from 0), then it is impossible for $|g_b(E_b)|$ to be small for many different projection maps $g_b$. This is a contradiction, which completes the proof. 

\subsection{Structure of the paper}\label{StructureOfPaperSec}
In Section \ref{ShmerkinSection} we will introduce Shmerkin's nonlinear discretized projection theorem, and we will recast this theorem into a (slightly technical) statement about energy dispersion; this will be Lemma \ref{EnergyDispersionForF}. In Section \ref{auxiliaryFunctionSection} we will introduce the auxiliary function $K_P,$ which helps quantify whether $P$ (locally) looks like a special form, and we show that it has two key properties. First, if $K_P$ vanishes identically, then $P$ is a special form. Second, when $K_P$ is large, then there can be few solutions to $P(x,y)=P(x',y')$; this will be Proposition \ref{countingNumberQuadruplesNablePWedgeP}. Proposition \ref{countingNumberQuadruplesNablePWedgeP} is used to prove most of the results in the paper. The structure of the paper is as follows.

\medskip

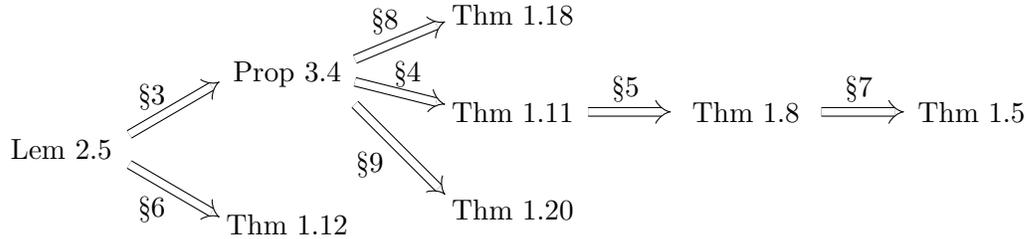
\begin{tikzpicture}
  \node (A) at (0.7,2) {Lem \ref{EnergyDispersionForF}};
  \node (B) at (3.7,3) {Prop \ref{countingNumberQuadruplesNablePWedgeP}};
  \node (C) at (3.7,1) {Thm \ref{convexProjections}};
  \node (D) at (6.7,3.8) {Thm \ref{main-entropy-growth}};
  \node (E) at (6.7,2.5) {Thm \ref{curvatureProjections}};
  \node (F) at (9.8, 2.5) {Thm \ref{pinnedDistanceCor}};
  \node (G) at (12.8,2.5) {Thm \ref{unifLowerBoundPinnedDistances}};
  \node (H) at (6.7,1.2) {Thm \ref{dimension-expander}};
   
   \draw[-implies, double distance=0.1cm]
   (1.6,2.2) -- (2.8,2.9);

   \draw[-implies, double distance=0.1cm]
   (1.6,1.8) -- (2.8,1.1);

   \draw[-implies, double distance=0.1cm]
   (4.6, 3.2) -- (5.8,3.7);

   \draw[-implies, double distance=0.1cm]
   (4.6, 2.9) -- (5.8,2.6);

   \draw[-implies, double distance=0.1cm]
   (4.6, 2.6) -- (5.8,1.4);


   \draw[-implies, double distance=0.1cm]
   (7.7, 2.5) -- (8.8,2.5);


   \draw[-implies, double distance=0.1cm]
   (10.8, 2.5) -- (11.9,2.5);

   \node (K) at (1.9,2.7) {\S \ref{auxiliaryFunctionSection}};

   \node (L) at (1.9,1.2) {\S \ref{convexitySec}};

   \node (M) at (5,3.7) {\S \ref{mainThmProofsSectionOne}};

   \node (N) at (5.3,3.0) {\S \ref{BlaschkeSection}};

   \node (O) at (4.8,1.8) {\S \ref{dimExpansionSec}};


   \node (Q) at (8.2,2.8) {\S \ref{pinnedDistancesSection}};

   \node (R) at (11.3,2.8) {\S \ref{pinnedFalconerAndVisibilitySection}};

\end{tikzpicture}

\subsection{Notation}
In what follows, $\delta$ will denote a small positive number. We will be interested in the asymptotic behaviour of various quantities as $\delta\searrow 0$. If $f$ and $g$ are functions, we write $f\lesssim g$ (or $f = O(g)$ or $g = \Omega(f)$) if there is a constant $C$ (independent of $\delta$) so that $f\leq Cg$. Such a constant will be called an ``implicit constant in the $\lesssim$ notation.'' If $f\lesssim g$ and $g\lesssim f$, we write $f\sim g$. While our main theorems involve functions in the plane, some of our intermediate results will be slightly more general and will involve functions with domain $\RR^d$. All implicit constants in the $\lesssim$ notation are allowed to depend on $d$. We will also write $f=O_\alpha(g)$ if there is a constant $C$ (which may depend on $\alpha$) so that $f\leq Cg$. 

If $X\subset\RR^d$, we will use $\mathcal{E}_{\delta}(X)$ to denote the $\delta$-covering number of $X$; $\#X$ to denote the cardinality of $X$; and $|X|$ to denote the Lebesgue measure of $X$. To improve clarity, we sometimes use $\vol_d(X)$ in place of $|X|$ to emphasize the ambient dimension. If $t>0$ we use $N_t(X)$ to denote the $t$-neighborhood of $X$. 

Our arguments will frequently involve sets of the form $I_1\times I_2\times\ldots\times I_d,$ where $I_1,\ldots,I_d$ are closed intervals. We will refer to such sets as rectangles.

\subsection{Thanks}
The authors would like to thank Pablo Shmerkin for numerous helpful comments and suggestions. The authors would like to thank Michael Christ for stimulating discussions and for alerting them to Blaschke curvature, and its connection to polynomial and analytic special forms \cite{Chr}. The authors would like to thank the anonymous referee for his or her suggestions and corrections.

\section{Shmerkin's nonlinear discretized projection theorem and its consequences}\label{ShmerkinSection}
In this section we will discuss Shmerkin's nonlinear generalization \cite{Shm} of Bourgain's discretized projection theorem \cite{B10}. Shmerkin's nonlinear projection theorem concerns smooth functions of the form $G\colon I\to\RR$, where $I\subset[0,1]^3$ is a rectangle. We will call these ``projection functions.''

In what follows, it will be helpful to introduce some additional notation. If $G\colon I\to \RR$ is a projection function, define $G_{(z)}(x,y) = G(x,y,z)$. In particular, 
\begin{equation*}
\nabla G_{(z)}(x,y)=\big( \partial_x  G(x,y,z),\ \partial_y  G(x,y,z)\big),
\end{equation*} 
so $\nabla G_{(z)}(x,y)$ is a vector in $\RR^2$. For such a function $G$, and for $(x,y,z)\in I$, define the map
\begin{equation}\label{defnTheta}
\theta_{(x,y)}(z) = \angle\operatorname{dir}(\nabla G_{(z)}(x,y)).
\end{equation}

With these definitions, we can now state Shmerkin's result from \cite{Shm}.
\begin{thm}[Shmerkin]\label{Shm}
For every $\eta>0$ and $C>0$, there exists $\tau=\tau(\eta)>0$ and $\delta_0=\delta_0(\eta,C)>0$ such that the following holds for all $0<\delta<\delta_0$. Let $I=I_1\times I_2\times I_3 \subset[0,1]^3$ be a rectangle, let $X\subset I_1\times I_2$ be a union of $\delta$-squares, and let $Z\subset I_3$ be a union of $\delta$-intervals.

Let $G\colon I\to\RR$ be a projection function that satisfies 
\begin{equation}\label{boundsOnG}
\sup_{z\in Z}\Vert G_{(z)}\Vert_{C^2(I_1\times I_2)}<C,\quad\quad\inf_{(x,y,z)\in X\times Z}|\nabla G_{(z)}(x,y)|>C^{-1}.
\end{equation}

Suppose that $X$ satisfies the non-concentration estimate
\begin{equation}\label{Xnonconcentration}
|X \cap Q|\leq \delta^{\eta}|X|,
\end{equation}
whenever $Q\subset[0,1]^2$ is a square of side-length $|X|^{1/2}$. Suppose furthermore that for each $(x,y)\in X$ and each arc $J\subset S^1$ of length $|J|\geq\delta$, we have the ``transversality'' estimate
\begin{equation}\label{thetaimagenoncon}
|\{z\in Z\colon \theta_{(x,y)}(z)\in J\}| \leq \delta^{-\tau}|J|^{\eta}|Z|.
\end{equation}

Then there is a small bad set $Z_{\operatorname{bad}}\subset Z$ with $|Z_{\operatorname{bad}}|\leq \delta^{\tau}|Z|$ so that for all $z\in Z\backslash Z_{\operatorname{bad}}$, the following holds: Let $X^\prime\subset X$ with $|X^\prime|\geq \delta^{\tau}|X|$. Then
\begin{equation}
\mathcal{E}_{\delta}(\{G(x,y,z)\colon (x,y)\in X^\prime\})\geq\delta^{-\tau}\mathcal{E}_{\delta}(X)^{1/2}.
\end{equation}
\end{thm}

Our task for the remainder of this section is to reformulate Theorem \ref{Shm} as an energy dispersion estimate.

\subsection{Weakening the hypotheses of Theorem~\ref{Shm}}
Our first task is to state a version of Theorem~\ref{Shm} where the bounds \eqref{boundsOnG} are slightly less restrictive. We will also recast several of the hypotheses and conclusions to align more closely with the setup of Theorem \ref{main-entropy-growth}. 

\begin{lem}\label{Shmcor}
For every $0<\alpha<1$ and $0<\kappa\leq\alpha$, there exists $\sigma=\sigma(\alpha,\kappa)>0$ and $\delta_0=\delta_0(\alpha,\kappa)$ such that the following holds for all $0<\delta<\delta_0$.

 Let $I=I_1\times I_2\times I_3\subset[0,1]^3$ and let $G\colon I \to\RR$ be a projection function that satisfies the non-degeneracy conditions
\begin{align}
\Vert G\Vert_{C^2(I)}&\leq\delta^{-\sigma},\label{boundOnC2NormG}\\
\inf_{I}|\nabla G_{(z)}(x,y)|& \geq \delta^\sigma,\label{boundOnNablaG} \\
\inf_{I} |\partial_z \theta_{(x,y)}(z)|&\geq \delta^{\sigma}.  \label{lowerBoundThetap}
\end{align}

Let $A_1,A_2,A_3$ be sets, with $A_i\subset I_i,$ and let $A=A_1\times A_2\times A_3$. Suppose that for each index $i$ and each interval $J$ of length at least $\delta$, we have the non-concentration condition
\begin{equation}\label{nonconcorA1234}
\mathcal{E}_{\delta}(A_i\cap J)\le \delta^{-\sigma}|J|^\kappa\delta^{-\alpha}.
\end{equation}
Let $R\subset\RR$ be a set with $\mathcal{E}_{\delta}(R)\leq \delta^{-\alpha-\sigma}$.
Then 
\begin{equation}\label{volOfTriples}
\mathcal{E}_{\delta}\big( \{(a_1,a_2,a_3)\in A \colon G(a_1,a_2,a_3)\in R\} \big) \leq \delta^{-3\alpha+\sigma}.
\end{equation}
\end{lem}
\begin{proof}
Our proof consists of three main steps. While there are a few technical details, each of these steps follow standard arguments. First, note that our projection function $G$ satisfies the conditions \eqref{boundOnC2NormG} and \eqref{boundOnNablaG}. These are different from the requirement \eqref{boundsOnG} that is needed when applying Theorem \ref{Shm}.  We will apply a re-scaling argument to create a new function $\tilde G$ that satisfies \eqref{boundsOnG}. 

Second, our projection function $G$ satisfies the condition \eqref{lowerBoundThetap}, and our sets $A_1,A_2,A_3$ satisfy the non-concentration condition \eqref{nonconcorA1234}. These are different from the corresponding requirements \eqref{thetaimagenoncon} and \eqref{Xnonconcentration}. We will verify that our new function $\tilde G$ and appropriately constructed sets $X$ and $Z$ satisfy \eqref{Xnonconcentration} and \eqref{thetaimagenoncon}.

Third, we will show that our desired bound \eqref{volOfTriples} follows by applying Theorem \ref{Shm} to our newly constructed function $\tilde G$. In brief, if the bound \eqref{volOfTriples} failed then there must exist a large subset $A_3^\prime\subset A_3$ so that $G(A_1,A_2,z)$ has small $\delta$-covering number for each $z\in A_3^\prime$. We will show that this contradicts Theorem \ref{Shm}. 

\medskip

\noindent \textbf {Constructing $\tilde G$: A re-scaling argument}.
First we will find a slightly smaller sub-rectangle of $I$ that still captures most of the triples from \eqref{volOfTriples}. Let $I'=I_1'\times I_2'\times I_3'\subset I$ be a rectangle with $|I_1'|,|I_2'|\leq \delta^{2\sigma}$ and $|I_3'|\leq\delta^{4\sigma}$  that maximizes the quantity
\begin{equation*}
\mathcal{E}_{\delta}\big( \{(a_1,a_2,a_3)\in I'\cap A \colon G(a_1,a_2,a_3)\in R\} \big).
\end{equation*}
Define $A_i^\prime = A_i\cap I_i'$. and define $A'=A_1'\times A_2'\times A_3'$.
Since $I\subset [0,1]^3$, by pigeonholing, we have
\begin{equation}\label{mostEntropyCapturedSubSquare}
\begin{split}
\mathcal{E}_{\delta}&\big( \{(a_1,a_2,a_3)\in A \colon G(a_1,a_2,a_3)\in R\} \big) \\
&\leq \delta^{-8\sigma} \mathcal{E}_{\delta}\big( \{(a_1,a_2,a_3)\in A^\prime\colon G(a_1,a_2,a_3)\in R\} \big). 
\end{split}
\end{equation}

Let
\begin{equation*}
m = \inf_{I^\prime}|\nabla G_{(z)}(x,y)|.
\end{equation*}
By \eqref{boundOnNablaG} we have $m\geq\delta^{\sigma}$. By \eqref{boundOnC2NormG} we have
\begin{equation}\label{nablaGRoughlyConst}
\sup_{I^\prime}|\nabla G_{(z)}(x,y)|\leq m + \delta^{-\sigma}\operatorname{diam}(I')\leq 4m. 
\end{equation}

Let $x_0$ be the left endpoint of $I_1'$, let $y_0$ be the left endpoint of $I_2'$, and let $z_0$ be the left endpoint of $I_3'$. Define $\tilde I_1 = \delta^{-2\sigma}(I_1'-x_0)$, $\tilde I_2 = \delta^{-2\sigma}(I_2'-y_0)$, and $\tilde I_3 = I_3' - z_0$. Thus $\tilde I\subset[0,1]^3$ is a rectangle with bottom corner $(0,0,0)$. For $(x,y,z)\in \tilde I$, define 
\begin{equation*}
\tilde G(x,y,z) = m^{-1}\delta^{-2\sigma}( G(\delta^{2\sigma} x+x_0, \delta^{2\sigma} y+y_0, z+z_0)-G(x_0,y_0,z_0)).
\end{equation*}
We have $\tilde G\colon \tilde I\to\RR$, with $\tilde G(0,0,0)=0$. Indeed, if $(x,y,z)\in\tilde I$, then $(\delta^{2\sigma} x+x_0\in I_1'; \delta^{2\sigma} y+y_0\in I_2';$ and $z+z_0\in I_3'$, and thus $G(\delta^{2\sigma} x+x_0, \delta^{2\sigma} y+y_0, z+z_0)-G(x_0,y_0,z_0)$ is well-defined. 

We have 
\begin{equation}\label{rescaledGTilde}
\begin{split}
\partial_x \tilde G(x,y,z) & = m^{-1} G_1(\delta^{2\sigma} x + x_0, \delta^{2\sigma} y + y_0, z+z_0),\\
\partial_y \tilde G(x,y,z) & = m^{-1} G_2(\delta^{2\sigma} x + x_0, \delta^{2\sigma} y + y_0, z+z_0),
\end{split}
\end{equation}
where for clarity we write $G_1(\cdot,\cdot,\cdot)$ (resp.~$G_2(\cdot,\cdot,\cdot)$ ) to denote the partial derivative of $G$ with respect to its first (resp.~second) variable; this notation will not be used anywhere else in the paper. \eqref{rescaledGTilde} and the definition of $m$ implies that 
\begin{equation}\label{infGradientTildeG}
\inf_{\tilde I}|\nabla \tilde G_{(z)}(x,y)|\geq 1.
\end{equation}
On the other hand, by \eqref{nablaGRoughlyConst} we have
\begin{equation*}
\sup_{\tilde I}|\nabla \tilde G_{(z)}(x,y)|\leq 6.
\end{equation*}
This implies 
$
|\tilde G(x, y, z)-\tilde G(0,0,z)| \le  10
$ for all $(x, y, z) \in \tilde I$.
By \eqref{boundOnC2NormG}, we have
$|\tilde G(0,0,z)|
\le \delta^{-\sigma}m^{-1}|z| \le \delta^{-2\sigma}|I_3'| \le 1,$
for all $z\in \tilde I_3$. Thus
\begin{equation}
|\tilde G(x, y, z)| \le  11\quad\textrm{for all }\ (x, y, z) \in \tilde I.
\end{equation}

Next, we can compute
\begin{equation*}
|\partial_{xx}\tilde G(x,y,z)| = |m^{-1}\delta^{2\sigma} (\partial_{xx} G)(\delta^{2\sigma} x + x_0, \delta^{2\sigma} y + y_0, z+z_0)|\leq 1,
\end{equation*}
where the final inequality used the bound $m\geq\delta^\sigma$ and \eqref{boundOnC2NormG}. A similar computation shows that the other second derivatives of $\tilde G$ with respect to $x$ and $y$ are bounded by 1, and thus
\begin{equation}\label{C2NormTildeG}
\sup_{z\in \tilde I_3 }\Vert \tilde G_{(z)}\Vert_{C^2(\tilde Q)}\leq 100.
\end{equation}
Finally, if $(x,y,z)\in\tilde I$, then
\begin{equation*}
\tilde\theta_{(x,y)}(z)=\theta_{(\delta^{2\sigma} x + x_0, \delta^{2\sigma} y + y_0)}(z+z_0),
\end{equation*}
so for each $(x,y,z)\in\tilde  I$, \eqref{lowerBoundThetap} implies
\begin{equation}\label{sizeThetaTildeG}
|\partial_z \tilde \theta_{(x,y)}(z)|\geq \delta^{\sigma}.
\end{equation}

Let $\tilde A_1 = N_{\delta}(\delta^{-2\sigma}(A_1^\prime-x_0))$, $\tilde A_2 = N_{\delta}(\delta^{-2\sigma}(A_2^\prime-y_0))$, and $\tilde A_3 = N_{\delta}(A_3^\prime-z_0)$, and define $\tilde A = \tilde A_1\times \tilde A_2\times \tilde A_3$.  We have that 
%
The sets $\tilde A_1$ and $\tilde A_2$ satisfy a slightly weaker version of \eqref{nonconcorA1234}---if $J\subset[0,1]$ is an interval, then
\begin{equation}\label{nonConcentrationTildeAi}
\begin{split}
\mathcal{E}_{\delta}(\tilde A_i\cap J) &\leq \delta^{-2\sigma}\mathcal{E}_{\delta}(A_i \cap \delta^{\rho}(J+x_0))\\
&\leq |J|^{\kappa}\delta^{-\alpha-3\sigma}.
\end{split}
\end{equation}
The set $\tilde A_3$ is just the $\delta$-thickening of a translate of $A_3^\prime$, so \eqref{nonConcentrationTildeAi} follows immediately from \eqref{nonconcorA1234}. 

Finally, let $\tilde R = m^{-1}\delta^{-2\sigma}(R - G(x_0,y_0,z_0))$ and note that $\mathcal{E}_{\delta}(\tilde R)\le \delta^{-\alpha-4\sigma}$.  Then 
\begin{equation}\label{entropyCapturedTildeAi}
\mathcal{E}_{\delta}\big( \{(a_1,a_2,a_3)\in A'\colon G(a_1,a_2,a_3)\in R\} \big)
\leq\delta^{4\sigma} \mathcal{E}_{\delta}\big( \{(a_1,a_2,a_3)\in \tilde A \colon \tilde G(a_1,a_2,a_3)\in \tilde R\} \big).
\end{equation}
Combining \eqref{mostEntropyCapturedSubSquare} and \eqref{entropyCapturedTildeAi}, we see that in order to establish \eqref{volOfTriples}, it suffices to prove the bound
\begin{equation}\label{modifiedVolOfTriples}
\mathcal{E}_{\delta}\big( \{(a_1,a_2,a_3)\in \tilde A \colon \tilde G(a_1,a_2,a_3)\in \tilde R\} \big)\lesssim \delta^{-3\alpha+5\sigma}.
\end{equation}
We will establish the bound \eqref{modifiedVolOfTriples} by applying Theorem \ref{Shm} to $\tilde G$. To do so, we must verify that $\tilde G$, $\tilde A_1$, $\tilde A_2$, and $\tilde A_3$ satisfy the hypotheses of Theorem \ref{Shm}.

\medskip
\noindent \textbf{Verifying the conditions for Theorem \ref{Shm}}
Let $\eta = (1-\alpha)\kappa$, and let $\tau>0$ be the value obtained by applying Theorem \ref{Shm} with this choice of $\eta$. Let 
\begin{equation}\label{sigmasmall}
\sigma < \min\big( (1-\alpha)\kappa/20,\ \tau/11\big).
\end{equation}
Define
\begin{equation*}
T:=\{(a_1,a_2,a_3)\in \tilde A\colon \tilde G(a_1,a_2,a_3)\in \tilde R\}.
\end{equation*}
Since $\tilde A_1,\tilde A_2$ an $\tilde A_3$ are unions of $\delta$ intervals, to obtain \eqref{modifiedVolOfTriples} it suffices to prove that
\begin{equation}\label{desiredBoundOnT}
\vol_3(T)\leq \delta^{3-3\alpha+5\sigma}.
\end{equation}

First, we can suppose that for each index $i$, 
\begin{equation}\label{AiNotTooSmall}
\vol_1(\tilde A_i)\geq \delta^{1-\alpha+14\sigma}.
\end{equation} 
Indeed, if \eqref{AiNotTooSmall} fails for some index $i$ then \eqref{desiredBoundOnT} follows from the trivial bound 
\begin{equation*}
\begin{split}
\vol_3(T) &\leq \vol_3(\tilde A_1 \times \tilde A_2 \times \tilde A_3)\\
&\leq \delta^{2-2\alpha-9\sigma}\min_{1\leq i\leq 3}\vol_1(A_i),
\end{split}
\end{equation*}
where on the last line we used \eqref{nonConcentrationTildeAi} with $J=[0,1]$.

Let $X=\tilde A_1\times \tilde A_2$ and let $Z = \tilde A_3$. We will verify that $\tilde G$ and the sets $X$ and $Z$ satisfy the hypotheses of Theorem \ref{Shm} with parameter $\eta$. To begin, \eqref{boundsOnG} follows from \eqref{infGradientTildeG} and \eqref{C2NormTildeG}. 

Next we will verify that $X$ satisfies the non-concentration estimate \eqref{Xnonconcentration}. Let $r = \vol_2(X)^{1/2}\leq \delta^{1-\alpha-3\sigma}$. Let $B(x,r)$ be a ball of radius $r$, and let $J_1\times J_2$ be a square of side-length $2r$ that contains $B(x,r)$. Then by \eqref{nonConcentrationTildeAi}, we have
\begin{equation*}
\begin{split}
\vol_2(X\cap B(x,r) )
&\le \vol_1(\tilde A_1\cap J_1)\vol_1(\tilde A_2\cap J_2)\\
&\leq \big( (\delta^{-3\sigma}(2\delta^{1-\alpha-3\sigma})^{\kappa}\delta^{1-\alpha}\big)^2\\
&\leq 4^\kappa \delta^{2\kappa(1-\alpha) - 40\sigma}  |\tilde A_1| |\tilde A_2|\\
&\leq \delta^{2\kappa(1-\alpha) - 40\sigma} \vol_2(X),
\end{split}
\end{equation*}
where on the third line we used \eqref{AiNotTooSmall} and on the last line we used \eqref{sigmasmall}. 

Finally, we will verify the transversality estimate \eqref{thetaimagenoncon}. Note that for each $(x,y)\in \tilde I_1\times \tilde I_2$, the function $\tilde\theta_{(x,y)}(z)$ is monotone and continuous on $\tilde I_3$. Thus for each interval $K$, the pre-image
\begin{equation*}J:=\theta_{(x,y)}^{-1}(K)\end{equation*} 
is  necessarily an interval.
Moreover, \eqref{sizeThetaTildeG} implies that for every $z_1,z_2\in J$, we have
\begin{equation*}
\frac{|\tilde \theta_{(x,y)}(z_1)-\tilde \theta_{(x,y)}(z_2)|}{|z_1-z_2|}\ge \delta^{\sigma}.
\end{equation*}
Thus $|J|\le \delta^{-\sigma}| K|$. 
It follows that for every interval $K\subset \RR$ of length at most $2r$, 
we have
\begin{align*}
|\{z\in \tilde A_3\colon \tilde \theta_{(x,y)}(z)\in K\} |
&=|\tilde A_3\cap J| \\
&\leq \delta^{-\sigma}|J|^{\kappa}\delta^{1-\alpha}\\
&\leq \delta^{-\sigma-\kappa \sigma}|K|^{\kappa}\frac{\delta^{1-\alpha}}{|\tilde A_3|}|\tilde A_3|\\ 
&\le \delta^{-(6+\kappa)\sigma}|\tilde A_3|r^{\eta}r^{\alpha\kappa}2^{\kappa}\\
&\le \delta^{-\tau}|\tilde A_3|r^{\eta},
\end{align*}
where we used \eqref{sigmasmall} and \eqref{AiNotTooSmall}.
Thus \eqref{thetaimagenoncon} holds.

\medskip

\noindent \textbf{Applying Theorem \ref{Shm}}
Apply Theorem \ref{Shm} to $\tilde G$, $X=\tilde A_1\times \tilde A_2$, and $Z=\tilde A_3$. We obtain a set $(\tilde A_3)_{\operatorname{bad}}\subset  \tilde A_3$ with
\begin{equation*}
\vol_1((\tilde A_3)_{\operatorname{bad}})\leq \delta^{\tau}\vol_1(\tilde A_3).
\end{equation*}

For each $c\in \tilde A_3$, define
\begin{equation*}
X_c:=\{(x,y)\in \tilde A_1\times \tilde A_2\colon \tilde G(x,y,c)\in \tilde R\}.
\end{equation*}
Let
\begin{equation*}
\tilde A_3':=\left\{z\in \tilde A_3\colon \vol_2(X_z)\ge \tfrac{\vol_3(T)}{2\vol_1(\tilde A_3)}\right\}.
\end{equation*}

Now, suppose \eqref{desiredBoundOnT} fails; we will obtain a contradiction. By Fubini, we have 
\begin{equation*}\vol_3(T)\le \vol_1(\tilde A_3')\vol_1(\tilde A_1)\vol_1(\tilde A_2)+\vol_1(\tilde A_3\setminus \tilde A_3')\cdot \tfrac{\vol_3(T)}{2\vol_1(\tilde A_3)},
\end{equation*}
which implies
\begin{align*}
|\tilde A_3'|
&\ge \frac{\vol_3(T)}{2|\tilde A_1||\tilde A_2|}\\
&=\frac{\vol_3(T)}{2|\tilde A_1||\tilde A_2||\tilde A_3|}|\tilde A_3|\\
&\ge \tfrac12 \delta^{14\sigma}|\tilde A_3|,
\end{align*}
where on the final line we used \eqref{nonConcentrationTildeAi} (with $J=[0,1]$). 

In particular, since $\sigma>0$ satisfies \eqref{sigmasmall} (and assuming $\delta>0$ is sufficiently small depending on $\alpha$ and $\kappa$), we have $|\tilde A_3'|> \delta^{\tau}|\tilde A_3|$ and thus
\begin{equation}
\tilde A_3'\not\subseteq (\tilde A_3)_{\rm bad}.
\end{equation}

Fix an element $z\in \tilde A_3'\backslash (\tilde A_3)_{\rm bad}$ and define $X^\prime = X_z$. 
By the definition of $\tilde A_3^\prime$, we have 
\begin{equation}
\vol_2(X_z)\ge
\tfrac{\vol_3(T)}{2|\tilde A_3|}\ge  \tfrac12\delta^{2-2\alpha+5\sigma}\geq \tfrac12 \delta^{11\sigma}\vol_2(X)\geq \delta^{\tau}\vol_2(X),
\end{equation}
and thus by Theorem \ref{Shm} we have
\begin{equation}
\mathcal{E}_{\delta}\big(\{\tilde G(x,y,z)\colon (x,y)\in X_z\}\big) \geq \delta^{-\tau}\mathcal{E}_{\delta}(X)^{1/2}.
\end{equation}
On the other hand, 
\begin{equation}
\mathcal{E}_{\delta}\big(\{\tilde G(x,y,z)\colon(x,y)\in X_z\}\big)
\le \mathcal{E}_{\delta}(\tilde R)
\le  \delta^{-\alpha-4\sigma}
\leq \delta^{-10\sigma}\mathcal{E}_{\delta}(X)^{1/2}\le \delta^{-\tau}\mathcal{E}_{\delta}(X)^{1/2}.
\end{equation}
This is a contradiction. We conclude that \eqref{desiredBoundOnT} holds, which completes the proof.
\end{proof}

\subsection{Reformulating Theorem \ref{Shm} as an energy dispersion estimate}\label{reformulationEnergyDispersionSec}
As the title suggests, in this section we will reformulate Theorem \ref{Shm} as an energy dispersion estimate. The basic idea is that we can count the number of solutions to $P(x,y) = P(x',y')$  with $x,x'\in A,\ y,y'\in B$ by counting the size of the intersection 
\begin{equation}\label{informalIntersection}
\{(x,x',y,y'): P(x,y)-P(x',y')=0\} \cap (A^2\times B^2).
\end{equation}
We will express the surface $Z(P(x,y)-P(x',y'))$ as the graph $y' = G(x,x',y'),$ and then use Lemma \ref{Shmcor} to estimate the size (or more accurately, $\delta$-covering number) of the set \eqref{informalIntersection}. We now turn to the details.

\begin{defn}\label{defnHFDefn}
If $F(x,x',y,y')$ is a function (that is at least twice differentiable on its domain), define
\begin{equation}\label{defnHF}
\begin{split}
H_F(x,x',y,y') &= (\partial_x F)(\partial_{y'} F)(\partial_{x'y}F)-(\partial_xF)(\partial_yF)(\partial_{x'y'}F)\\
&\qquad-(\partial_{x'}F)(\partial_{y'}F)(\partial_{xy}F)+(\partial_{x'}F)(\partial_yF)(\partial_{xy'}F).
\end{split}
\end{equation}
We will be particularly interested in functions $F$ of the form $F(x,x',y,y') = P(x,y) - P(x',y')$. In this case $H_F$ simplifies to
\begin{equation}\label{defnHFForP}
H_F(x,x',y,y') = \partial_xP(x,y)\partial_yP(x,y)\partial_{x'y'}P(x',y')-\partial_{x'}P(x',y')\partial_{y'}P(x',y')\partial_{xy}P(x,y).
\end{equation}
\end{defn}

\begin{defn}
Let $F\colon I_1\times I_2\times I_3\times I_4\to\RR$, and let $\pi$ be the projection to the first three coordinates.  We say $F$ has full projection if $\pi(Z(F))=I_1\times I_2\times I_3$. 
\end{defn}

\begin{lem}\label{EnergyDispersionForF}
For every $0< \alpha< 1$ and $0 < \kappa \leq \alpha,$ there exists $\eps=\eps(\alpha,\kappa)>0$ and $\delta_0 = \delta_0(\alpha,\kappa)$ such that the following holds for all $0<\delta<\delta_0$.

Let $I=I_1\times I_2\times I_3\times I_4\subset[0,1]^4$ be intervals and let $F\colon I\to\RR$ be a smooth function with full projection. Suppose that 
\begin{equation}\label{HFBigOnZAndLowerBoundPartials}
|H_F|,\ |\partial_x F|,\ |\partial_y F|,\ |\partial_{x'} F|,\ |\partial_{y'} F| \geq\delta^{\eps}\quad\textrm{on}\ Z(F)\cap I,
\end{equation}
and 
\begin{equation}\label{C3BoundForF}
\Vert F\Vert_{C^2(I)}\leq\delta^{-\eps}.
\end{equation}

Let $A_i\subset I_i$, $i=1,2,3,4$ be sets. Suppose that for each index $i=1,\ldots,4$ and each interval $J$ of length at least $\delta$, $A_i$ satisfies the non-concentration condition
\begin{equation}\label{noncon4}
\mathcal{E}_{\delta}(A_i\cap J)\le |J|^\kappa \delta^{-\alpha-\eps}.
\end{equation}


Then we have the energy dispersion estimate
\begin{equation}\label{energyNonConcentrationInLem}
\mathcal{E}_{\delta}\big((A_1\times A_2\times A_3\times A_4) \cap Z(F)  \big) \leq  \delta^{-3\alpha+\eps}.
\end{equation}
\end{lem}
\begin{proof}
By the bound \eqref{HFBigOnZAndLowerBoundPartials} for $\partial_{y'} F$ and the assumption that $F$ has full projection, we have that for each $(x,x',y)\in I_1\times I_2\times I_3$, there is precisely one $y' \in I_4$ so that $F(x,x',y,y')=0$, and
\begin{equation*}
\begin{split}
\mathcal{E}_{\delta}\big((A_1\times A_2\times A_3\times & A_4) \cap Z(F)  \big) \\
&\leq \delta^{-\eps}\mathcal{E}_{\delta}\{(x,x',y)\in A_1\times A_2\times A_3\colon\ \exists\ y'\in A_4\ \textrm{s.t.}\ F(x,x',y,y')=0\}.
\end{split}
\end{equation*}
In particular, to establish \eqref{energyNonConcentrationInLem} it suffices to prove the estimate
\begin{equation}\label{sufficesToBoundProjection}
\mathcal{E}_{\delta}\Big( \pi\big( (A_1\times A_2\times A_3\times  A_4) \cap Z(F) \big) \Big)\leq \delta^{-3\alpha+2\eps}.
\end{equation}


Since $F$ has full projection, there is a smooth function $G\colon I_1\times I_2\times I_3 \to I_4$ so that $F(x,x',y,G(x,x',y))=0$ for all $(x,x',y)\in I_1\times I_2\times I_3$.
By the implicit function theorem, \eqref{C3BoundForF}, and \eqref{HFBigOnZAndLowerBoundPartials},  we have the bound
\begin{equation*}
\delta^{2\eps} \le  \min_{I}  \left|\frac{\partial_x F}{\partial_{y'} F}\right|\le  \max_{I}  \left|\frac{\partial_x F}{\partial_{y'} F}\right| \le  \delta^{-2\eps},
\end{equation*}
and similarly for $\left|\partial_{x'} /\partial_{y'} F\right|$ and $\left|\partial_{y} F/\partial_{y'} F\right|$. Thus

\begin{equation}\label{nablaGEstimate}
\delta^{2\eps}\le   \min_{(x,x',y)\in I_1\times I_2\times I_3}|\nabla G_{(y)}(x,x')|.
\end{equation}

By continuing to differentiate our expression for $G$ and using \eqref{HFBigOnZAndLowerBoundPartials} for  $\partial_{y'} F$ and \eqref{C3BoundForF}, we obtain the bound

\begin{equation}\label{C2GEstimate}
\Vert G_{(y)}\Vert_{C^2(I_1\times I_2\times I_3)}\lesssim \delta^{-10\eps}.
\end{equation}

Define 
\begin{equation*}
\varphi_{(x,x')}(y):=\frac{\partial_{x'} G(x,x',y)}{\partial_x G(x,x',y)} =\tan(\theta_{(x,x')}(y)).
\end{equation*} 
We have
\begin{equation}\label{expressVarPhi}
\varphi_{(x,x')}(y)=\frac{\partial_{x'} F}{\partial_x F},
\end{equation}
and 
\begin{equation}
\begin{split}\label{expressionForPartialYVarphi}
\partial_y & \varphi_{(x,x')}(y)\\
&=\frac{(\partial_x F)(\partial_{y'} F)(\partial_{x'y}F)-(\partial_xF)(\partial_yF)(\partial_{x'y'}F)-(\partial_{x'}F)(\partial_{y'}F)(\partial_{xy}F)+(\partial_{x'}F)(\partial_yF)(\partial_{xy'}F)}{(\partial_{y'} F) (\partial_x F)^2}\\
&=\frac{H_F}{(\partial_{y'} F) (\partial_x F)^2},
\end{split}
\end{equation}
where in \eqref{expressVarPhi} and \eqref{expressionForPartialYVarphi}, $F$ (and its partial derivatives) and $H_F$ are evaluated at the point $(x,y,G(x,x',y))$. Thus by \eqref{HFBigOnZAndLowerBoundPartials} and \eqref{C3BoundForF}, we have
\begin{equation}
|\varphi_{(x,x')}(y)|\leq\delta^{-2\eps},
\end{equation}
while by \eqref{HFBigOnZAndLowerBoundPartials} and \eqref{expressionForPartialYVarphi}, we have
\begin{equation}\label{comparePartialPhiWithHF}
|\partial_y \varphi_{(x,x')}(y)|\geq \delta^{6\eps}|H_F|\geq\delta^{7\eps}.
\end{equation}
Next, we compute 
\begin{equation*}
\begin{split}
|\partial_y \varphi_{(x,x')}(y)|
&=|\partial_y(\tan(\theta_{(x,x')}))|\\
&=|(1+\tan^2(\theta_{(x,x')}(y))||\partial_y \theta_{(x,x')}(y)|,
\end{split}
\end{equation*}
and thus
\begin{equation}\label{derivativePartialThetaBig}
|\partial_y \theta_{(x,x')}(y)|= \frac{|\partial_y \varphi_{(x,x')}(y)|}{1+(\varphi_{(x,x')}(y))^2}\geq \delta^{11\eps}.
\end{equation}

The estimates \eqref{nablaGEstimate}, \eqref{C2GEstimate}, and \eqref{derivativePartialThetaBig} are precisely the requirements needed to apply Lemma \ref{Shmcor} to $G$ and the sets $A_1,A_2,A_3,$ and $R = A_4$.
If $\eps\leq \eps_1(\alpha,\kappa)>0$ and $0<\delta<\delta_0(\alpha,\kappa)>0$ are selected sufficiently sufficiently small, then there exists $\sigma=\sigma(\alpha,\kappa)>0$ so that
\begin{equation}\label{bdOnZcapQcapA}
\mathcal{E}_{\delta}\Big( \pi\big( (A_1\times A_2\times A_3\times  A_4) \cap Z(F) \big) \Big) \leq \delta^{-3\alpha+\sigma}.
\end{equation}
We conclude that \eqref{sufficesToBoundProjection} hold for some $\eps=\eps(\alpha,\kappa)>0$.
\end{proof}

\section{The auxiliary function $K_P$}\label{auxiliaryFunctionSection}
In this section we will define and study the auxiliary function $K_P$. 

\begin{defn}
Let $P(x,y)$ be smooth. Define
\begin{equation*}
K_P(x,y) = \nabla P\wedge \nabla\Big(\frac{\partial_x P\partial_y P}{\partial_{xy}P}\Big).
\end{equation*}
\end{defn}
\noindent A computation shows that (formally),
\begin{equation}\label{relateKpLog}
\begin{split}
\big(\partial_x P \partial_y P\big)^2&\partial_{xy}\left(\log\left(\frac{\partial_xP}{\partial_yP}\right)\right)\\
&=(\partial_yP)^2\big( \partial_x P \partial_{xxy}P - \partial_{xx}P\partial_{xy}P\big)
-(\partial_x P)^2 \big( \partial_y P \partial_{xyy}P - \partial_{xy}P\partial_{yy}P\big)\\
&\qquad\qquad\qquad\qquad\qquad\qquad\qquad\qquad\qquad\qquad=\big(\partial_{xy}P\big)^2\Big(\nabla P \wedge \nabla \Big(\frac{\partial_x P \partial_y P}{\partial_{xy}P}\Big) \Big).
\end{split}
\end{equation}
Then the middle expression is well-defined on the domain of $P$. The first equality holds on the domain of the LHS of \eqref{relateKpLog} (here we extend $\log$ to $\RR\backslash\{0\}$ by defining $\log(-x)=\log(x)$ for $x$ positive), while the second equality holds on the domain of the RHS of \eqref{relateKpLog}. The auxiliary function $K_P$ is useful because of the following result from Elekes and Ronyai~\cite{ER}. A detailed proof can be found in \cite[Lemma 10]{RaSh}.
\begin{lem}[Elekes and R\'onyai]\label{special}
Let $U\subset\RR^2$ be a connected open set and let $P\colon U\to\RR$ be analytic. If $P_x$ or $P_y$ vanishes identically on $U$, then $P$ is an analytic special form, in the sense of Definition \ref{defnAnalyticSpecialForm}. If neither $P_x$ nor $P_y$ vanishes identically on $U$, then $P$ is an analytic special form if and only if $\partial_{xy}\left(\log\left(\frac{\partial_xP}{\partial_yP}\right)\right)$ vanishes everywhere on $U$ that it is defined (i.e.~everywhere $\partial_xP$ and $\partial_yP$ are non-zero). 
\end{lem}

The identity \eqref{relateKpLog} and Lemma~\ref{special} have the following consequence.
\begin{lem}\label{whenIsPSpecialFormProp}
Let $U\subset\RR^2$ be a connected open set and let $P\colon U\to\RR$ be analytic (resp.~polynomial). Suppose that none of $\partial_x P$, $\partial_y P$, and $\partial_{xy}P$ vanishes identically on $U$. Then $K_P$ vanishes identically on $U$ if and only if $P$ is an analytic (resp.~polynomial) special form.
\end{lem}

Next, we will show that if $K_P$ is far from 0, then $P$ satisfies an energy dispersion estimate.

\begin{prop}\label{countingNumberQuadruplesNablePWedgeP}
For every $0< \alpha< 1$ and $0 < \kappa \leq \alpha,$ there exists $\eps=\eps(\alpha,\kappa)>0$ and $\delta_0 = \delta_0(\alpha,\kappa)>0$ such that the following holds for all $0<\delta<\delta_0$. Let $I\subset[0,1]^2$ be a rectangle, and let $P\colon I\to[0,1]$ be smooth. Suppose that
\begin{equation}\label{lowerBoundPartialsAndK}
|\partial_x P|,\ |\partial_y P|,\ |\partial_{xy}P|,\ |K_P|\geq\delta^\eps\quad\textrm{on}\ I,
\end{equation}
and
\begin{equation}\label{C3BoundsForP}
\Vert P\Vert_{C^3(I)}\leq\delta^{-\eps}.
\end{equation} 

%
Let $A,B$ be sets with $A\times B\subset I$. Suppose that for all intervals $J$ of length at least $\delta$, $A$ and $B$ satisfy the non-concentration condition.
\begin{equation}\label{nonConcentrationABApBp}
\begin{split}
\mathcal{E}_{\delta}(A \cap J) &\le  |J|^\kappa\delta^{-\alpha-\eps},\\
\mathcal{E}_{\delta}(B \cap J) &\le  |J|^\kappa\delta^{-\alpha-\eps}.
\end{split}
\end{equation}
Then
\begin{equation}\label{semiDiscretizedEnergy}
\mathcal{E}_{\delta}\big(\{ (x,x',y,y')\in A\times A\times B\times B \colon P(x,y) = P(x',y')\}\big)\leq \delta^{-3\alpha+\eps}.
\end{equation}
\end{prop}
\begin{proof}
Define $F(x,x',y,y') =P(x,y)-P(x',y')$. If $I=I_1\times I_2$, then $F$ has domain $I'=I_1\times I_1\times I_2\times I_2$. By \eqref{lowerBoundPartialsAndK}, $\partial_{y'}F$ never changes sign on $I'$, and thus for each $(x,x',y)\in I_1\times I_1\times I_2$, there is at most one $y'\in I_2$ with $F(x,x',y,y')=0$. Our next task is to cut the domain of $F$ into smaller rectangles, so that on most rectangles, for each $x,x',y$, there is precisely one $y'$ with $F(x,x',y,y')=0$. The size of \eqref{semiDiscretizedEnergy} will be bounded by the corresponding sum taken over these sets (plus a small error term corresponding to the rectangles where the above property fails). 

Let $t_0=t_0(\kappa)>0$ be a small constant to be specified below. Cover $I'$ by interior-disjoint rectangles, each of whose dimensions is $~\frac{t_0}{6}\delta^{2\eps+5\eps/\kappa}$ (a slight adjustment might be needed if $6|I_i|/ (t_0\delta^{2\eps+5\eps/\kappa})$ is not an integer).
For each such rectangle $J=J_1\times J_2\times J_3\times J_4$, let $\tilde  J_4$ be an interval of length $t_0\delta^{5\eps/\kappa}$, with the same midpoint as $J_4$. Define $\tilde J = J_1\times J_2\times J_3\times \tilde J_4$. 

We claim that if $J\cap Z(F)\neq\emptyset$ and if $\tilde J\subset I',$ then for each $(x,x',y)\in J_1\times J_2\times J_3$, there is $y'\in\tilde J_4$ with $F(x,x',y,y')=0$. To prove this claim, observe that the assumption $J\cap Z(F)\neq\emptyset$ implies there exists $(x_0,x_0',y_0,y_0')\in J$ with $P(x_0,y_0)=P(x_0',y_0')$. By \eqref{C3BoundsForP}, we thus have 
\begin{align*}
|P(x,y)-P(x',y_0')|
&\le |P(x,y)-P(x_0,y_0)|+|P(x_0,y_0)-P(x_0',y_0')|+|P(x_0',y_0')-P(x',y_0')|\\
&\le \delta^{-\eps}|(x,y)-(x_0,y_0)|+ 0 +\delta^{-\eps}|x_0-x'|\\
&\leq \delta^{-\eps}\operatorname{diam}(J)\leq \frac{t_0}{2}\delta^{\eps+5\eps/\kappa}.
\end{align*} 
By \eqref{lowerBoundPartialsAndK}, there exists $y$ with $|y_0-y|\leq \frac{(t_0/2)\delta^{\eps+5\eps/\kappa}}{\delta^{\eps}}=\frac{t_0}{2}\delta^{5\eps/\kappa}$ with $P(x,y)=P(x',y')$; by our definition of $\tilde J_4$ we have $y\in\tilde J_4$. This completes the proof of the claim.

Finally, we will estimate the contribution to \eqref{semiDiscretizedEnergy} coming from rectangles $J$ for which $\tilde J$ is not contained in $I'$. Such rectangles must be of the form $J_1\times J_2\times J_3\times J_4$, where $J_4$ is near one of the endpoints of $I_2$; specifically, if $I_2=[a,b]$, then $J_4\subset [a, a+t_0\delta^{5\eps/\kappa}]\cup [b-t_0\delta^{5\eps/\kappa}, b]$. By \eqref{nonConcentrationABApBp}, if the constant $t_0=t_0(\kappa)>0$ is selected sufficiently small, then 
\[
\mathcal{E}_{\delta}\big(B\cap [a, a+t_0\delta^{5\eps/\kappa}]\big) \leq \delta^{-\eps}(t_0\delta^{5\eps/\kappa})^\kappa\delta^{-\alpha}\leq \frac{1}{4}\delta^{4\eps-\alpha},
\] 
and similarly for  $B\cap [b-t_0\delta^{5\eps/\kappa}, b].$ Using the bound $\mathcal{E}_{\delta}(A)\leq\delta^{-\eps-\alpha}$ and $\mathcal{E}_{\delta}(B)\leq\delta^{-\eps-\alpha}$ from \eqref{nonConcentrationABApBp}, we have  
\[
\mathcal{E}_{\delta}(\{(x',y,y')\in A\times B\times B\colon y'\in [a, a+t_0\delta^{5\eps/\kappa}]\cup [b-t_0\delta^{5\eps/\kappa}, b]\})\leq\frac{1}{2}\delta^{2\eps-3\alpha}.
\]
By \eqref{C3BoundsForP} (specifically the bound on $|\partial_{x}P|\leq\delta^{-\eps})$, this implies
\[
\mathcal{E}_{\delta}(\{(x,x',y,y')\in A\times A\times B\times B\colon P(x,y)=P(x',y'),\ y'\in [a, a+t_0\delta^{5\eps/\kappa}]\cup [b-t_0\delta^{5\eps/\kappa}, b]\})\leq\frac{1}{2}\delta^{\eps-3\alpha}.
\]
Thus the contribution from rectangles $J$ with $\tilde J\not\subset I'$ is at most half our desired final bound \eqref{semiDiscretizedEnergy}. 

The number of rectangles $J$ with $\tilde J\subset I'$ is at most  $6^4(t_0\delta^{2\eps+5\eps/\kappa})^{-4}\leq 6^4 t_0^{-4}\delta^{-28\eps/\kappa}$ (a better estimate is possible since not all rectangles $J$ intersect $Z(F)$, but this will not matter for our arguments). Thus to establish the bound \eqref{semiDiscretizedEnergy}, it suffices to establish the bound
\begin{equation}\label{semiDiscretizedEnergyInJ}
\mathcal{E}_{\delta}\big(\{ (x,x',y,y')\in J\cap (A\times A\times B\times B) \colon P(x,y) = P(x',y')\}\big)\leq \delta^{-3\alpha+30\eps/\kappa}
\end{equation}
for each rectangle $J$ with $J\cap Z(F)\neq\emptyset$ and $\tilde J\subset I'$. 

For the remainder of the argument we will fix such a rectangle $J=J_1\times J_2\times J_3\times J_4$. Unless stated otherwise, when we refer to variables $x,x',y,y'$, we will assume that $(x,x',y,y')\in J$. By construction, there is a smooth function $G\colon J_1\times J_2\times J_3\to\tilde J_4$ so that $F(x,x',y',G(x,x',y))=0$.

Recall Definition \ref{defnHFDefn}, which defines the auxiliary function $H_F$. By \eqref{defnHFForP} we have 
\begin{equation*}
\begin{split}
H_F(x,x',y,y') 
&=\partial_{xy}P(x,y)\partial_{x'y'}P(x^\prime,y^\prime)\Big(\frac{\partial_x P (x,y) \partial_y P(x,y)}{\partial_{xy}P(x,y)}-\frac{\partial_{x'}P(x^\prime,y^\prime)\partial_{y'}P(x^\prime,y^\prime)}{\partial_{x'y'}P(x^\prime,y^\prime)}\Big).
\end{split}
\end{equation*} 
Our next task is to control the size of the region where $H_F(x,x',y,G(x,x',y))$ is small. Fix $s_0>0$.
%
%
%

%
%
Fix $(x_0,y_0)\in J_1\times J_2$ and suppose there is a point $x_0'\in J_3$ with $|H(x_0,x_0',y_0,G(x_0,x_0',y_0))|\leq s_0$. Consider the curve 
\[
\gamma = \{(x',y')\in J_3\times \tilde J_4 \colon P(x_0,y_0) = P(x^\prime,y^\prime)\} = \{(x', G(x_0,x',y_0))\colon x'\in J_3\}.
\]
Then $\gamma$ is a smooth curve, and, by \eqref{lowerBoundPartialsAndK} and \eqref{C3BoundsForP}, at each point the tangent to this curve has slope $\leq\delta^{-2\eps}$. Let $\gamma(t)\colon S\to \RR^2$ be a unit speed parameterization of $\gamma$, where $S\subset \RR$ is an interval, $y_0'=G(x_0,x_0',y_0)$, and $\gamma(0) = (x_0',y_0')$. 

For each $t\in S$, the unit tangent vector to $\gamma$ at $t$ is given by $\frac{\nabla P^\perp (\gamma(t))}{|\nabla P(\gamma(t))|}$. Thus for all $t\in S$ we have
\begin{equation*}
\begin{split}
\Big|\partial_t\Big(\frac{\partial_xP(\gamma(t))\partial_yP(\gamma(t))}{\partial_{xy}P(\gamma(t))}\Big)\Big|&=\Big|\frac{\nabla P^\perp}{|\nabla P|} \cdot \nabla\big(\frac{(\partial_x P)(\partial_y P)}{\partial_{xy}P}\big)(\gamma(t))\Big| \\
&= \frac{1}{|\nabla P|} \Big|\Big (\nabla P\wedge \nabla\big(\frac{(\partial_x P)(\partial_y P)}{\partial_{xy}P}\big)\Big)(\gamma(t))\Big|\\
&= \frac{|K_P(\gamma(t))|}{|\nabla P(\gamma(t))|}\\
&\geq \delta^{2\eps}.
\end{split}
\end{equation*}
Therefore if $\gamma(t) = (x'(t),y'(t))$, then 
\[
\Big|\frac{\partial_{x'}P(x'(0),y'(0))\partial_{y'}P(x'(0),y'(0))}{\partial_{x'y'}P(x'(0),y'(0)}-\frac{\partial_{x'}P(x'(t),y'(t))\partial_{y'}P(x'(t),y'(t))}{\partial_{x'y'}P(x'(t),y'(t))}\Big|\geq \delta^{2\eps}t,
\]
and hence

\begin{equation}
\begin{split}
&\Big|\frac{\partial_x P (x_0,y_0) \partial_y P(x_0,y_0)}{\partial_{xy}P(x_0,y_0)}-\frac{\partial_{x'}P(x^\prime(t),y^\prime(t))\partial_{y'}P(x^\prime(t),y^\prime(t))}{\partial_{x'y'}P(x^\prime(t),y^\prime(t))}\Big|\\
&\geq \Big|\frac{\partial_{x'}P(x^\prime(t),y^\prime(t))\partial_{y'}P(x^\prime(t),y^\prime(t))}{\partial_{x'y'}P(x^\prime(t),y^\prime(t))}-\frac{\partial_{x'}P(x^\prime(0),y^\prime(0))\partial_{y'}P(x^\prime(0),y^\prime(0))}{\partial_{x'y'}P(x^\prime(0),y^\prime(0))}\Big| \\
&\qquad - \Big|\frac{\partial_x P (x_0,y_0) \partial_y P(x_0,y_0)}{\partial_{xy}P(x_0,y_0)}-\frac{\partial_{x'}P(x^\prime(0),y^\prime(0))\partial_{y'}P(x^\prime(t),y^\prime(t))}{\partial_{x'y'}P(x^\prime(0),y^\prime(0))}\Big|\\
&\geq \delta^{2\eps}t - \delta^{-2\eps}s_0.
\end{split}
\end{equation}
We conclude that
\begin{equation}\label{derivativeOfHFAsFuncT}
|H_F(x_0,x'(t),y_0,y'(t))|
\geq\big|\partial_{xy}P(x,y)\partial_{x'y'}P(x^\prime(t),y^\prime(t))\big|\big(\delta^{2\eps}t - s\big)\geq \delta^{2\eps}(\delta^{2\eps}t-2^{-\eps}s_0).
\end{equation}

Using \eqref{C3BoundsForP}, one can easily verify that $|\nabla H_F|\le 12 \delta^{-3\eps}$. 
Combined with \eqref{derivativeOfHFAsFuncT}, we have that 
if $(x,y)\in J_1\times J_2$ with $|x-x_0|\leq\delta$ and $|y-y_0|\leq\delta$, then
\begin{equation}\label{derivativeOfHFAsFuncTAtXY}
|H_F(x,x'(t),y,y'(t))|
\geq \delta^{2\eps}(\delta^{2\eps}t-\delta^{-2\eps}s_0)-12\delta^{1-3\eps}.
\end{equation}

Recall that $\gamma(t)$ is a unit speed parameterization, and so $|(x_0',y_0')-\gamma(t)|\le t$. We conclude that for each $(x_0,y_0)\in J_1\times J_2$, if there exists a point $x_0'\in J_3$ with $|H(x_0,x_0',y_0,G(x_0,x_0',y_0))|\leq s_0$, then for each $s\geq s_0$ we have that the set 
\[
\{x'\in J_3\colon |H_F(x,x',y,G(x,x',y))|\leq s\ \quad\textrm{for some}\ x\in J_1\cap [x_0-\delta,x_0+\delta],\ y\in J_2\cap [y_0-\delta,y_0+\delta]\}
\]
is contained in an interval of length at most $2\delta^{-6\eps}s$ (this interval necessarily contains $x_0'$).

Since there are at most $\delta^{-2\alpha-2\eps}$ $\delta$-separated points $(x_0,y_0)\in J_1\times J_2$ for which there exists $x_0'\in J_3$ with $|H(x_0,x_0',y_0,G(x_0,x_0',y_0))|\leq s_0$, we can use \eqref{nonConcentrationABApBp} to bound the covering number of each corresponding interval; setting $s=2s_0$, we conclude that
\begin{equation}\label{coveringNumberSmallH}
\begin{split}
\mathcal{E}_{\delta}\Big\{(x,x',y)\colon |H(x,x',y,G(x,x',y))|\leq s_0\Big\} & \lesssim \delta^{-2\alpha-2\eps}\big((\delta^{-6\eps}(2s_0))^\kappa\delta^{-\alpha-\eps}\big)\\
& \lesssim \delta^{-3\alpha-9\eps}s_0^\kappa.
\end{split}
\end{equation}
Since $\Vert \nabla G\Vert \leq\delta^{-2\eps}$, this implies
\begin{equation}\label{coveringNumberBdCloseToW}
\mathcal{E}_{\delta}\big(\{ (x,x',y,y')\in J_1\times J_2\times J_3\times J_4\colon P(x,y)=P(x',y'),\ |H(x,x',y,y')|\leq s_0\} \big)\lesssim \delta^{-3\alpha-11\eps}s_0^\kappa.
\end{equation}

Next, let 
\begin{equation}\label{defnW}
W = \Big\{(x,x',y)\in J_1\times J_2\times J_3\colon |H(x,x',y,G(x,x',y))|\leq c_0\delta^{5\eps}s_0\Big\},
\end{equation}
where $c_0>0$ is a small constant to be specified below. 
Cover $\RR^3 \backslash W$ by interior-disjoint cubes, with the property that each cube $Q$ in the cover has side-length comparable to its distance from $W$. The existence of such a collection of cubes is guaranteed by the Whitney cube decomposition (see Theorem \ref{whitneyCubeDecomp} below). Let $\mathcal{Q}_1$ denote the set of cubes $Q$ satisfying $\operatorname{dist}(Q,W)\leq c_0\delta^{5\eps}s_0$, and let $\mathcal{Q}_2$ denote the set of cubes $Q$ that intersect $J_1\times J_2\times J_3$, but which are not in $\mathcal{Q}_1$. 

We claim that
\begin{equation}
\bigcup_{Q\in\mathcal{Q}_1}Q \subset \Big\{(x,x',y)\colon |H(x,x',y,G(x,x',y))|\leq s_0\Big\}.
\end{equation}
Indeed, if $Q\in\mathcal{Q}_1$, then $\operatorname{dist}(Q, W)\leq c_0\delta^{5\eps}s_0$. Since the diameter of $Q$ is $O(c_0\delta^{5\eps}s_0)$, we have that each point in $Q$ has distance $O(c_0\delta^{5\eps}s_0)$ from $W$.
Using \eqref{lowerBoundPartialsAndK}, and \eqref{C3BoundsForP} one can verify that the gradient of the trivariate function 
$$
(x,x',y)\mapsto H(x,x',y, G(x,x',y))$$ has norm $O(\delta^{-5\eps})$. 
Thus for every $(x,x',y)\in Q$, we have

\begin{equation}\label{sizeOfHOfG}
|H(x,x',y,G(x,x',y))|\lesssim \delta^{-5\eps}(c_0\delta^{5\eps}s_0).
\end{equation}
If we select $c_0$ from \eqref{defnW} sufficiently small, then the LHS of \eqref{sizeOfHOfG} is at most $s_0$. Fixing this choice of $c_0$, we have

\begin{equation}\label{decomposeJ1J2J3}
J_1\times J_2\times J_3 \subset \bigcup_{Q\in\mathcal{Q}_2}Q \ \cup\ \Big\{(x,x',y)\colon |H(x,x',y,G(x,x',y))|\leq s_0\Big\}.
\end{equation}

Recall that our goal is to establish \eqref{semiDiscretizedEnergyInJ}. We will do this by bounding the number of quadruples $(x,x',y,y')\in A\times A \times B \times B$ with $(x,x',y)\in \bigcup_{Q\in\mathcal{Q}_2}Q$, and the number of quadruples with $(x,x',y)\in (J_1\times J_2\times J_3)\ \backslash\ \bigcup_{Q\in\mathcal{Q}_2}Q$. By \eqref{decomposeJ1J2J3}, we have $|H(x,x',y,G(x,x',y))|\leq s_0$ for each $(x,x',y)$ of the latter kind, and thus we can use \eqref{coveringNumberBdCloseToW} to bound the contribution to  \eqref{semiDiscretizedEnergyInJ} of quadruples arising from such $(x,x',y)$. We will select

\begin{equation}\label{choiceS0}
s_0 = \delta^{100\eps/\kappa^2}.
\end{equation}
With this choice of $s_0$, the RHS of \eqref{coveringNumberBdCloseToW} is $O(\delta^{-3\alpha-11\eps+100\eps/\kappa})$, and thus in turn is at most $\frac{1}{2}\delta^{-3\alpha+30\eps/\kappa}$. Comparing with \eqref{semiDiscretizedEnergyInJ}, this contribution is acceptable. 

Our next task is to obtain an analagous bound for 
\begin{equation}\label{boundSumOverQ2}
\sum_{Q\in\mathcal{Q}_2} \mathcal{E}_{\delta}\big(\{ (x,x',y,y')\in J\cap (A\times A\times B\times B) \colon P(x,y) = P(x',y'),\ (x,x',y)\in Q\}\big).
\end{equation}

Each cube in $\mathcal{Q}_2$ intersects $J_1\times J_2\times J_3 \subset [0,1]^3$, and each cube has side-length $\gtrsim \delta^{5\eps}s_0$. For technical reasons, it will be convenient to further partition the cubes in $\mathcal{Q}_2$ into a set of $O(\delta^{-48\eps}s_0^{-3})$ smaller rectangular prisms, each of which has diameter at most $\delta^{16\eps}s_0$. Denote this new set of prisms by $\mathcal{Q}_3.$

Fix $Q = Q_1\times Q_2\times Q_3 \in\mathcal{ Q}_3$, so $Q\times {\tilde J}_4 = Q_1\times Q_2\times Q_3\times {\tilde J}_4$. 
Note that $F$ has full projection on $Q\times {\tilde J}_4$. 
We claim that there exists a subinterval $K_Q\subset \tilde{J}_4$ such that $F$ has a full projection on $Q\times K_Q$ and $K_Q$ has length at most $\delta^{14\eps}s_0$. 
Indeed, first note that, by \eqref{lowerBoundPartialsAndK} and \eqref{C3BoundsForP}, $|\nabla G|$ is bounded by $\delta^{-2\eps}$ on $Q\times \tilde{J}_4$.
Thus, for $(x_1,x_1',y_1,y_1'), (x_2,x_2',y_2,y_2')\in (Q\times \tilde{J}_4)\cap Z(F)$, we have 
\begin{equation}\label{distanceGvaluesOnQ}
|y_1'-y_2'|\lesssim \delta^{-2\eps}(\delta^{16\eps}s_0),
\end{equation}
where here we used the property that the diameter of $Q$ is at most $\delta^{16\eps}s_0$. 

Finally, fix $Q = Q_1\times Q_2\times Q_3\in \mathcal{Q}_3$ and let $K_Q$ be as above. By the definition of the collection $\mathcal{Q}_2$, we have 
$
|H_F(x,x',y,G(x,x',y))|\gtrsim \delta^{5\eps}s_0
$ on each point in $Q$. Since $(Q\times K_Q)\cap Z(F)\neq \emptyset$, there exists $(x_1,x_1',y_1,y_1')\in Q\times K_Q$, such that $y_1'=G(x_1,x_1',y_1)$. Recalling that $|\nabla H_F|\le 12 \delta^{-3\eps}$, we conclude that 
$$
|H_F(x,x',y,y')|\ge \delta^{10\eps}s_0-12 \delta^{-3\eps}(\delta^{14\eps}s_0)\ge \tfrac12 \delta^{10\eps}s_0,
$$
for every $(x,x',y,y')\in Q\times K_Q$.

Let $\eps'$ and $\delta_0'$ be the value of $\eps$ and $\delta_0$ coming from Lemma \ref{EnergyDispersionForF} (with the value of $\alpha$ and $\kappa$ specified as input to the present lemma). If $\eps$ is selected sufficiently small so that $10\eps+ 100\eps/\kappa^2\leq \eps'$ (i.e. $\tfrac12\delta^{10\eps}s_0\geq\delta^{\eps'}$), then $F$ satisfies hypotheses \eqref{HFBigOnZAndLowerBoundPartials} and \eqref{C3BoundForF}, and the sets $A\cap Q_1, B\cap Q_2, A\cap Q_3$ and $B\cap K_Q$ satisfy the hypothesis \eqref{noncon4}. By construction, the function $F\colon Q_1\times Q_2 \times Q_3 \times K_Q\to\RR$ has full projection. We conclude that for all $0<\delta<\delta_0$,
\begin{equation}\label{boundForOneJ}
\mathcal{E}_{\delta}\Big(  (A\times A\times B\times B) \cap (Q_1\times Q_2\times Q_3\times {\tilde J}_4) \cap  Z(F)  \Big) \leq  \delta^{-3\alpha+\eps'}.
\end{equation}
Summing over all $O(\delta^{-48\eps}s_0^{-3})$ such cubes, we conclude that 
\begin{equation}\label{sumOverQ2Bd}
\eqref{boundSumOverQ2}\lesssim \delta^{-3\alpha-48\eps+\eps'} s_0^{-3}.
\end{equation}
Finally, using our choice of $s_0$ from \eqref{choiceS0}, we conclude that if $\eps>0$ is selected sufficiently small (depending on $\eps'$, which in turn depends on $\alpha$ and $\kappa$), then $\eqref{boundSumOverQ2}\leq \frac{1}{2}\delta^{-3\alpha+30\eps/\kappa}$. 
We conclude that \eqref{semiDiscretizedEnergyInJ} holds, and we are done. 
\end{proof}

\section{Blaschke curvature and discretized projections} \label{BlaschkeSection}

In this section we will prove Theorem \ref{curvatureProjections}. First, we will show that if a non-concentrated set has small image under three different maps whose gradients are pairwise linearly independent at each point, then the image of this set under these maps must also be non-concentrated. 

\begin{lem}\label{findingLargeNonconcentratedCartesianProduct}
Let $Q\subset[0,1]^2$ be a square, let $\alpha,\eta>0$, $C\geq 1$. Then the following is true for all $\delta>0$ sufficiently small. Let $f_1,f_2,f_3\in C^2(Q)$. Suppose that for each $i=1,2,3$ we have
\begin{equation}\label{boundC2NormFi}
\Vert f_i\Vert_{C^2(Q)}\leq C,
\end{equation}
and for each $i\neq j$ and each $z\in Q$ we have
\begin{equation}\label{boundsWedgeOfP}
|\nabla f_i(z) \wedge \nabla f_j(z)|\geq C^{-1}. 
\end{equation}
Let $X\subset Q$ with with $\mathcal{E}_{\delta}(X) \geq \delta^{-\alpha}$. Suppose that for all balls $B$ of radius $r\geq\delta$ we have 
\begin{equation}\label{nonConcentrationOnBalls}
\mathcal{E}_{\delta}(X \cap B) \leq r^{\alpha}\delta^{-\alpha-\eta}.
\end{equation}
Suppose that for each $i=1,2,3$ we have
\begin{equation}\label{EHasSmallProjections}
\mathcal{E}_{\delta}(f_i(X)) \leq \delta^{-\alpha/2-\eta}.
\end{equation}
Then there exists $X'\subset X$, with 
\begin{equation}\label{ABLargeIntersectionX}
\mathcal{E}_{\delta}(X')\gtrsim (\delta^{\eta}C^{-1})^{O(1)}\mathcal{E}_{\delta}(X),
\end{equation}
such that for all intervals $J$ of length at least $\delta$, and each index $i=1,2,3$, we have
\begin{equation}\label{ABNonConcentrated}
\begin{split}
&\mathcal{E}_{\delta}(f_i(X') \cap J)\leq (C |\log\delta|\delta^{-\eta})^{O(1)} |J|^{\alpha/2}\delta^{-\alpha/2}.
\end{split}
\end{equation}
\end{lem}
\begin{proof}
First, we claim that the set $X$ cannot be too concentrated near the boundary of $Q$, i.e.
\begin{equation}\label{XNearBdryQ}
\mathcal{E}_{\delta}\big(X \cap N_{C\delta}(\operatorname{bdry}(Q))\big)\leq \frac{1}{2}\mathcal{E}_{\delta}(X).
\end{equation}
Indeed, if \eqref{XNearBdryQ} failed, then there is a line segment $\ell\subset[0,1]^2$ (one of the four edges of the square $Q$) with $\mathcal{E}_{\delta}(X \cap N_{C\delta}(\ell))\geq \frac{1}{8}\mathcal{E}_{\delta}(X)$. But then by \eqref{boundC2NormFi} and \eqref{boundsWedgeOfP}, we can find a segment $\ell'\subset\ell$ with $\mathcal{E}_{\delta}(X \cap N_{C\delta}(\ell'))\geq C^{-O(1)}\mathcal{E}_{\delta}(X)$ and an index $i$ so that $|\nabla f\cdot v|\geq C^{-O(1)})$ on $\ell$, where $v$ is the unit vector parallel to $\ell$. But this implies 
\[
\mathcal{E}_{\delta}(f_i(X \cap N_{C\delta}(\ell')))\geq C^{-O(1)}\mathcal{E}_{\delta}(X),
\]
which contradicts \eqref{EHasSmallProjections}, provided $\delta>0$ is sufficiently small (depending on $\alpha, \eta$, and $C$).

Next, by \eqref{boundC2NormFi} and \eqref{boundsWedgeOfP}, we can select a square $Q'\subset Q\backslash N_{C\delta}(\operatorname{bdry}(Q))$  with 
\[
\mathcal{E}_{\delta}(X\cap Q')\geq C^{-O(1)}\mathcal{E}_{\delta}(X)
\]
so that for each $i\neq j$, $f_i\otimes f_j\colon Q'\to \RR^2$ is injective on $N_{C\delta}(Q')\subset Q$, where $N_{C\delta}(Q')$ is the $C\delta$-neighborhood of $Q'$. In particular, since \eqref{boundC2NormFi} and \eqref{boundsWedgeOfP} continue to hold on $Q'$, if $z, z'\in Q'$, then
\begin{equation}\label{fifjComparableZiZj}
|(f_i(z),f_j(z))-(f_i(z'),f_j(z'))|\geq C^{-O(1)}|z-z'|.
\end{equation}

Next, let $X_0\subset X\cap Q'$ be a maximal $C^{O(1)}\delta$ separated subset of $X\cap Q'$, where the implicit constant is chosen so that if $z,z'\in X_0$, then for each $i\neq j$ we have
\begin{equation}\label{quantitativeDiffeoDelta}
|(f_i(z),f_j(z))-(f_i(z'),f_j(z'))|\geq 2\delta. 
\end{equation}
Note that $\#X_0 \geq C^{-O(1)}\mathcal{E}_{\delta}(X),$ so
\begin{equation}\label{lowerBdEntropyX0}
\mathcal{E}_{\delta}(X_0) \gtrsim (\delta^{\eta}C^{-1})^{O(1)}\delta^{-\alpha}. 
\end{equation}
In the argument that follows, we will show that if $X_0$ is a set satisfying \eqref{nonConcentrationOnBalls}, \eqref{EHasSmallProjections}, \eqref{fifjComparableZiZj}, \eqref{quantitativeDiffeoDelta}, and \eqref{lowerBdEntropyX0}, then there is a set $X_1\subset X_0$ that continues to satisfy these inequalities (with slightly worse implicit constants for \eqref{lowerBdEntropyX0}), and that also satisfies \eqref{ABNonConcentrated} for $i=1$. We can then repeat this argument to find a set $X_2\subset X_1$ that continues to satisfy the above inequalities, and also satisfies \eqref{ABNonConcentrated} for $i=2$, and then a set $X_3\subset X_2$ that satisfies  \eqref{ABNonConcentrated} for $i=3$. To complete the proof of the lemma, we define $X'=X_3$. 

Our next goal is to construct the set $X_1$ described above. Let $A_0\subset f_1(X_0)$ and $B_0\subset f_2(X_0)$ be maximal $\delta$-separated sets. Define $E_0\subset A_0\times B_0$, with $(x,y)\in E_0$ if there exists a point $z\in X_0$ with $|x-f_1(z)|\leq\delta,\ |y-f_2(z)|\leq\delta$. Let $\pi_x,\pi_y$ be the projections from $E_0$ to $A_0$ and $B_0$, respectively. By \eqref{EHasSmallProjections} we have $\#A_0 \leq \delta^{-\alpha/2-\eta}$ and $\#B_0 \leq \delta^{-\alpha/2-\eta}$, and by construction we have $\#E_0 = \#X_0\geq (\delta^{\eta}C^{-1})^{O(1)}\delta^{-\alpha}$. 

Using a standard popularity argument (see e.g. \cite[Lemma 2.8]{DG}), we can select sets $A\subset A_0$, $B\subset B_0$ and an edge set $E\subset E_0$ so that $\#E\geq (\#E_0)/2$; each element of $A$ is adjacent to at least $(\#E_0)/(4 \# A_0)\gtrsim (\delta^{\eta}C^{-1})^{O(1)}\delta^{-\alpha/2}$ and at most $\#B=\delta^{-\alpha/2-\eta}$ elements of $B$. Similarly, each element of $B$ is adjacent to at least $\gtrsim (\delta^{\eta}C^{-1})^{O(1)}\delta^{-\alpha/2}$ and at most $\delta^{-\alpha/2-\eta}$ elements of $A$.

Define $X_1=X_0\cap N_{\delta}( (f_1\otimes f_2)^{-1}(E)).$ We claim that $X_1$ satisfies \eqref{ABNonConcentrated} for $i=1$. To verify this, let $J\subset\RR$ be an interval of length $t\geq\delta$ and suppose
\[
\mathcal{E}_{\delta}(f_1(X_1) \cap J) = K t^{\alpha/2} \delta^{-\alpha/2},
\]
for some $K\geq 1$. Our goal is to show that $K\leq (C |\log\delta|\delta^{-\eta})^{O(1)}$. 

By construction we have $\#(A\cap J) = \mathcal{E}_{\delta}(f_1(X_1) \cap J) = K t^{\alpha/2} \delta^{-\alpha/2}$. For each $z=(x,y)\in E$, define 
\[ 
m(z)= \#\pi_y^{-1}(\pi_y(z)) = \#\{ x'\in A\colon (x',y)\in E\}.
\]
We have $1\leq m(z)\leq \#(A\cap J)\leq\delta^{-1}$.

By dyadic pigeonholing, there is a number $1\leq m \leq \#(A\cap J)$ and a set $E'\subset E \cap f_1^{-1}(J))$ with 
\[
\#E' \geq |\log\delta|^{-1} \#(E \cap f_1^{-1}(J))\gtrsim \big( \#(A\cap J)\big)\big((\delta^{\eta}C^{-1})^{O(1)}\delta^{-\alpha/2}\big) =(\delta^{\eta}C^{-1}|\log\delta|^{-1})^{O(1)} K t^{\alpha/2}\delta^{-\alpha},
\]
so that $m(z)\sim m$ for all $z\in E'$. Note that if $z=(x,y)\in E'$ and $x'\in A$, then $(x',y)\in E'$.

We have $\pi_y(E')\subset B$, and thus $\#(\pi_y(E'))\leq \#B\leq \delta^{-\alpha/2-\eta}$. On the other hand $(\#E')/m \leq \#(\pi_y(E'))$, so by \eqref{EHasSmallProjections} we have
\[
m \gtrsim \frac{(\#E')}{\#(\pi_y(E')} \gtrsim 
\frac{(\delta^{\eta}C^{-1}|\log\delta|^{-1})^{O(1)} K t^{\alpha/2}\delta^{-\alpha}}{\delta^{-\alpha/2-\eta}}= (\delta^{\eta}C^{-1}|\log\delta|^{-1})^{O(1)} Kt^{\alpha/2}\delta^{-\alpha/2}.
\]

By \eqref{nonConcentrationOnBalls} we have that for each $T\geq 1$,
\[
\mathcal{E}_{C^Tt}(E') \geq \frac{\# E'}{(C^{T}t)^{\alpha}\delta^{-\alpha-\eta}}\gtrsim C^{-T\alpha} (\delta^{\eta}C^{-1}|\log\delta|^{-1})^{O(1)} K t^{-\alpha/2}.
\]

Thus we can select a $C^Tt$-separated set $B'\subset B$ with 
\[
\#B' \gtrsim C^{-T\alpha}(\delta^{\eta}C^{-1}|\log\delta|^{-1})^{O(1)} K t^{-\alpha/2},
\]
so that for each $b\in B'$, there is a set $E'_b\subset E'$ with $\#E'_b\sim m$, and  $\pi_y(z) = b$ for each $z\in E'_b$. Let $X_b = (f_1\otimes f_2)^{-1}(E'_b)$. Since $f_1(X_b)$ is contained in an interval of length $t$ and since $f_2(X_b)$ is contained in an interval of length $\delta$ for each $z\in E'_b$, by \eqref{fifjComparableZiZj} (with $i=1,j=2$) we have that $X_b$ is contained in a ball of radius  $C^{O(1)}t$. Thus by \eqref{boundC2NormFi}, $f_3(X_b)$ is contained in an interval of length $C^{O(1)}t$. By \eqref{fifjComparableZiZj} (with $i=2,j=3)$, we have that $\mathcal{E}_{\delta}(f_3(X_b))\gtrsim C^{-O(1)}m$. Next, note that for all $b\in B'$, $f_1(X_b)\subset J$, which is an interval of length $t$. Thus if the constant $T=O(1)$ is chosen sufficiently large, then by \eqref{fifjComparableZiZj} (with $i=1,j=3)$, we have that if $b,b'\in B'$ are distinct, then the sets $f_3(X_b)$ and $f_3(X_b)$ are contained in disjoint intervals (of length $C^{O(1)}t$ ). 

Thus 
\[
\mathcal{E}_{\delta}(f_3( (f_1\circ f_2)^{-1}(E')) \geq \sum_{b\in B'} \mathcal{E}_{\delta}(f_3(E'_b)) \geq C^{-O(1)}  (\#B') m.
\]

We conclude that
\begin{equation}
\begin{split}
\mathcal{E}_{\delta}(f_3(X_1)) & \geq C^{-O(1)} (\#B') m\\
& \gtrsim (\delta^{\eta}C^{-1}|\log\delta|^{-1})^{O(1)} \big(K t^{-\alpha/2}\big)\big(Kt^{\alpha/2}\delta^{-\alpha/2}\big)\\
& = K^2  (\delta^{\eta}C^{-1}|\log\delta|^{-1})^{O(1)} \delta^{-\alpha/2}.
\end{split}
\end{equation}
Comparing this with \eqref{EHasSmallProjections} for $i=3$, we conclude that
\[
K \leq (C|\log\delta|   \delta^{-\eta})^{O(1)}.
\]
Thus the set $X_1$ satisfies \eqref{ABNonConcentrated} for $i=1$, which completes the proof. 
\end{proof}

Next, we will show that if a non-concentrated set is contained in the small neighborhood of a curve, and if two functions have linearly independent gradients on this curve, and these gradients are not normal to the curve, then at least one of the functions must have large image on the set.
\begin{lem}\label{twoFunctionsOnASet}
Let $0<\delta<s$, and let $\eta>0$. Let $\gamma\subset[0,1]^2$ be a simple (no self-intersections) smooth curve. Let $f_1,f_2\colon N_{s}(\gamma)\to\RR$ be functions, with
\begin{equation}\label{upperBoundC1NormFi}
\Vert f_i\Vert_{C^1(N_s(\gamma))}\leq\delta^{-\eta},
\end{equation}

\begin{equation}\label{f1f2LargeWedge}
|\nabla f_1(x)\wedge \nabla f_2(x)|\geq\delta^{\eta},\quad x\in\gamma,
\end{equation}

\begin{equation}\label{FiCdotVx}
|\nabla f_i(x)\cdot v(x)|\geq\delta^{\eta}, \quad x \in\gamma,
\end{equation}
where $v(x)$ is the unit vector tangent to $\gamma$ at $x$.

Let $A\subset N_{s}(\gamma)$ be a set, with $\mathcal{E}_{\delta}(A)\geq\delta^{-\alpha+\eta}$, and suppose that for all balls $B$ of radius $r\geq\delta$, $A$ satisfies the non-concentration condition
\begin{equation}\label{nonConcTwoFuncLem}
\mathcal{E}_{\delta}(A\cap B)\leq r^\alpha\delta^{-\alpha-\eta}.
\end{equation}
Then for at least one index $i$, we have
\begin{equation}\label{conclusionTwoFunctionsOnASet}
\mathcal{E}_{\delta}(f_i(A))\gtrsim s^{-\alpha/2} \delta^{-\alpha/2+5\eta}|\log\delta|^{-1}.
\end{equation}
\end{lem} 
\begin{proof}
Cover $N_s(\gamma)$ by finitely overlapping balls of radius $O(s)$. After dyadic pigeonholing, we can select a set of $N$ balls, each of which satisfy $\mathcal{E}_{\delta}(A\cap B) \gtrsim |\log\delta|^{-1}\delta^{-\alpha+\eta}N^{-1}$. Denote this set of balls by $\mathcal{B}$. By \eqref{nonConcTwoFuncLem}, we have
\begin{equation}\label{NIsLarge}
N\gtrsim |\log\delta|^{-1}\delta^{2\eta}s^{-\alpha}. 
\end{equation} 
By \eqref{upperBoundC1NormFi}, for each ball $B\in\mathcal{B}$ and each index $i$, we have that $f_i(B)$ is contained in an interval of length $O(s\delta^{-\eta})$. Let $\gamma(t)$ be a unit speed parameterization of $\gamma$. By \eqref{FiCdotVx}, if $|t-t'|\geq Cs\delta^{-2\eta}$ for a sufficiently large constant $C=O(1)$, then if $B,B'\in\mathcal{B}$ are balls containing $\gamma(t)$ and $\gamma(t')$ respectively, then the intervals $f_i(B)$ and $f_i(B')$ are disjoint. Fix such a choice of $C$. After refining our collection of balls by a multiplicative factor of $O(\delta^{2\eta})$, we can suppose that for each pair of distinct balls $B$ and $B'$ and each index $i$, the intervals $f_i(B)$ and $f_i(B')$ are disjoint. Let $\mathcal{B}'\subset\mathcal{B}$ be such a refinement. 

By \eqref{f1f2LargeWedge}, for each ball $B\in\mathcal{B}$ there exists an index $i$ so that
\[
\mathcal{E}_{\delta}(f_i(A\cap B))\gtrsim \delta^{\eta}\mathcal{E}_{\delta}(A\cap B)^{1/2}\gtrsim |\log\delta|^{-1/2}\delta^{-\alpha/2+(3/2)\eta}N^{-1/2}
\]
Choose an index $i$ so that the above inequality holds for at least half of the balls $B\in\mathcal{B}'$; denote this set of balls $\mathcal{B}''$.  We have
\begin{equation*}
\begin{split}
\mathcal{E}_{\delta}(f_i(A))&\geq \sum_{B\in\mathcal{B}''}\mathcal{E}_{\delta}(f_i(A\cap B))\\
&\gtrsim (\delta^{2\eta}N)|\log\delta|^{-1/2}\delta^{-\alpha/2+(3/2)\eta}N^{-1/2}\\
&\gtrsim N^{1/2} \delta^{-\alpha/2+(7/2)\eta}|\log\delta|^{-1/2}\\
&\gtrsim s^{-\alpha/2} \delta^{-\alpha/2+5\eta}|\log\delta|^{-1}. \qedhere
\end{split}
\end{equation*}
\end{proof}

With these lemmas, we are now ready to begin the proof of Theorem \ref{curvatureProjections}. We will actually state and prove the following slightly more general (and technical) version of the theorem.

\begin{curvatureProjectionsThm}
For each $0<\alpha<2$, there exists $\eta=\eta(\alpha)>0$ and $\delta_0=\delta_0(\alpha)>0$ so that the following is true for all $0<\delta<\delta_0$. Let $Q\subset [0,1]^2$ be a square, and let $\phi_1,\phi_2,\phi_3\colon Q\to\RR$ be smooth, with 
\begin{equation}\label{boundOnC3Phi}
\Vert \phi_i\Vert_{C^2(Q)}\leq\delta^{-\eta}.
\end{equation}
Suppose that for each $i\neq j$, we have
\begin{equation}\label{quantitativeLinIndependenceGradient}
|\nabla \phi_i\wedge \nabla \phi_j|\geq\delta^{\eta}\quad\textrm{on}\ Q,
\end{equation}
and 
\begin{equation}\label{quantitativeCurvature}
\Big| \frac{\partial}{\partial \phi_1}\frac{\partial}{\partial \phi_2} \log \frac{\partial \phi_3 / \partial\phi_1}{\partial \phi_3 / \partial\phi_2}\Big|\geq\delta^{\eta}\quad\textrm{on}\ Q.
\end{equation}

Let $X\subset Q$ be a set with $\mathcal{E}_{\delta}(X)\geq \delta^{-\alpha}$, and suppose that for all balls $B$ of diameter $r\geq\delta$, $X$ satisfies the non-concentration condition
\begin{equation}\label{nonConcentrationCondXCoveringNumber}
\begin{split}
\mathcal{E}_{\delta}(X \cap B) &\le  r^\alpha\delta^{-\alpha-\eta}.
\end{split}
\end{equation}
Then for at least one index $i\in \{1,2,3\}$ we have
\begin{equation}\label{atLeastOneBigProjectionCoveringNumber}
\mathcal{E}_{\delta}(\phi_i(X))\geq \delta^{-\alpha/2-\eta}.
\end{equation}
\end{curvatureProjectionsThm}
\begin{rem}
To see that Theorem \ref{curvatureProjections}$^*$ implies Theorem \ref{curvatureProjections}, cover $K\subset U$ by open squares, whose closures are contained in $U$. Since $K$ is compact, there is a square $S\subset U$ with $\mathcal{E}_{\delta}(X\cap S)\gtrsim\delta^{-\alpha}$. Since $\phi_1,\phi_2,\phi_3$ have gradients that are pairwise linearly independent, and since the Blaschke curvature form is non-vanishing on $U$, we have that \eqref{quantitativeLinIndependenceGradient} and \eqref{quantitativeCurvature} holds for all $\delta>0$ sufficiently small. 
\end{rem}
\begin{proof}
Suppose that \eqref{atLeastOneBigProjectionCoveringNumber} is false for $i=1,2,3$. If $\eta=\eta(\alpha)>0$ and $\delta_0=\delta_0(\alpha)>0$ are selected sufficiently small, we will obtain a contradiction.

First, we will select a sub-square $Q'\subset Q$ with $\mathcal{E}_{\delta}(X\cap Q')\gtrsim \delta^{-\alpha+O(\eta)}$ so that for each $i\neq j$, $\phi_i\otimes\phi_j$ is injective on $Q'$.  Next, we apply Lemma \ref{findingLargeNonconcentratedCartesianProduct} to $X\cap Q'$ and the functions $\phi_1,\phi_2,\phi_3$. We obtain a set $X'\subset Q'$ with $\mathcal{E}_{\delta}(X')\gtrsim \delta^{-\alpha+O(\eta)}$ so that for all intervals $J$ of length at least $\delta$, and for each index $i$, we have
\begin{equation}
\mathcal{E}_{\delta}(\phi_i(X'\cap J))\lesssim |J|^{\alpha/2}\delta^{-\alpha/2+O(\eta)}.
\end{equation}

The boundary of $Q'$ is a union of four line segments. We will first prove the following claim:
{\bf Claim 1:} If the constant $C=C(\alpha)$ is chosen sufficiently large, then for each of these line segments $L$, we have 
\[
\mathcal{E}_{\delta}(N_{\delta^{C\eta}}(L)\cap X') \leq \frac{1}{8}\mathcal{E}_{\delta}(X').
\]
Suppose not. Cover $L$ by finitely overlapping balls of radius $\delta^{100\eta}$. By pigeonholing, there must exist a ball $B$ with
\[
\mathcal{E}_{\delta}(N_{\delta^{C\eta}}(L)\cap B \cap X') \geq\frac{1}{8} \delta^{100\eta} \mathcal{E}_{\delta}(X')\gtrsim \delta^{-\alpha+O(\eta)}.
\]
Let $v$ be the unit vector parallel to $L$. By \eqref{quantitativeLinIndependenceGradient}, there exist indices $i\neq j$ so that $|v\cdot \nabla \phi_i(x_0)|\gtrsim\delta^\eta$ and $|v\cdot \nabla \phi_j(x_0)|\gtrsim\delta^\eta$, where $x_0$ is the center of the ball $B$. But by \eqref{boundOnC3Phi} and the fact that $B$ has radius $\delta^{100\eta}$, we have that $|v\cdot \nabla \phi_i(x_0)|\gtrsim\delta^\eta$ and $|v\cdot \nabla \phi_j(x_0)|\gtrsim\delta^\eta$ for all $x\in B$. We now apply Lemma~\ref{twoFunctionsOnASet} to conclude that (after interchanging the indices $i$ and $j$ if necessary)
\[
\mathcal{E}_{\delta}(\phi_i(N_{\delta^{C\eta}}(L)\cap B \cap X'))\gtrsim \delta^{-C\eta\alpha/2}\delta^{-\alpha/2+O(\eta)}|\log\delta|^{-1}.
\]
If $C=C(\alpha)$ is selected sufficiently large, then \eqref{atLeastOneBigProjectionCoveringNumber} holds for $\phi_i$ and we are done. This concludes the proof of Claim 1.

Next, let $\psi=\phi_1\otimes\phi_2$ and let $V=\psi(Q')$. By \eqref{boundOnC3Phi} and \eqref{quantitativeLinIndependenceGradient}, we have that $|D\psi|=O(\delta^{-\eta})$, and $|\det D\psi|\gtrsim\delta^{O(\eta)}$. This implies that the $\delta^{C\eta}$-neighborhood of the boundary of $Q'$ is mapped to the $\delta^{C+O(\eta)}$ neighborhood of the boundary of $V$. In particular, by Claim 1 we can select $C'=C'(\alpha)$ so that if $Y \subset\psi(X')$ denotes the subset of $\psi(X')$ that has distance at least $\delta^{C'\eta}$ from the boundary of $V$, then $\mathcal{E}_{\delta}(Y)\gtrsim\delta^{-\alpha+O(\eta)}$. To simplify notion, in what follows we will use $O_\alpha(\eta)$ to denote a quantity that is bounded by $C_0\eta$, where $C_0=C_0(\alpha)$ may depend on $\alpha$. 

Next we will introduce the Whitney cube decomposition. The version stated here can be found in \cite[Appendix J]{Gra}.
\begin{thm}[Whitney]\label{whitneyCubeDecomp}
Let $\Omega$ be an open non-empty proper subset of $\RR^d$. Then there exists a countable family $\mathcal{Q}$ of closed cubes such that
\begin{itemize}
\item $\bigcup_{Q\in\mathcal{Q}} Q = \Omega$, and the cubes in $\mathcal{Q}$ have disjoint interiors.
\item For each cube $Q\in\mathcal{Q}$, we have $2Q\not\subset\Omega$, where $2Q$ denotes the 2-fold dilate of $Q$. 
\end{itemize}
\end{thm}
Apply Theorem \ref{whitneyCubeDecomp} to $V$. While the set of cubes $\mathcal{Q}$ might be infinite, the second item from Theorem \ref{whitneyCubeDecomp} implies that $\delta^{O_\alpha(\eta)}$ cubes from $\mathcal{Q}$ intersect $Y$. Hence we can find a cube $Q_0\subset V$ of with $\mathcal{E}_{\delta}(Q_0\cap Y)\gtrsim\delta^{-\alpha+O_\alpha(\eta)}$. This cube has side-length at most $\delta^{-\eta}$. Cutting this cube into sub-cubes if necessary, we can find a cube $Q_1\subset Q_0$ with side-length at most one, with $\mathcal{E}_{\delta}(Q_1\cap Y)\gtrsim\delta^{-\alpha+O_\alpha(\eta)}.$ Let $Y'=(Y\cap Q_1)-(x_1,y_1),$ where $(x_1,y_1)$ is the bottom left corner of $Q_1$. Let $A = \pi_x(Y')$ and $B = \pi_y(Y')$. Then $A,B\subset[0,1]$. Since $X'$ satisfies the conclusion of Lemma \ref{findingLargeNonconcentratedCartesianProduct}, for all intervals $J$ of length at least $\delta$, we have the non-concentration estimates
\begin{equation*}
\begin{split}
\mathcal{E}_{\delta}(A\cap J)&\leq |J|^{\alpha/2}\delta^{-\alpha/2+O(\eta)},\\
\mathcal{E}_{\delta}(B\cap J)&\leq |J|^{\alpha/2}\delta^{-\alpha/2+O(\eta)}.
\end{split}
\end{equation*}
Define $P = \phi_3(\psi^{-1})$. We have that 
\begin{equation}\label{PC2NormOnQ}
\Vert P\Vert_{C^2(Q_1)}\lesssim \delta^{-O(\eta)},
\end{equation}
and
\begin{equation}\label{boundsOnPartialPxPy}
|\partial_x P|\gtrsim\delta^{O(\eta)},\quad |\partial_y P|\gtrsim\delta^{O(\eta)}\quad\textrm{on}\ Q_1.
\end{equation}
Since the Blaschke curvature form is invariant under change of coordinates, by \eqref{quantitativeLinIndependenceGradient} and \eqref{quantitativeCurvature}, we have  
\[
\bigg|\partial_{xy}\bigg(\log\Big(\frac{\partial_xP}{\partial_yP}\Big)\bigg)\bigg| \gtrsim\delta^{O(\eta)}\quad\textrm{on}\ Q_1.
\]
By \eqref{relateKpLog}, this implies
\[
|\partial_{xy}P|\gtrsim\delta^{O(\eta)}\quad\textrm{on}\ Q_1.
\]
Define $A'\times B'=A\times B\cap Q_1$. The function $P$ and the sets $A',B', Q_1\cap Y'$ satisfy the hypotheses of Proposition \ref{countingNumberQuadruplesNablePWedgeP} (with $\kappa=\alpha$). We conclude that if $\eta=\eta(\alpha)>0$ is chosen sufficiently small, then there exists $\eps=\eps(\alpha)>0$ so that
\[
\mathcal{E}_{\delta}\big(\{ (x,x',y,y')\in A'\times A'\times B'\times B' \colon P(x,y) = P(x',y')\}\big)\leq \delta^{-(3/2)\alpha+\eps}.
\]
By Lemma \ref{CSLem}, this implies
\[
\mathcal{E}_{\delta}(P(Q_2\cap Y'))\gtrsim \frac{ \delta^{-2\alpha+O(\eta)}}{\delta^{-(3/2)\alpha+\eps}}=\delta^{-\alpha/2-\eps+O(\eta)},
\]
and thus 
\[
\mathcal{E}_{\delta}(\phi_3(X))\gtrsim \delta^{-\alpha/2-\eps+O(\eta)}. 
\]
If $\eta=\eta(\alpha)>0$ is sufficiently small, then we have $\mathcal{E}_{\delta}(\phi_3(X))\geq \delta^{-\alpha/2-\eta},$ as desired.  
\end{proof}

\section{Pinned distinct distances: a single-scale estimate}\label{pinnedDistancesSection}
In this section we will prove Theorem \ref{pinnedDistanceCor}. To begin, we will perform some computations to calculate the Blaschke curvature for the (squared) distance functions from three points. 
\subsection{Blaschke curvature estimates for pinned distances}
In this section we will compute the Blaschke curvature for the pinned distance functions from three points. Define
\[
p_1=(t,0),\quad p_2=(-t,0),\quad p_3=(a,b).
\]

Define
\begin{equation}
\begin{split}
\phi_1(x,y) &= (x-t)^2 + y^2,\\
\phi_2(x,y) &= (x+t)^2 + y^2,\\
\phi_3(x,y) &= (x-a)^2 + (y-b)^2.
\end{split}
\end{equation}

We have 
\begin{equation}
\begin{split}
x &= \frac{\phi_2-\phi_1}{4t},\\
y &=  \Big( \frac{\phi_1}{2} + \frac{\phi_2}{2} - t^2 - \frac{(\phi_2 -\phi_1)^2}{16t^2}\Big)^{1/2},
\end{split}
\end{equation}
and thus
\[
\phi_3 = \frac{1}{16}\Big( \Big(\frac{4 at + \phi_1 - \phi_2}{t}\Big)^2
+ \Big(-4b +   \sqrt{8\phi_1 + 8 \phi_2 - (\phi_1-\phi_2)^2/t^2 -16t^2}\ \ \Big)^2\Big).
\]
We therefore have
\begin{equation}
\begin{split}
\frac{\partial\phi_3}{\partial\phi_1}&=\frac{ - b x + (t+a)y  -bt}{2 ty},\\
\frac{\partial\phi_3}{\partial\phi_2}&=\frac{  \phantom{-}b x + (t-a)y-bt}{2 ty},
\end{split}
\end{equation}
i.e. $\frac{\partial\phi_3}{\partial\phi_1}$ vanishes on the line passing through $p_1=(-t,0)$ and $p_3=(a,b)$, and $\frac{\partial\phi_3}{\partial\phi_2}$ vanishes on the line passing through $p_2=(t,0)$ and $p_3=(a,b)$. Both functions are finite off the line passing through $p_1=(-t,0)$ and $p_2(t,0)$. 

Next, we compute
\[
d\phi_1\wedge d\phi_2 = (-8ty) dxdy,
\] 
and finally, we compute
\begin{equation}\label{computeCurvature}
\Big(\frac{\partial}{\partial \phi_1}\frac{\partial}{\partial \phi_2} \log \frac{\partial \phi_3 / \partial\phi_1}{\partial \phi_3 / \partial\phi_2}\Big)\ d\phi_1\wedge d\phi_2
=b\frac{f(x,y)}{g^2(x,y)}dxdy,
\end{equation}
where 
\begin{equation}\label{defnOfF}
\begin{split}
f(x,y) = 
&\big(a b^2 t^4 - b^2 t^4 x - 2 a b^2 t^2 x^2 + 2 b^2 t^2 x^3 + a b^2 x^4 - b^2 x^5\big)\\
&+\big(-4 a b t^4 + 4 a^2 b t^2 x + 4 b t^4 x - 4 a^2 b x^3 -  4 b t^2 x^3 + 4 a b x^4 \big)y\\
&+\big(-3 a^3 t^2 - a b^2 t^2 + 3 a t^4 + 3 a^2 t^2 x +  3 b^2 t^2 x - 3 t^4 x + 3 a^3 x^2 \\
&\qquad\qquad\qquad\qquad\qquad\qquad- 3 a b^2 x^2 -  3 a t^2 x^2 - 3 a^2 x^3 + b^2 x^3  + 3 t^2 x^3 \big)y^2\\
&+\big(2 a b t^2 + 2 a^2 b x  - 2 b^3 x  - 6 b t^2 x + 2 a b x^2 \big)y^3\\
&+\big(a^3 + 2 a b^2  - a t^2  - 3 a^2 x  + 2 b^2 x +  3 t^2 x \big)y^4\\
&
+\big(- 2 a b\big)y^5,
\end{split}
\end{equation}
and 
\[
g(x,y) = y \big(- b x + (t+a)y  -bt\big) \big(b x + (t-a)y-bt\big).
\]
We see that $|g|=O(1)$ on $[0,1]^2$, and $g$ vanishes on the three lines spanned by the points $(-t,0),\ (t,0)$, and $(a,b)$. We see that $f$ is a polynomial of degree at most 5 in the variables $x$ and $y$, and the coefficient of $x^5$ is $-b^2$.

\subsection{Pinned distances from three points}
In this section we will prove Theorem \ref{pinnedDistanceCor}. Suppose the proposition fails for some $c>0$. If $c$ is sufficiently small (depending on $\alpha$), we will obtain a contradiction. Applying a rotation, translation, and reflection, we can suppose that $p_1=(t,0)$, $p_2=(-t,0)$ and $p_3=(a,b)$, with $t\gtrsim \delta^{c}$ and $b\gtrsim\delta^{c}$. 

Define $\phi_i(z)=|p_i-z|^2$. $|\nabla\phi_i|$ is bounded away from 0 on $[0,1]^2$ except in a small neighborhood of $p_i$. $\phi_i$ and $\phi_j$ have linearly independent gradients except on the line $\ell_{ij}$ that contains $p_i$ and $p_j$. If $v_{ij}$ is a unit vector parallel to $\ell_{ij}$ and if $k$ is the remaining index, then $\nabla\phi_k(z)\cdot v_{ij}$ vanishes at one point on $\ell_{ij}$, which we will denote by $q_k$.

Our next task is to find a connected open region $U$ so that $\mathcal{E}_{\delta}(X \cap U)\geq\frac{1}{2}\delta^{-\alpha+c}$, and $U$ is far from the points $p_i$ and $q_i$, and the lines $\ell_{ij}$. The set $U$ will be $[0,1]^2$, with certain neighborhoods of $p_i,q_i$ and $\ell_{ij}$ removed. To begin, define
\[
U_0=[0,1]^2\backslash \Big( \bigcup_{i} N_{C^{-1} \delta^{2c/\alpha}}(p_i)\ \cup\ \bigcup_{i} N_{C^{-1} \delta^{2c/\alpha}}(q_i)\Big).
\]
If $C=O(1)$ is selected sufficiently large, then by \eqref{nonConcentrationCondDist}, we have $\mathcal{E}_{\delta}(X\cap U_0)\geq\frac{3}{4}\delta^{-\alpha+c}$. We have
\[
\delta^{2c/\alpha}\lesssim |\nabla\phi_i|\lesssim 1\quad\textrm{on}\ U_0.
\]
Furthermore, since $B(q_i,\ C^{-1}\delta^{2c/\alpha})$ is disjoint from $U_0$ for each index $i$, we have
\begin{equation}
\begin{split}
&|\nabla\phi_1(x)\wedge\nabla\phi_2(x)|,\ |\nabla\phi_2(x)\wedge\nabla\phi_3(x)|\gtrsim \delta^{O_\alpha(c)}\quad\textrm{for all}\ x\in \ell_{p_1,p_3}\cap U_0, \\
&|\nabla\phi_1(x)\wedge\nabla\phi_3(x)|,\ |\nabla\phi_2(x)\wedge\nabla\phi_3(x)|\gtrsim \delta^{O_\alpha(c)}\quad\textrm{for all}\ x\in \ell_{p_1,p_2}\cap U_0, \\
&|\nabla\phi_1(x)\wedge\nabla\phi_2(x)|,\ |\nabla\phi_1(x)\wedge\nabla\phi_3(x)|\gtrsim \delta^{O_\alpha(c)}\quad\textrm{for all}\ x\in \ell_{p_2,p_3}\cap U_0.
\end{split}
\end{equation}
Applying Lemma~\ref{twoFunctionsOnASet} to $\phi_1$ and $\phi_2$ with $s=\delta^{cC}$ and $\eta \sim \delta^{O_\alpha(c)}$, then either 
\begin{equation}\label{PhiIXBig}
\mathcal{E}_{\delta}(\phi_i\big(X\cap U_0 \cap N_{\delta^{C c}}(\ell_{p_1,p_2}))\big)\gtrsim \delta^{O_\alpha(c)-\alpha cC/2-\alpha/2}|\log\delta|^{-1}
\end{equation}
for at least one index $i\in\{1,2\}$, or
\begin{equation}\label{XSmallOnLine}
\mathcal{E}_{\delta}(X \cap U_0 \cap N_{\delta^{C c}}(\ell_{p_1,p_2}))\leq \frac{1}{12}\delta^{-\alpha+c}.
\end{equation}
If $C = C(\alpha)$ is chosen sufficiently large and $\delta>0$ is sufficiently small, then \eqref{PhiIXBig} implies \eqref{atLeastOneBigPinnedDistance} and we are done. Thus we may suppose that \eqref{XSmallOnLine} holds. We argue similarly for $\ell_{p_2,p_3}$ and $\ell_{p_1,p_3}$. Define 
\[
U_1=U_0\backslash\{ N_{\delta^{C c}}(\ell_{p_1,p_2})\ \cup\ N_{\delta^{C c}}(\ell_{p_2,p_3})\ \cup\ N_{\delta^{C c}}(\ell_{p_1,p_3})\}.
\]
Then for each pair of indices $i\neq j$, we have $|\nabla\phi_i(x)\wedge\nabla\phi_j(x)|\gtrsim\delta^{O_\alpha(c)}$ for all $x\in U_1.$ In fact, the above inequality holds on a slightly larger set; if $C_1=C_1(\alpha)$ is selected sufficiently large, then 
\begin{equation}\label{largeWedge}
|\nabla\phi_i(x)\wedge\nabla\phi_j(x)|\gtrsim\delta^{O_\alpha(c)}\quad\textrm{for all}\ x\in N_{\delta^{C_1c}}(U_1).
\end{equation}

Our next task is to bound the size of the subset of $X\cap U_1$ where the function $f$ from \eqref{defnOfF} is small. Let $s>0$ be a small parameter. Since $f$ is a polynomial of degree at most 5, which contains the monomial $-b^2x^5$, and since $b\gtrsim\delta^{c}$, there exists an absolute constant $C_0$ so that
\[
\sup_{(x,y)\in [0,1]^2}|f(x,y)|\geq C_0^{-1} \delta^{C_0 c}.
\]

Next, we will introduce the following Remez-type inequality. See \cite[Theorem 2]{BG3}
\begin{thm}[Remez]\label{remezInez}
Let $P\colon\RR^d\to\RR$ be a polynomial of degree at most $D$. Let $\Omega\subset\RR^d$ be an open convex set. Let $m=\sup_{z\in \Omega}|P(z)|$. Then for all $0<\lambda<1$, 
\[
|\{z\in \Omega \colon |P(z)|\leq\lambda m\}|\leq 4d|\Omega|\lambda^{1/D}.
\]
\end{thm}

Applying Theorem \ref{remezInez}, we have
\begin{equation}\label{upperBoundAreaPleqS}
|\{(x,y)\in [0,1]^2\colon |f(x,y)|< s \}|\leq 8(C_0 \delta^{-C_0c} s)^{1/5}.
\end{equation}
We will specify the choice of $s=\delta^{O(c)}$ below. The set $S=\{(x,y)\in [0,1]^2\colon |f(x,y)|< s\}$ is relatively open in $[0,1]^2$. The boundary of this set is contained in 
\[
Z=Z(x)\cup Z(y)\cup Z(x-1)\cup Z(y-1)\cup Z(f-s)\cup Z(f+s).
\]
In particular, if 
\begin{equation}\label{defnW}
w=\big(8(C_0 \delta^{-C_0c} s)^{1/5}\big)^{1/2},
\end{equation}
then 
\begin{equation}\label{SContainedInNwZ}
S\subset N_{w}(Z).
\end{equation}
To verify \eqref{SContainedInNwZ}, note that if $z\in S \backslash  N_{w}(Z)$, then $\operatorname{dist}(z,Z)\geq w$, and hence $B(z,w)\subset S$, and thus $|S|\geq|B(z,w)|\geq \pi w^2,$ which contradicts \eqref{upperBoundAreaPleqS}.

Define $g=xy(x-1)(y-1)(f-s)(f+s)$. Without loss of generality we will suppose $g$ is square-free (if not, we will remove any repeated components). Then $g$ is a polynomial of degree at most 14, and $Z=Z(g)$. At each point $p \in Z$ for which $\nabla g$ is non-zero, the unit tangent vector to $Z$ is given by $v_p = \frac{(\partial_y g,\ -\partial_x g)}{|\nabla g|}$. By \eqref{largeWedge}, for each $p\in Z\cap U_1$, there are at least two indices $i$ for which $|v_p\cdot \nabla\phi_i(p)|\gtrsim \delta^{O_\alpha(c)}$, i.e. there exists a constant $C_3=C(\alpha)$ (that is independent of $p$ and $i$) so that
\[
\big( (\partial_y g)(\partial_x \phi_i) - (\partial_x g)(\partial_y\phi_i)\big)^2  \geq C_3 \delta^{C_3 c}\big( (\partial_x g)^2 + (\partial_y g)^2\big).
\] 

Cut the curve $Z$ at each point where it intersects $B(0,2)$; at each point where $|\nabla g|$ vanishes; and at each point satisfying
\[
\big( (\partial_y g)(\partial_x \phi_i) - (\partial_x g)(\partial_y\phi_i)\big)^2  = C_3 \delta^{C_\alpha c}\big( (\partial_x g)^2 + (\partial_y g)^2\big)\quad\textrm{for some index}\ i.
\] 
By slightly perturbing the constant $C_3$ if necessary, we can ensure that the latter set of points is finite. All together, we have cut $Z$ at $O(1)$ points, and after removing these points, $Z\cap B(0,1)$ is a union of $M = O(1)$ simple smooth curves. For each such curve $\gamma$, there are at least two indices $i,j$ so that $|v_p\cdot \nabla\phi_i(p)|\geq C_3 \delta^{C_3 c}$ for all $p\in\gamma$, and similarly for $\phi_j$. 

Let $C_4=C_4(\alpha)$ be chosen below, and fix $s=\delta^{C_5c}$ for $C_5=C_5(\alpha)$, so that the quantity $w$ from \eqref{defnW} satisfies $w=\delta^{C_4 c}$. We now attempt to apply Lemma~\ref{twoFunctionsOnASet} to each of the simple smooth curves constructed above. We conclude that if there is at least one curve $\gamma$ so that
\begin{equation}\label{curveLargeXIntersect}
\mathcal{E}_{\delta}(X \cap N_{\delta^{C_4 c}}\gamma)\geq (10M)^{-1} \delta^{-\alpha+c},
\end{equation}
then by Lemma~\ref{twoFunctionsOnASet} there is an index $i$ so that
\begin{equation}\label{largePhiIX}
\mathcal{E}_{\delta}(\phi_i(X))\geq \mathcal{E}_{\delta}(\phi_i(X \cap N_{\delta^{C_4 c}}\gamma))\gtrsim \delta^{-\alpha/2-C_4 c \alpha/2 +O_\alpha(c)}|\log\delta|^{-1}.
\end{equation}
If $C_4=C_4(\alpha)$ is chosen sufficiently large, then the RHS of \eqref{largePhiIX} is larger than $\delta^{-\alpha/2-c}$, and we are done. Hence we can conclude that \eqref{curveLargeXIntersect} fails for each of the curves $\gamma$, with $C_4$ as described above. We fix this choice of $C_4$, which in turn fixes the value of $C_5$ and $s=\delta^{C_5c}$.

Define $U = U_1\backslash\bigcup_{\gamma} \cap N_{2\delta^{C_4c}}\gamma$. With $s$ as defined above, we have
\begin{equation}\label{sizeOfFOnNbhdOfU2}
|f(z)|\geq s\quad\textrm{for all}\ z\in N_{\delta^{C_4c}}(U).
\end{equation}

Finally, let $C_6=\max(C_1,C_4)=O_\alpha(1)$, where $C_1$ is the constant from \eqref{largeWedge}. By Theorem \ref{whitneyCubeDecomp}, we can cover $N_{\delta^{C_6c}}(U)$ by interior-disjoint squares, so that the 2-fold dilate of each square is not contained in $N_{\delta^{C_6c}}(U)$. This mean, in particular, that each of the squares intersecting $U$ has side-length $\gtrsim \delta^{C_6c}$, and hence at most $O(\delta^{-C_6c})$ such squares intersect $U$. Thus by pigeonholing, we can select a square $Q$ with $\mathcal{E}_{\delta}(Q\cap X)\gtrsim\delta^{-\alpha+O_\alpha(c)}$. By \eqref{largeWedge} and \eqref{sizeOfFOnNbhdOfU2}, we have
\begin{equation}\label{controlOfphiOnQ}
|\nabla\phi_i\wedge\nabla\phi_j|\gtrsim\delta^{O_\alpha(c)},\quad |f|\gtrsim \delta^{O_\alpha(c)}\quad\textrm{on}\ Q.
\end{equation}

Thus if $c$ and $\delta_0$ are selected sufficiently small (depending on $\alpha)$, then by \eqref{computeCurvature} and \eqref{controlOfphiOnQ}, the functions $\phi_1,\phi_2,\phi_3$ satisfy the hypotheses of Theorem \ref{curvatureProjections}$^*$ on $Q$. We conclude that there exists an index $i$ so that
\begin{equation}\label{phiILarge}
\mathcal{E}_{\delta}( \phi_i(X) ) \geq \mathcal{E}_{\delta}(\phi_i(X\cap Q)) \geq \delta^{-\alpha/2-2c}.
\end{equation}
Finally, since $\phi_i(x)=(\Delta_{p_i}(x))^2$ and $|\phi_i(x)|\leq 2$ for $x\in [0,1]^2$, \eqref{phiILarge} implies that
\begin{equation}\label{finalIneq}
\mathcal{E}_{\delta}( \Delta_{p_i}(X) ) \gtrsim \delta^{-\alpha/2-2c}.
\end{equation}
Decreasing $\delta_0$ slightly if necessary to account for the implicit constant in the quasi-inequality \eqref{finalIneq}, we obtain \eqref{atLeastOneBigPinnedDistance}.

\section{Convexity and discretized projections}\label{convexitySec}
In this section we will prove Theorem \ref{convexProjections}. We will actually state and prove the following slightly more general (and technical) version of the theorem.
\begin{convexProjectionsThm}
For each $0< \alpha< 2,$ there exists $c=c(\alpha)>0$ and $\delta_0=\delta_0(\alpha)>0$ so that the following holds for all $0<\delta<\delta_0$. Let $I,J\subset[0,1]$ be intervals and let $u\colon I\to\RR,\ v\colon J\to\RR$ be smooth. Suppose that
\begin{equation}\label{derivbounds}
\begin{split}
\delta^c\leq \Big|\frac{d}{dx}u(x)\Big|\leq\delta^{-c},\quad \delta^c\leq \Big|\frac{d^2}{(dx)^2}u(x)\Big|\leq\delta^{-c}, &\quad\textrm{for}\ x\in I,\\
\delta^c\leq \Big|\frac{d}{dx}v(x)\Big|\leq\delta^{-c},\quad\quad\quad\; \Big|\frac{d^2}{(dx)^2}v(x)\Big|\leq\delta^{-c}&\quad\textrm{for}\ x\in J.
\end{split}
\end{equation}
Let $\phi_1(x,y)=x,$ $\phi_2(x,y)=y,$  $\phi_3(x,y)=\frac12(x+y)$, and $\phi_4(x,y)=u(x)+v(y)$. 
Let $X\subset I\times J$, with $\mathcal{E}_{\delta}(X)\ge \delta^{-\alpha+c}$, and suppose that for all squares $Q$ of sidelength $r\geq\delta$, $X$ satisfies the non-concentration condition
\begin{equation}\label{nonConcentrationCondXConvexEntropy}
\mathcal{E}_{\delta}(X\cap Q) \le  r^\alpha\delta^{-\alpha-c}.
\end{equation}
Then for at least one index $i$ we have
\begin{equation*}\tag{\ref{atLeastOneBigProjectionConvex}}
\mathcal{E}_{\delta}(\phi_i(X))\geq \delta^{-\alpha/2-c}.
\end{equation*}
\end{convexProjectionsThm}
\begin{proof}
Let $X$ be a union of $\delta$-squares such that $\mathcal{E}_\delta(X)= \delta^{-\alpha+c}$ and $X$ satisfies \eqref{nonConcentrationCondXConvexEntropy}. 
Put $A=\phi_1(X)$, $B=\phi_2(X)$, $C=\frac12\phi_3(X)$, and $D=\phi_4(X)$. 
Suppose that 
\begin{equation}\label{allSetsSmall}
\mathcal{E}_{\delta}(A),\; \mathcal{E}_{\delta}(B),\; \mathcal{E}_{\delta}(C),\; \mathcal{E}_{\delta}(D)\le \delta^{-\alpha/2-c}.
\end{equation}
If $c,\delta>0$ are sufficiently small, we will reach a contradiction. 
Replacing $X$ by a large subset of $X$, we may assume without loss of generality that
\begin{equation}\label{nonconABC}
\begin{split}
&\mathcal{E}_{\delta}(\phi_i(X) \cap K)\le \delta^{-c} |K|^{\alpha/2}\delta^{-\alpha/2},
\end{split}
\end{equation}
for all intervals $K$ of length at least $\delta$ and for each index $i=1,2,3$.
Indeed, this follows from Lemma~\ref{findingLargeNonconcentratedCartesianProduct} applied to $\phi_1$, $\phi_2$, $\phi_3$ and the set $X\subset I\times J$.

Note  that our assumption \eqref{allSetsSmall} and the fact that $X\subset A\times B$ imply 
\begin{equation}\label{ABlarge}
\mathcal{E}_{\delta}(A), \mathcal{E}_{\delta}(B)\ge \delta^{-\alpha/2+2c}.
\end{equation}

Similarly, we have
\begin{equation}\label{lowerC}
\mathcal{E}_{\delta}(C)\gtrsim \delta^{-\alpha/2+2c}.
\end{equation}
Indeed, let $\Gamma$ be the graph of $(x,y)\mapsto \frac12(x+y)$ restricted to $(x,y)\in X$. Then $\mathcal{E}_{\delta}(\Gamma)\approx \mathcal{E}_{\delta}(X)$. Up to reorder of the coordinates, $\Gamma$ is also the graph of $(x,z)\mapsto 2z-x$ restricted to the set $Z=\{(x,\tfrac12(x+y))\mid (x,y)\in X\}\subset A\times C$.
Thus
$$
\delta^{-\alpha+c}\approx  \mathcal{E}_{\delta}(\Gamma)\approx \mathcal{E}_{\delta}(Z) \lesssim \mathcal{E}_{\delta}(A\times C)\le \delta^{-\alpha/2-c}\mathcal{E}_{\delta}(C)
$$
which gives \eqref{lowerC}.

Let 
$$
Y:=\{(t,x,y)\in A\times B\times B\mid (t,x),\; (t,y)\in X \}.
$$
We have
\begin{equation}\label{Ylarge}
\mathcal{E}_{\delta}(Y)\gtrsim \delta^{-3\alpha/2+3c}.
\end{equation} 
Indeed, by Cauchy--Schwarz inequality,
$$
|Y|
= \int_{t\in A}|\phi_1^{-1}(t)\cap X|^2dt\ge 
\left(\int_{t\in A}|\phi_1^{-1}(t)\cap X|dt\right)^2/|A|=|X|^2/|A|  \gtrsim
\delta^{3-3\alpha/2+3c}.
$$

Consider the variety $V\subset[0,1]^4\times \RR$ given by
\begin{align*}
z&=\tfrac12(t+x)\\
w&=u(t)+v(y).
\end{align*}
Let $\Gamma'\subset V$ denote the graph of $(t,x,y)\mapsto (\tfrac12(t+x),u(t)+v(y))$ restricted to $Y$.
Then
$$
\mathcal{E}_{\delta}(\Gamma')\ge \mathcal{E}_{\delta}(Y)\gtrsim \delta^{-3\alpha/2+3c}.
$$
Let $S\subset B\times B\times C$ denote the projection of $\Gamma'$ to the $xyz$-space. We can view $\Gamma'$ as the graph of $(x,y,z)\mapsto (t,w)$ restricted to $S$. 
Indeed, for $(x,y,z)\in S$ we have 
$(t,w)=(2z-x,u(2z-x)+v(y))$. Then
\begin{equation}\label{Slarge}
\mathcal{E}_{\delta}(S)\gtrsim \delta^{c}\mathcal{E}_{\delta}(\Gamma')\gtrsim \delta^{-3\alpha/2+4c},
\end{equation}
where for this we used the upper bounds from \eqref{derivbounds}. 

Let
$$
G(x,y,z)=u(2z-x)+v(y),
$$
which is defined on 
$$
{\rm dom}(G)=[0,1]^3\cap \{2z-x\in I, \ y\in J\}.
$$
Then $S\subset {\rm dom}(G)\cap (B\times B\times C)$, $G(S)\subset D$, and $\mathcal{E}_{\delta}(S)\gtrsim \delta^{-3\alpha/2+4c}$.

We are almost ready to apply Lemma~\ref{Shmcor} to $G$. As a final preparation step we would like to restrict $G$ to a cube that contains a portion of $S$. For this we fix a parameter $0<r<1$, and consider the family ${\mathcal Q}$ of cubes in $[0,r]^3+r\ZZ^3$ that intersect $S$. 
Using $\mathcal{E}_{\delta}(S)\le \mathcal{E}_{\delta}(B\times B\times C)$, we have 
\begin{equation}\label{sizeQ}
\#{\mathcal Q}\le r^{-3\alpha/2-3c}.
\end{equation}

Recall that ${\rm dom}(G)=[0,1]^3\cap \{2z-x\in I, \ y\in J\}$. Write $I=[a_0,a_1]$
and let $\partial {\mathcal Q}\subset {\mathcal Q}$ denote the subset of cubes that intersect one of the planes $2z-x=a_0$ and $2z-x=a_1$.  
We claim that 
\begin{equation}\label{boundaryDGsmall}
\mathcal{E}_{\delta}\left(S\cap \bigcup_{Q\in\partial{\mathcal Q}} Q\right)
\lesssim r^{\alpha/2}\delta^{-3\alpha/2-3c}.
\end{equation}
Indeed, let $\mathcal T$ denote the projection of the cubes in $\partial \mathcal Q$ to the $xy$-plane. Then $\mathcal T$ is a family of interior-disjoint squares of side length $r$  that intersect $B\times B$. 
Each square in $\mathcal T$ is the projection of at most four cubes from $\partial \mathcal Q$, as is easy to verify. Thus, we have
\begin{align*}
|S\cap \bigcup_{Q\in\partial{\mathcal Q}} Q|
&\le \sum_{Q\in \partial{\mathcal Q}}|(B\times B\times C)\cap Q| \\ 
&= \sum_{Q=T\times K\in \partial{\mathcal Q}}|(B\times B)\cap T||C\cap K| \\ 
&\le \delta^{-c}r^{\alpha/2}\delta^{1-\alpha/2}\sum_{Q=T\times K\in \partial{\mathcal Q}}|(B\times B)\cap T| \\ 
&\lesssim 4\delta^{-c}r^{\alpha/2}\delta^{1-\alpha/2}|B\times B| \\ 
&\lesssim \delta^{-3c}r^{\alpha/2}\delta^{3-3\alpha/2},
\end{align*}
where for the third line we use the non-concentration inequality \eqref{nonconABC}.
This proves \eqref{boundaryDGsmall}.

Letting $r\approx\delta^{14c/\alpha}$ and in view of \eqref{Slarge} and \eqref{boundaryDGsmall}, we get
\begin{equation}\label{interiorDGlarge}
\mathcal{E}_{\delta}\left(S\cap \bigcup_{Q\in{\mathcal Q}\setminus\partial{\mathcal Q}} Q\right)
\ge \frac12 \mathcal{E}_{\delta}(S).
\end{equation}
By the pigeonhole principle and \eqref{sizeQ}, there exists a cube $Q_0\in {\mathcal Q}\setminus\partial{\mathcal Q}$ such that 
\begin{equation}\label{ScapQ0large}
|S\cap Q_0|\ge \delta^{3-3\alpha/2+\eta},
\end{equation}
where $\eta=c(37+42c/\alpha)$.

Finally, note that $Q_0\cap {\rm dom}(G)$ is a rectangular box, and write $Q_0\cap {\rm dom}(G)=I_1\times I_2\times I_3$. 
We apply Lemma~\ref{Shmcor} to $G:I_1\times I_2\times I_3\to \RR$; let $\sigma=\sigma(\alpha)$ be the quantity given by the lemma. Write $A_1=B\cap I_1$, $A_2=B\cap I_2$ and $A_3=C\cap I_3$. Combining  
\eqref{allSetsSmall} and \eqref{ScapQ0large}, we see that
$$
\mathcal{E}_{\delta}(A_i)\ge \delta^{-\alpha/2+\eta+2c},~~\text{for $i=1,2,3$}.
$$
By \eqref{nonconABC}, we have the non-concentration inequality 
$$
\mathcal{E}_{\delta}(A_i\cap K)\le 
 \delta^{-c}r^{\alpha/2}\delta^{-\alpha/2},
$$
where $K$ is an interval of length $r\ge \delta$, for $i=1,2,3$. 

We have 
\begin{align*}
G_x&=-u'(2z-x),~~G_y=v'(y),~~G_z=2u'(2z-x)
\end{align*}
and
\begin{align*}
G_{xx}&=u''(2z-x),~~
G_{xy}=G_{yz}=0,~~G_{xz}=-2u''(2z-x)\\
G_{yy}&=v''(y),~~G_{zz}=4u''(2z-x).
\end{align*}
In view of \eqref{derivbounds}, we have
$$
\|G\|_{C^2(I_1\times I_2\times I_3)}\lesssim \delta^{-c}
$$
and
$$
\inf_{I_1\times I_2\times I_3}|\nabla G_{(z)}(x,y)|\ge \delta^c.
$$
So $G$ satisfies \eqref{boundOnC2NormG} and \eqref{boundOnNablaG}.

We next verify \eqref{lowerBoundThetap}. We have
\begin{align*}
\partial_z \theta_{(x,y)}(z)
&=\partial_z \arctan\frac{G_y}{G_x}\\
&=\frac{G_{yz}G_x-G_yG_{xz}}{G_x^2+G_y^2}\\
&=\frac{2u''(2z-x)v'(y)}{(u'(2z-x))^2+(v'(y))^2}.
\end{align*}
Thus, in view of \eqref{derivbounds}, we have
\begin{align*}
\inf_{I_1\times I_2\times I_3} |\partial_z \theta_{(x,y)}(z)|&\gtrsim \delta^{4c}. 
\end{align*}

Thus Lemma~\ref{Shmcor} tells us that
$$
\mathcal{E}_{\delta}(S\cap Q_0)\le \mathcal{E}_{\delta}\left(\{(x,y,z)\in A_1\times A_2\times A_3\mid G(x,y,z)\in D \}\right)\le \delta^{-3\alpha/2 +\sigma}.
$$
However, choosing $c=c(\alpha)>0$ sufficiently small,
this contradicts \eqref{ScapQ0large}. Thus our assumption \eqref{allSetsSmall} is false, which proves the theorem. 
\end{proof}

\section{Exceptional vantage points for pinned distances}\label{pinnedFalconerAndVisibilitySection}
In this section we will prove Theorem \ref{unifLowerBoundPinnedDistances} using its single-scale variant, Theorems \ref{pinnedDistanceCor}. 
Suppose that the set 
\[
K = \{p\in \RR^2\colon \dim\Delta_p(A) \leq (\dim A)/2 + c\}
\] 
is curved. If $c$ is sufficiently small (depending on $\alpha$) then we will obtain a contradiction. Our arguments will be similar to those in \cite[Section 7]{B10}. Our first task will be to reduce to the case where $A$ and $K$ are contained in $[0,1]^2$. 

Let $\mu$ be a Borel probability measure on $K$ that satisfies \eqref{ballEstimateForTriples}. Select $R$ sufficiently large so that $\dim(A \cap B(0,R))\geq\dim A - c/2$ and $\mu(B(0,R))\geq 1/2$. Define $h(x,y) = (2R)^{-1}(x,y) + (1/2, 1/2)$, so $h(B(0,R))\subset[0,1]^2$. Abusing notation slightly, we will replace $K$ by $h(K\cap B(0,R))$, and we will replace $A$ by $h(A\cap B(0,R))$. Then $K,A\subset [0,1]^2$, $\dim(A)\geq\alpha-c/2$, and $\dim \Delta_p(A)\leq\alpha/2+c< (\dim A)/2 + 2c$ for each $p\in K$. Continuing our abuse of notation, we will let $\mu$ be a Borel probability measure on $K$ that satisfies  \eqref{ballEstimateForTriples}.

Let $\nu$ be a Borel probability measure on $A$ that satisfies the Frostman condition 
\begin{equation}\label{frostmanOnA}
\nu(A \cap B(x,r))<Cr^{\alpha-c}.
\end{equation}
Let $\delta_0=\delta_0(\dim A)$ be the quantity from Theorem \ref{pinnedDistanceCor}, and let $k_0\in\ZZ$ be sufficiently large so that $2^{-k_0}<\delta_0$. Let $p\in K$. Since $\dim(\Delta_p(A))< \dim(A)/2 + 2c$, we can cover $\Delta_p(A)$ by a union of intervals $\bigcup_{I\in \mathcal{G}_p}I = \bigcup_{k\geq k_0}\bigcup_{I\in \mathcal{G}_{p,k}}I$, where $\mathcal{G}_{p,k}$ is a set of at most $\delta^{-\alpha/2-2c}$ intervals of length $2^{-k}$. Since $\nu$ is a probability measure, we have 
\[
\sum_{k\geq k_0}\sum_{I\in\mathcal{G}_{p,k}}\nu( d_p^{-1}(I)) \geq 1.
\]
Integrating with respect to $\mu$, we conclude that
\[
\sum_{k\geq k_0}\int \sum_{I\in\mathcal{G}_{p,k}}\nu( d_p^{-1}(I)) d\mu(p) =\int \sum_{k\geq k_0}\sum_{I\in\mathcal{G}_{p,k}}\nu( d_p^{-1}(I)) d\mu(p) \geq 1,
\]
and thus there exists $k\geq k_0$ so that
\[
\int \sum_{I\in\mathcal{G}_{p,k}}\nu( d_p^{-1}(I)) d\mu(p) \geq \frac{6}{\pi^2}k^{-2}.
\]
Fix this choice of $k$. Define $\delta=2^{-k}$ (so $k^{-2} \sim  |\log\delta|^{-2}$). Cover $[0,1]^2$ by squares of the form $[n\delta,\ (n+1)\delta]\times[m\delta,\ (m+1)\delta]$. After dyadic pigeonholing, we can choose a set $\mathcal{S}$ of such squares, and a number $w>0$ so that if we define $A' = A \cap \bigcup_{S\in\mathcal{S}}S,$ then
\begin{equation}\label{mostMassCaptured}
\int \sum_{I\in\mathcal{G}_{p,k}}\nu(A'\cap d_p^{-1}(I)) d\mu(p) \gtrsim |\log\delta|^{-3},
\end{equation}
and $w\leq \nu(A'\cap S)<2w$ for each $S\in\mathcal{S}$. By \eqref{mostMassCaptured}, we have $|\log\delta|^{-3} w^{-1} \lesssim  \#\mathcal{S}\leq w^{-1}$. By \eqref{frostmanOnA} we have $w\leq C \delta^{\alpha-c}$. Define $A_{\delta}=\bigcup_{S\in\mathcal{S}}S$. By \eqref{mostMassCaptured} we have $\nu(A_{\delta})\gtrsim|\log \delta|^{-3}$, and thus $\#\mathcal{S}\gtrsim \delta^{-3}/(C \delta^{\alpha-c})\gtrsim \delta^{-\alpha+c}|\log\delta|^{-2}.$

Note that if $p\in\RR^2$, $I$ is an interval of length $\delta$, and $d_p(z)\in I$, then by the triangle inequality, $\Delta_p(S)\subset 3I$, where $S\in\mathcal{S}$ is a square containing $z$, and $3I$ is the interval of length $3\delta$ with the same midpoint as $I$. In particular, 
\[
\nu( d_p^{-1}(I))\lesssim w \#\{S\in\mathcal{S}\colon d_p(S) \subset 3I\}\lesssim(\#\mathcal{S})^{-1}\#\{S\in\mathcal{S}\colon d_p(S) \subset 3I\}.
\] 
For each $p\in K,$ define $\mathcal{I}_p=\{3I\colon I\in \mathcal{G}_{p,k}\}$ (note that $\#\mathcal{I}_p\leq \delta^{-\alpha/2-2c})$ and define 
\[
\mathcal{S}_p = \bigcup_{I \in \mathcal{I}_p} \{S\in\mathcal{S}\colon \Delta_p(S)\subset I\}.
\]
The sets $\{S\in\mathcal{S}\colon \Delta_p(S)\subset I\}$ are boundedly overlapping as $I$ ranges over the elements in $\mathcal{I}_p$, and thus \eqref{mostMassCaptured} becomes
\[
\int (\# \mathcal{S}_p) d\mu(p)  \gtrsim |\log\delta|^{-3}(\#\mathcal{S}).
\]
For each $S\in\mathcal{S}$, define $m(S) = \mu(\{p \in K\colon S\in \mathcal{S}_p\})$. We have

\begin{align}
|\log\delta|^{-9}(\#\mathcal{S})^3&\lesssim  \Big(\sum_{S\in\mathcal{S}}m(S)\Big)^3\leq(\#\mathcal{S})^2\Big(\sum_{S\in\mathcal{S}} m(S)^3\Big),\label{S3}
%
\end{align}
and thus
\begin{align}
&\int \#(\mathcal{S}_{p_1}\cap \mathcal{S}_{p_2} \cap \mathcal{S}_{p_3})d\mu(p_1)d\mu(p_2)d\mu(p_3)\gtrsim |\log\delta|^{-9}(\#\mathcal{S}),\label{TriplesSi}%
\end{align}
Define
\begin{align}
\mathcal{T} & = \big\{(p_1,p_2,p_3)\in K^3\colon \#(\mathcal{S}_{p_1}\cap \mathcal{S}_{p_2} \cap \mathcal{S}_{p_3}) \geq C_0^{-1}|\log\delta|^{-9}(\#\mathcal{S})\big\},\label{quadruplesBigSquareIntersection}
\end{align}
If $C_0$ is a sufficiently large absolute constant, then $\mu^3(\mathcal{T})\geq C_0^{-1}|\log\delta|^{-9}$.

Next, by \eqref{ballEstimateForTriples}, for each $r>0$ we have
\begin{align} 
&\mu^3\{(p_1,p_2,p_3)\colon |T(p_1,p_2,p_3)|\leq r \}\leq C r^{\beta},\label{triplesSmallTriangle}
\end{align}

Comparing \eqref{quadruplesBigSquareIntersection} and \eqref{triplesSmallTriangle}, if the constant $C_1$ is chosen sufficiently large, then there exists $(p_1,p_2,p_3)\in K^3$ so that
\begin{align}
&\#(\mathcal{S}_{p_1}\cap \mathcal{S}_{p_2} \cap \mathcal{S}_{p_3}) \geq C_0^{-1}|\log\delta|^{-9}(\#\mathcal{S})\label{lotsOfSInsideSp1p2p3p4},
\end{align}
and 
\begin{align}
&|T(p_1,p_2,p_3)|> C_1^{-1}|\log\delta|^{-9/\beta},\label{p1p2p3p4SpanBigTriangles}
\end{align}

Define $\mathcal{S}'=\mathcal{S}_{p_1}\cap \mathcal{S}_{p_2} \cap \mathcal{S}_{p_3}$. Then $\mathcal{E}_{\delta}(\Delta_{p_i}(\mathcal{S}')\lesssim \delta^{-\alpha/2-2c}$ for each index $i$. If $c=c(\alpha)>0$ and $\delta_0=\delta_0(\alpha)>0$ are selected sufficiently small, then this violates Theorem \ref{pinnedDistanceCor}.

\section{Discretized expanding polynomials} \label{mainThmProofsSectionOne}
In this section we will prove Theorem \ref{main-entropy-growth}. Recall that Proposition \ref{countingNumberQuadruplesNablePWedgeP} controls the number of solutions to $P(a,b) = P(a',b')$ on a region where $P_x,P_y,P_{xy}$ and $K_P$ are large. If any of these quantities vanish identically, then $P$ is a special form. But it is possible that none of these quantities vanish identically, but nonetheless they are small on a substantial portion of $X\subset A\times B$. In the next sub-section, we will show that after restricting to a suitable sub-square of $[0,1]^2$, we can suppose that $P_x,P_y,P_{xy}$ and $K_P$ are large.

\subsection{Analytic varieties do not concentrate on fractal sets}\label{analFunctNonConcSec}
In this section we will introduce several technical results showing how products of non-concentrated sets can intersect thin neighborhoods of varieties $Z(f)$, where $f$ is an analytic function. Our setup will be as follows. Let $A_1,\ldots,A_d\subset[0,1]$ be sets and let $A = A_1\times\cdots\times A_d$. Let $\delta>0$ and suppose that for all intervals $J$ of length at least $\delta$, each set $A_i$ satisfies the non-concentration condition
\begin{equation*}
\mathcal{E}_{\delta}(A_i \cap J) \leq \delta^{-\eta}|J|^{\kappa}\mathcal{E}_{\delta}(A_i).
\end{equation*}

We would like to show that if $f\colon\RR^d\to\RR$ is analytic, then $|N_{t}(Z(f))\cap A|$ must be much smaller than $|A|$. Results of this type, combined with Whitney's cube decomposition, will allow us to find a (reasonably) small collection of interior disjoint cubes whose union covers most of $A$, so that $f$ is not too small on each cube. Results of this type will help us control the size of various ``bad'' regions where discretized expansion might fail.

\begin{thm}[Stratification of the vanishing locus of an analytic function \cite{Lo2}]\label{stratificationRealAnalSet}
Let $U\subset\RR^d$ be an open set that contains $[0,1]^d$ and let $\phi\colon U \to\RR$ be real analytic and not identically zero. Then there is a finite set $\mathcal{M}$ of smooth (proper) submanifolds of $\RR^d$ so that $Z(\phi)\cap [0,1]^d \subset \bigcup_{M\in\mathcal{M}}M$.
\end{thm}

\begin{lem}
\label{neighborhoodAnalyticVariety}
Let $M$ be a smooth (proper) submanifold of $\RR^d$, and suppose $M \cap [0,1]^d$ is compact. Then there is a constant $C>0$ so that the following holds. 

Let $A_1,\ldots,A_d\subset[0,1]$ be sets and let $A = A_1\times\cdots\times A_d$. Let $\delta>0$ and suppose that for all intervals $J$ of length at least $\delta$, each set $A_i$ satisfies the non-concentration condition
\begin{equation*}
\mathcal{E}_{\delta}(A_i \cap J) \leq \delta^{-\eta}|J|^{\kappa}\mathcal{E}_{\delta}(A_i).
\end{equation*}
Then
for each $s>\delta$ we have
\begin{equation}\label{thickenedNbhdEstimate}
\mathcal{E}_{\delta}\big(  A\cap N_{s}(M) \big) \leq C \delta^{-\eta} s^{\kappa}\mathcal{E}_{\delta}(A).
\end{equation}
\end{lem}
\begin{proof}
For each $p\in M\cap [0,1]^d$, let $i(p)\in \{1,\ldots,d\}$ be an index so that the coordinate unit vector $e_{i(p)}$ is not contained in the tangent space $T_p(M)$. If $t>0$ is selected sufficiently small, then the set $M_p = M\cap B(p,t)$ is a smooth manifold, and there exists a constant $C_p$ so that for all $q,q'\in M_p$, we have
\begin{equation}\label{eiCoordinateBd}
|e_{i(p)}\cdot q - e_{i(p)}\cdot q'|\leq  C_p|\pi_{p}(q)-\pi_{p}(q')|,
\end{equation}
where $\pi_p\colon\RR^{d}\to\RR^{d-1}$ is the orthogonal projection to the $(d-1)$-dimensional subspace $e_{i(p)}^\perp\subset\RR^d$. 
Indeed, the above inequality follows from the fact that $\inf_{v\in T_p(M)}\angle(e_{i(p)},v)>0$, and the map $q\mapsto T_q(M)$ is continuous. 

Since $M\cap [0,1]^d$ is compact, we can cover $M\cap [0,1]^d$ by a finite set $\mathcal{M}$ of manifolds of the form $M_p$. It suffices to prove  \eqref{thickenedNbhdEstimate} for each of these manifolds. Let $C_0=\max_{M_p\in\mathcal{M}} C_p$. 

Let $M_p\in\mathcal{M}$. Let $s\geq\delta$ and let $\mathcal{Q}$ be the set of cubes of side-length $s$ aligned with the grid $(s\ZZ)^d$ that intersect $M_p$. We claim that at most $C_1$ cubes from $\mathcal{Q}$ can have the same projection under $\pi_{p}$, for some $C_1$ that depends only on $d$ and $C_0$. Indeed, if $C_1>2$ cubes have the same projection under $e_{i(p)}$, then select points $q,q'\in M_p$ from two cubes that have greatest distance. Since there are at least $C_1-2$ cubes in-between the cubes containing $q$ and $q'$, we have $|e_{i(p)}\cdot q - e_{i(p)}\cdot q'|\geq (C_1-2)s$. On the other hand, since all of the cubes (and in particular, the cubes containing $q$ and $q'$) have the same projection under $\pi_p$, we have that $|\pi_{p}(q)-\pi_{p}(q')|\leq \sqrt{d-1}s$. Thus by \eqref{eiCoordinateBd}, we must have $(C_1-2)s \leq C_0\sqrt{d-1}s,$ so $C_1\leq C_0\sqrt{d-1}+2$, which proves our claim. 

For each such cube $Q\in\mathcal{Q}$ we have the estimate 
\begin{equation}\label{coveringNumberSCube}
\mathcal{E}_{\delta}(Q\cap A) \leq \delta^{-\eta}s^{\kappa}\mathcal{E}_{\delta}(A_{i(p)}) \mathcal{E}_{\delta}( \pi_{p}(A\cap Q)).
\end{equation}
Summing \eqref{coveringNumberSCube} over all cubes in $\mathcal{Q}$ and using the fact that at most $C_1$ cubes have the same projection under $\pi_p$, we obtain
\begin{equation*}
\begin{split}
\mathcal{E}_{\delta}(N_s(M)\cap A)&\leq C_1 \delta^{-\eta}s^{\kappa}\mathcal{E}_{\delta}(A_{i(p)})\mathcal{E}_{\delta}( \pi_p(A))\\
&\leq C_1 \delta^{-\eta}s^{\kappa}\mathcal{E}_{\delta}(A).\qedhere
\end{split}
\end{equation*}
\end{proof}

\begin{rem}
Note that in the special case where where exists an index $i$ so that $e_i$ is not contained in $T_p(M)$, for every $p\in M$, it is sufficient in Lemma~\ref{neighborhoodAnalyticVariety} to require the non-concentration condition only for $A_i$.
\end{rem}

Combining Lemma~\ref{neighborhoodAnalyticVariety} and Theorem~\ref{stratificationRealAnalSet}, we obtain the following.

\begin{cor}\label{coveringNbhAnalyticSet}
Let $U\subset\RR^d$ be an open set that contains $[0,1]^d$ and let $\phi\colon U \to\RR$ be real analytic and not identically zero. Then there is a constant $C$ so that the following holds. 

Let $A_1,\ldots,A_d\subset[0,1]$ be sets and let $A = A_1\times\cdots\times A_d$. Let $\delta>0$ and suppose that for all intervals $J$ of length at least $\delta$, each set $A_i$ satisfies the non-concentration condition
\begin{equation*}
\mathcal{E}_{\delta}(A_i \cap J) \leq \delta^{-\eta}|J|^{\kappa}\mathcal{E}_{\delta}(A_i).
\end{equation*}
Then
for each $s>\delta$ we have
\begin{equation}\label{CoveringThinNbhdZPhi}
\mathcal{E}_{\delta}\big(  A\cap N_{s}(Z(\phi)) \big) \leq C \delta^{-\eta} s^{\kappa}\mathcal{E}_{\delta}(A).
\end{equation}
\end{cor}

\begin{thm}[\L{}ojasiewicz's inequality \cite{Lo}]\label{Lojasiewicz}
Let $U\subset\RR^d$ be open, let $\phi\colon U \to\RR$ be real analytic and let $K\subset\RR^d$ be compact. Then there exists a constant $C>0$ so that for each $z\in K$, we have
\begin{equation*}
|\phi(z)| \geq C^{-1}\cdot  {\rm dist}(z, Z(\phi))^{C}.
\end{equation*} 
\end{thm}

Combining Theorem \ref{whitneyCubeDecomp}, Theorem \ref{Lojasiewicz}, and Corollary \ref{coveringNbhAnalyticSet}, we obtain the following 
\begin{cor}\label{analyticFunctionsLargeOnQ}
Let $U\subset\RR^d$ be an open set that contains $[0,1]^d$ and let $f_1,\ldots,f_k\colon U\to\RR$ be real analytic functions, none of which are identically zero. Then there is a constant $C_1=C_1(f_1,\ldots,f_k)$ so that the following holds.

Let $A_1,\ldots,A_d\subset[0,1]$ be sets and let $A = A_1\times\cdots\times A_d$. Let $\delta>0$ and suppose that for all intervals $J$ of length at least $\delta$, each set $A_i$ satisfies the non-concentration condition
\begin{equation*}
\mathcal{E}_{\delta}(A_i \cap J) \leq \delta^{-\eta}|J|^{\kappa}\mathcal{E}_{\delta}(A_i).
\end{equation*}

Let $\delta>0$ and $0<w<\kappa/2$. Then there is a set of cubes $\mathcal{Q}$ with disjoint interiors that are contained in $[0,1]^d$, with $\#\mathcal{Q} \lesssim \delta^{-dw}$ and
\begin{equation}\label{cubeHasLargeVolume}
\mathcal{E}_{\delta}\big(A\ \backslash\  \bigcup_{Q\in\mathcal{Q}}Q\big) \leq C_1 \delta^{\frac{w\kappa}{C_1}-\eta}\mathcal{E}_{\delta}(A),
\end{equation}
so that for each index $j$ we have

\begin{equation}\label{fjRoughlyConstOnQ}
|f_j(z)|\geq  \delta^{w}\ \textrm{on}\ \bigcup_{Q\in\mathcal{Q}}Q.
\end{equation}
\end{cor}
We remark briefly on how the above results yield Corollary \ref{analyticFunctionsLargeOnQ}. Let $Z = \bigcup_{i=1}^k \{f_i=0\}$. By Theorem \ref{Lojasiewicz}, there is a constant $C_1$ so that  $|f_j|\geq\delta^w$ on $[0,1]^d\backslash N_{C_1 \delta^{w/C_1}}(Z)$. Without loss of generality we can suppose that $C_1\geq 1$, and hence $C_1 \delta^{w/C_1}\geq\delta^w$. Apply Theorem \ref{whitneyCubeDecomp} to the compliment of $Z$; the set of cubes intersecting $[0,1]^d\backslash N_{C_1\delta^{w/C_1}}(Z)\subset [0,1]^d\backslash N_{\delta^w}(Z)$ has cardinality $O(\delta^{-2dw})$. Finally, \eqref{cubeHasLargeVolume} follows from Corollary \ref{coveringNbhAnalyticSet} (after possibly increasing the constant $C_1$).

\subsection{Energy Dispersion implies Entropy Growth}
Corollary \ref{analyticFunctionsLargeOnQ} allows us to find a large square $Q$ where the functions $f_i$ are bounded away from zero. Proposition \ref{countingNumberQuadruplesNablePWedgeP} then says that there are few solutions to $P(a,b) = P(a',b')$ on this square. The next step is to use Cauchy-Schwarz to conclude that $P\big((A\times B) \cap Q\big)$ must be large. Here is a precise version of that statement
\begin{lem}\label{CSLem}
Let $Q\subset[0,1]$ be a square, and let $P\colon Q\to\RR$ be a smooth function with 
\[
|\partial_x P|\geq c\ \ \textrm{and}\ \ |\partial_y P|\geq c\quad\textrm{on}\ Q.
\]
Let $X\subset Q$ be a union of squares of side-length $\delta$. Then
\[
\mathcal{E}_{\delta}(P(X))\gtrsim \frac{c\big(\mathcal{E}_{\delta}(X)\big)^2}{\mathcal{E}_{\delta}\big(\{(x,y,x',y')\in X^2\colon P(x,y) = P(x',y')\}\big)}.
\]
\end{lem}
\begin{proof}
Since $X$ is a union of squares of side-length $\delta$, we can select a $\delta$-separated set $\tilde X\subset X$ with $N_{\delta/2}(\tilde X) \subset X$ and $\#\tilde X = \mathcal{E}_{\delta}(X)$.

Define 
\begin{equation*}
\tilde P(x,y)=\frac{1}{2}c\delta\lfloor 2(c\delta)^{-1}P(x,y)\rfloor.
\end{equation*}
Let $\mathcal{Q}$ be the set of quadruples $(\tilde x, \tilde y, \tilde x', \tilde y')\in (\tilde X)^2$  with  $\tilde P(\tilde x, \tilde y) = \tilde P(\tilde x^\prime,\tilde y^\prime).$ For such a quadruple we have
\begin{equation}\label{tildePCloseToP}
|P(\tilde x, \tilde y) - P(\tilde x^\prime,\tilde y^\prime)|\leq \frac{1}{2}c\delta.
\end{equation}
Since $|\partial_y P|\geq c $ on $Q$, there exists $y^\prime$ with $|\tilde y'-y'|<\delta/2$ so that $P(\tilde x,\tilde y) = P(\tilde x^\prime,y^\prime)$. Since $N_{\delta/2}(\tilde X) \subset X$, we have $(\tilde x',y')\in X$. Since the quadruples in $\mathcal{Q}$ are $\delta$-separated, the corresponding quadruples $\{(\tilde x,\tilde y,\tilde x', y'\}$ are $\delta/2$ separated. In particular, we have 
\begin{equation}\label{boundSizeQVsPQuadruples}
\# \mathcal{Q}
\leq 16 \mathcal{E}_{\delta}\big(\{(x,y,x',y')\in X^2\colon P(x,y) = P(x',y')\}\big).
\end{equation}
By Cauchy-Schwarz, 
\begin{equation}\label{CSBoundTildeP}
\# \tilde P(\tilde X) \geq
(\# \tilde X)^2/ \#\mathcal{Q}.
\end{equation}
Finally, since $\tilde P$ takes values in $(\frac{1}{2}c\delta)\ZZ$ and in view of \eqref{tildePCloseToP},
 \begin{equation}\label{controlPAB}
\mathcal{E}_{\delta}(P(X))   \gtrsim c\big(\#\tilde P(\tilde X)\big).
\end{equation} 
The result now follows by combining \eqref{boundSizeQVsPQuadruples}, \eqref{CSBoundTildeP}, and \eqref{controlPAB}. 
\end{proof}

\subsection{Proof of Theorem \ref{main-entropy-growth}} 
We are now ready to prove Theorem \ref{main-entropy-growth}. We can assume that none of $\partial_x P$, $\partial_y P,$ $\partial_{xy}P$, or $K_P$ vanishes identically, since if any of these quantities vanish identically, then by Lemma~\ref{whenIsPSpecialFormProp}, $P$ is a special form and we are done. 

Let $C_1$ be the constant obtained by applying Corollary \ref{analyticFunctionsLargeOnQ} to the functions $\partial_x P$, $\partial_y P,$ $\partial_{xy}P$, and $K_P$. We will apply this lemma with $w = 2C_1\eta/\kappa$ ($\eta>0$ will be chosen small enough that $w<\kappa/2$, and thus the hypothesis of Lemma \ref{analyticFunctionsLargeOnQ} is satisfied), and let $\mathcal{Q}$ be the resulting collection of cubes. With this choice of $w$, we have
\[
\mathcal{E}_{\delta}\Big( (A\times B) \backslash \bigcup_{Q\in\mathcal{Q}}Q\Big)\leq \frac{1}{2}\delta^{\eta}\mathcal{E}_{\delta}(A\times B),
\]
and thus

\begin{equation*}
\mathcal{E}_{\delta}\Big( E \cap \bigcup_{Q\in\mathcal{Q}}Q\Big)\geq \frac{1}{2}\mathcal{E}_{\delta}(E),
\end{equation*}
and 
\[
\#\mathcal{Q}\lesssim \delta^{-8C_1\eta/\kappa}.
\]
By pigeonholing, there is a cube $Q\in\mathcal{Q}$ so that
\begin{equation*}
\mathcal{E}_{\delta}((A\times B)\cap Q) \gtrsim \delta^{8w/\kappa}\mathcal{E}_{\delta}(E)\gtrsim \delta^{9C_1\eta/\kappa-2\alpha}.
\end{equation*}
On this cube, each of the functions $\partial_x P$, $\partial_y P,$ $\partial_{xy}P$, and $K_P$ size at least $\delta^w \gtrsim \delta^{C_1\eta/\kappa}$. Note as well that $|\nabla P|\lesssim 1$ on $Q$.

Let $A_1\times B_1 = (A\times B)\cap Q$ and let $E_1=E\cap Q$.  If $\eta = \eta(\alpha,\kappa,C_1)$ is selected sufficiently small, then by Proposition \ref{countingNumberQuadruplesNablePWedgeP}, there exists $\eps^\prime = \eps^\prime(\alpha,\kappa)>0$ so that 
\[
{\mathcal E}_\delta\big(\{ ( x,  x',  y,  y')\in A_1^2\times B_1^2\colon P(x,y)=P(x',y')\}\big)\lesssim \delta^{-3\alpha+\eps^\prime}. 
\]
and thus
\begin{equation}\label{quadruples}
{\mathcal E}_\delta\big(\{ ( x,  y,  x',  y')\in E_1^2 \colon P(x,y)=P(x',y')\}\big)\lesssim \delta^{-3\alpha+\eps^\prime}
\end{equation}
Applying Lemma \ref{CSLem} to $A_1\times B_1\subset Q$ with $c = \delta^w$, we conclude that
\[
\mathcal{E}_{\delta}(P(E))\geq \mathcal{E}_{\delta}(P(E_1))\gtrsim   \frac{\delta^{w} (\delta^{9C_1\eta/\kappa-2\alpha})^2}{\delta^{-3\alpha+\eps^\prime}}\geq \delta^{20C_1\eta/\kappa-\alpha-\eps'}.
\]
To complete the proof, select $\eta=\eta(\alpha,\kappa,C)$ sufficiently small so that $\eps^\prime\geq 40\eta C_1/\kappa$, and select $\eps = \eps^\prime/2$.

\section{Dimension expansion for analytic functions}\label{dimExpansionSec}
In this section we will prove Theorem \ref{dimension-expander}.
The basic idea is we select a square $Q$ so that the sets $A',B'$ defined by $A'\times B' = (A\times B)\cap Q$ have large dimension, and $\partial_xP,\ \partial_yP, \partial_{xy}P$, and $K_P$ are large on $Q$. We then discretize and apply Proposition \ref{countingNumberQuadruplesNablePWedgeP}.

\subsection{Finding a good square}
\begin{defn}
Let $A\subset\RR$, let $x\in\RR$, and let $\beta>0$. We say that $A$ has {\it local dimension} $\geq\beta$ at $x$ if $\dim(A\cap U)\geq\beta$ for every neighborhood $U$ of $x$. Otherwise we say $A$ has local dimension $<\beta$ at $x$. 
\end{defn}

\begin{lem}\label{localDimMaximal}
Let $A\subset\RR$. Then for each $\eps>0$, there is at least one point $x\in \RR$ where $A$ has local dimension $\geq\dim(A)-\eps$. 
\end{lem}
\begin{proof}
Suppose not. For each $x\in \RR$, let $U_x$ be a neighborhood of $x$ with $\dim(A\cap U_x)\leq \dim(A)-\eps$. Let $\mathcal{U}\subset \{U_x\colon x\in A\}$ be a countable sub-cover of $\RR$. Then $\dim(A) = \max_{U\in\mathcal{U}}\dim(A\cap U)\leq \dim(A)-\eps$, which is impossible.
\end{proof}

\begin{cor}\label{ManyPointsAchieveLocalDim}
Let $A\subset\RR$, let $\eps>0$, and let $B$ be the set of points $x\in\RR$ where $A$ has local dimension $\ge \dim(A)-\eps$. Then either $A\setminus B=\emptyset$ or $\dim A\setminus B < \dim A$. In particular, if $\dim A>0$ then 
there are infinitely many points $x\in\RR$ where $A$ has local dimension $\geq\dim(A)-\eps$. 
\end{cor}

\begin{lem}\label{findGoodPoint}
Let $f$ be a function that is analytic and not identically zero on an open neighborhood of $[0,1]^d$, let $\eps>0$, and let $A_1,\ldots,A_d\subset [0,1]$ have positive Hausdorff dimension. Then there is a point $p=(p_1,\ldots,p_d)\in (0,1)^d\backslash Z(f)$ so that for each index $i$, $A_i$ has local dimension $\geq\dim(A_i)-\eps$ at $p_i$.
\end{lem}
\begin{proof}
We will prove the result by induction on $d$. When $d=1$ the result follows from Corollary~\ref{ManyPointsAchieveLocalDim} and the fact that since $f$ is analytic and not identically zero on a neighborhood of $[0,1]$, $Z(f)\cap [0,1]$ is finite. Now suppose the result is true for $d-1$ and let $f,$ $\eps$, and $A_1,\ldots,A_d$ be as in the statement of the lemma. Since $f$ is analytic and not identically zero on a neighborhood of $[0,1]^d$, there are finitely many points $p_d\in [0,1]$ so that the (truncated) hyperplane $\{x_d=p_d\}\cap [0,1]^d$ is contained in $Z(f) \cap [0,1]^d$. Use Corollary \ref{ManyPointsAchieveLocalDim} to select a point $p_d\in (0,1)$ so that $A_d$ has local dimension $\geq\dim(A_d)-\eps$ at $p_d$, and the truncated hyperplane $\{x_d=p_d\}\cap [0,1]^d$ is not contained in $Z(f)\cap [0,1]^d$. Then the function $(x_1,\ldots,x_{d-1})\mapsto f(x_1,\ldots,x_{d-1},p_d)$ is analytic and not identically zero on a neighborhood of $[0,1]^{d-1}$. We now find the point $p=(p_1,\ldots,p_{d-1},p_d)$ by applying the induction hypothesis to this function and the sets $A_1,\ldots,A_{d-1}$.
\end{proof}

We are now ready to prove Theorem \ref{dimension-expander}. For the reader's convenience we will restate it here.
\begin{dimension-expanderThm}
For every $0< \alpha< 1$, there exists $c=c(\alpha)>0$ so that the following holds. Let $U\subset\RR^2$  be a connected open set that contains $[0,1]^2$ and let $P\colon U \to\RR$ be analytic (resp.~polynomial). Then either $P$ is an analytic (resp.~polynomial) special form, or for every pair of Borel sets $A,B\subset [0,1]$ of dimension at least $\alpha$, we have 
\[
\dim P(A \times B) \ge \alpha+c.
\] 
\end{dimension-expanderThm}
\begin{proof}
Let $\eps=\eps(\alpha)$ and $\delta_0=\delta_0(\alpha,P)$ be the quantities obtained by applying Proposition \ref{countingNumberQuadruplesNablePWedgeP} to $P$ with parameters $\alpha$ and $\kappa = \alpha/2$. We will prove that 
\begin{equation}\label{deimPAB}
\dim(P(A,B))>\alpha+\eps/2.
\end{equation} 

Without loss of generality we can suppose that $P(A,B)\subset [0,1]$. Suppose that $P$ is not a special form and \eqref{deimPAB} is false. Since $P$ is not a special form, we have that none of $\partial_x P$, $\partial_y P$ or $\partial_{xy} P$ vanish identically. By Lemma~\ref{whenIsPSpecialFormProp}, $K_P$ does not vanish identically. Define
\begin{equation*}
f= (\partial_x P)(\partial_y P)(\partial_{xy}P)^3(K_P).
\end{equation*}
Note that initially $f$ is not defined on $Z(\partial_{xy}P)$, but if we define $f$ to be $0$ on this set, then $f$ is analytic on a neighborhood of $[0,1]^2$, so we can use Lemma \ref{findGoodPoint} to select a point $p=(p_1,p_2) \in (0,1)^2\backslash Z(f)$ so that $A$ and $B$ have local dimension $\geq \alpha-\eps/20$ at $p_1$ and $p_2$, respectively. Let $I_0,J_0\subset [0,1]$ be closed intervals containing $p_1$ and $p_2$ respectively so that $(I_0\times J_0)\cap Z(f)=\emptyset$ and $p$ is contained in the interior of $I_0\times J_0$. Then there is a number $c>0$ so that 
\begin{equation}\label{lowerBoundKeyQuantities}
|\partial_x P(x,y)|\geq c,\ |\partial_y P(x,y)|\geq c,\ |\partial_{xy}P(x,y)|\geq c,\ |K_p(x,y)|\geq c\quad\forall\ (x,y)\in I_0\times J_0. 
\end{equation}
Define $A^\prime = A\cap I_0$ and $B^\prime = B\cap J_0$. Let $C = P(A',B')$.

Since $\dim(A')\geq \alpha-\eps/20$ and $\dim(B')\geq \alpha-\eps/20$, by Frostman's lemma (see e.g.~\cite[Theorem 8.8]{Mat}), there are Borel  measures $\mu$ and $\nu$, supported on $A'$ and $B'$ respectively, with $\mu(A')>0,$ and $\nu(B')>0$, so that
\begin{equation}\label{FrostmanConditionIneq}
\begin{split}
&\mu(A\cap I)\leq |I|^{\alpha-\eps/10}\ \textrm{for every interval}\ I,\\
&\nu(B\cap J)\leq |J|^{\alpha-\eps/10}\ \textrm{for every interval}\ J.
\end{split}
\end{equation}
In particular, $\mu$ and $\nu$ have no atoms. Let $\lambda$ be the pushforward of $\mu\times\nu$ by $P$, i.e. for each Borel set $U \subset \RR$, we have 
\begin{equation*}
\lambda(U) = (\mu\times\nu) (\{(x,y)\in A\times B: P(x,y) \in U\}).
\end{equation*}
Observe that $\lambda$ is a measure supported on $C$, and $\lambda(C)>0$. Let $\delta_1\leq\delta_0$ be a number of the form $2^{-k_1}$ for some positive integer $k_1$. In what follows, all implicit constants will be independent of $\delta_1$. If $\delta_1>0$ is selected sufficiently small (depending only on $\alpha$), 
then we will eventually show that \eqref{deimPAB} must hold.

Let $\{[x_i, x_i+r_i]\}$ be a covering of $C$ by intervals of length at most $\delta_1$, with
\begin{equation}\label{sumOfRi}
\sum_i r_i^{\alpha+(3/5)\eps}<1.
\end{equation}
Such a covering must exist, since by assumption we have $\dim(C)\leq \alpha+\eps/2$. Without loss of generality we can suppose that each interval is of the form $[n2^{-k},(n+1)2^{-k}]$ for some $k\geq k_1$ and some integer $1\leq n< 2^{-k_1}$. 

For each $k\geq k_1$, let $m_k = \sum \lambda([x_i,x_i+r_i])$, where the sum is taken over all intervals of length $r_i=2^{-k}$. Since $\sum_{k\geq k_1} m_k \geq \lambda(C)>0$ and $\sum_{k\geq k_1}k^{-2}\le \pi^2/6$, there must exist an index $k_2\geq k_1$ with $m_{k_2} \gtrsim \lambda(C)k_2^{-2}$. Define $\delta = 2^{-k_2}$, so $0<\delta\leq\delta_1$.

At this point it will be helpful to introduce some additional notation. We say $X\lessapprox Y$ if there is an absolute constant $K_1>0$ and a constant $K_2 >0$ (which may depend on $\lambda(C)$)  so that $A\leq K_2 |\log\delta|^{K_1} B$. If $A\lessapprox B$ and $B\lessapprox A$, we say $A\approx B$. In the arguments that follow, we can always take $K_1\leq 100$. With this notation we have $1\lessapprox m_{k_2}$. 

Let $C^\prime=\bigcup [x_i,x_i+r_i]$, where the union is taken over all intervals of length $r_i=\delta$ in the covering. Then $\lambda(C^\prime)=m_{k_2}\gtrapprox 1$, and $C^\prime$ is a union of interior-disjoint intervals of the form $[n\delta,(n+1)\delta]$, where $1\leq n<\delta^{-1}$ is an integer. After dyadic pigeonholing, we can select a set $C^{\prime\prime}\subset C^{\prime}$ that is again a union of $\delta$-intervals with
\begin{equation}\label{massOfCpp}
\lambda(C^{\prime\prime})\gtrapprox 1,
\end{equation} 
and each $\delta$-interval $I\subset C^{\prime\prime}$ has measure $\lambda(I)\sim \lambda(C^{\prime\prime})(\delta|C^{\prime\prime}|^{-1})$. 
Note that 
\begin{equation}\label{sizeOfCalC}
\delta^{1-\alpha+\eps/10}\lessapprox |C^{\prime\prime}|\le \delta^{1-\alpha-(3/5)\eps}.
\end{equation}
The lower bound on $|C''|$ follows from \eqref{FrostmanConditionIneq} and \eqref{massOfCpp}, while the upper bound follows from \eqref{sumOfRi}. 

Cover $I_0\times J_0$ by squares of the form $[n\delta,(n+1)\delta]\times[m\delta,(m+1)\delta]$, and let $\mathcal{Q}_0$ be the set of squares for which $P(Q)\cap C^{\prime\prime}\neq\emptyset$. We have
\begin{equation*}
\sum_{Q\in\mathcal{Q}_0}(\mu\times\nu)(Q)\geq \lambda(C^{\prime\prime})\gtrapprox 1. 
\end{equation*}
By dyadic pigeonholing, there is a set $\mathcal{Q}\subset\mathcal{Q}_0$ and numbers $m_x,m_y$ so that
\begin{equation}\label{QCapturesMostMass}
\sum_{Q\in\mathcal{Q}}(\mu\times\nu)(Q)\gtrapprox 1,
\end{equation}
$m_x\leq (\mu(\pi_x(Q))\leq 2m_x$, and $m_y\leq (\mu(\pi_x(Q))\leq 2m_y$ for each $Q\in\mathcal{Q}$. In particular,
\begin{equation*}
\#\mathcal{Q} \approx (m_xm_y)^{-1}.
\end{equation*}

Our next goal is to show that $\mathcal{Q}$ resembles (a dense subset of) a Cartesian product. Let $A_1=\bigcup_{Q\in\mathcal{Q}}\pi_x(Q)$ and let $B_1 = \bigcup_{Q\in\mathcal{Q}}\pi_y(Q).$ We have that $A_1$ and $B_1$ are interior-disjoint unions of $\delta$-intervals contained in $I_0$ and $J_0$, respectively; $m_x\leq \mu(I)<2m_x$ for each $\delta$ interval $I\subset A_1$; $m_y\leq \nu(J)<2m_y$ for each $\delta$ interval $J\subset B_1$; and $Q\subset A_1\times B_1$ for each square $Q\in\mathcal{Q}$. 

Clearly $A_1\times B_1$ is a union of $\delta$-squares, and there are at least $\#\mathcal{Q}\approx (m_xm_y)^{-1}$ squares in this union. On the other hand, we have $m_xm_y \leq (\mu\times\nu)(Q)< 4 m_x m_y$ for each $\delta$-square $Q\subset A_1\times B_1$, so $A_1\times B_1$ is a union of $\lesssim (m_xm_y)^{-1}$ $\delta$-squares. 

Our final task is to show that $A_1$ and $B_1$ each have size  roughly $\delta^{1-\alpha}$. First, since $\mathcal{Q}$ contains $\gtrapprox (m_xm_y)^{-1}$ squares and $A_1\times B_1$ contains $\lesssim (m_xm_y)^{-1}$ squares, there exists $y_0\in J_0$ so that the line $y=y_0$ intersects $\gtrapprox  \delta^{-1}|A_1|$ squares from $\mathcal{Q}$. By \eqref{sizeOfCalC} and \eqref{lowerBoundKeyQuantities}, we have
\begin{equation*}
\delta^{1-\alpha-\eps/10}\gtrapprox |C''|\geq \bigcup_{\substack{Q\in\mathcal{Q}\\
Q\cap \{y =y_0\}\neq\emptyset}} P(Q) \gtrsim \delta\#(\{Q\in\mathcal{Q}\colon Q\cap \{y=y_0\}\neq\emptyset\})\gtrapprox |A_1|,
\end{equation*}
and an identical argument shows that
\begin{equation*}
|B_1|\lessapprox  \delta^{1-\alpha-\eps/10}.
\end{equation*}

But \eqref{FrostmanConditionIneq} and \eqref{QCapturesMostMass} implies that $|A_1|,|B_1|\gtrapprox \delta^{1-\alpha+\eps/10}.$ 
We conclude that
\begin{equation}\label{sizeOfA0B0}
\delta^{1-\alpha+\eps/10}\lessapprox |A_1|,|B_1|\lessapprox  \delta^{1-\alpha-\eps/10}.
\end{equation}

Next, if $Q,Q^\prime\in \mathcal{Q}$ are two squares with the property that $P(Q)$ and $P(Q^\prime)$ intersect a common $\delta$-interval from $C''$, and if $c>0$ is the constant from \eqref{lowerBoundKeyQuantities}, then there are points $(x,y)\in N_{c^{-1}\delta}(Q)$ and $(x',y')\in N_{c^{-1}\delta}(Q')$ so that $P(x,y) = P(x',y')$. In particular, if we define $A_2 = N_{c^{-1}\delta}A_1$ and $B_2 = N_{c^{-1}\delta}B_1$, then
\begin{equation*}
\begin{split}
\mathcal{E}_{\delta}\big(& \{x,x',y,y'\in A_2^2\times B_2^2\colon P(x,y) = P(x',y')\big)\\
&\gtrsim \#\{ (Q,Q') \in \mathcal{Q}^2\colon P(Q)\ \textrm{and}\ P(Q')\ \textrm{intersect a common $\delta$-interval from}\ C^{\prime\prime}\}\\
&\gtrsim (\#\mathcal{Q})^2(\delta|C''|^{-1})\\
&\gtrapprox \delta^{-3\alpha + (4/5)\eps},
\end{split}
\end{equation*}
i.e. there is an absolute constant $K_1>0$ and a constant $K_2>0$ (which may depend on $\lambda(C)$) so that
\begin{equation}\label{lowerBoundNumberQuadruples}
\mathcal{E}_{\delta}\big( \{x,x',y,y'\in A_2^2\times B_2^2\colon P(x,y) = P(x',y')\big)\geq K_2^{-1}|\log\delta|^{-K_1}\delta^{-3\alpha + (4/5)\eps}.
\end{equation}

Finally, by \eqref{FrostmanConditionIneq}, the sets $A_2$ and $B_2$ obey the hypotheses of Proposition \ref{countingNumberQuadruplesNablePWedgeP} with parameters $\alpha$ and $\kappa = \alpha/2$. Since $\delta\leq\delta_0$, there exists a constant $K_3$ (which may depend on $\alpha$) so that 
\begin{equation}\label{upperBoundNumberQuadruples}
\mathcal{E}_{\delta}\big( \{x,x',y,y'\in A_2^2\times B_2^2\colon P(x,y) = P(x',y')\big)\leq K_3 \delta^{-3\alpha + \eps}.
\end{equation}
Thus if $\delta_1$ is selected sufficiently small (depending on $K_1,K_2,K_3$), then \eqref{upperBoundNumberQuadruples} contradicts \eqref{lowerBoundNumberQuadruples}. We conclude that \eqref{deimPAB} must hold.
\end{proof}


\begin{thebibliography}{9}
%
\bibitem{BB} M. Bays and E. Breuillard. Projective geometries arising from Elekes-Szab\'o problems. \emph{Sci. Ec. Norm. Super.} 54(3): 627--681, 2021.
%
%

\bibitem{Bl} W. Blaschke. \emph{Einf\"uhrung in die Geometrie der Waben}. Birkh\"auser, 1955.
%

\bibitem{B03} J. Bourgain. On the Erd\H{o}s-Volkmann and Katz-Tao ring conjectures. \emph{Geom. Funct. Anal} 15(1): 334--365. 2003.

%
\bibitem{B10} J. Bourgain. The discretized sum-product and projection theorems. \emph{J. Anal. Math} 112(1): 193--236, 2010. 

%
 \bibitem{BG2} J. Bourgain and A. Gamburd.  Uniform expansion bounds for Cayley graphs of $SL_2(F_p)$. \emph{Ann. Math.} 167(2): 625-642, 2008.

 \bibitem{BG} J. Bourgain and A. Gamburd. On the spectral gap for finitely-generated subgroups of $SU(2)$. \emph{Invent. Math.} 171: 83--121, 2008.
%

\bibitem{BKT} J. Bourgain, N. Katz, and T. Tao. A sum-product estimate in finite fields, and applications. \emph{Geom. Funct. Anal.} 14: 27--57, 2004. 

\bibitem{BG3} Y. Brudnyi and I. Ganzburg. On an extremal problem for polynomials in $n$ variables. \emph{Math. USSR Izvestijia} 7: 345--356, 1973.
%
%

\bibitem{Chr}M.~Christ. On Trilinear Oscillatory Integral Inequalities and Related Topics. \emph{arXiv:2007.12753}, 2020. 

%

\bibitem{DG} Dvir, Z., Gopi, S., On the number of rich lines in truly high dimensional sets. In: \emph{Proc. 31st Annu. Sympos. Comput. Geom.(SoCG 2015)}, Leibniz International Proceedings in Informatics (LIPIcs), 34, 584--598, 2015.
%
\bibitem{Gra}
L. Grafakos (2008). Classical Fourier Analysis. Springer. ISBN 978-0-387-09431-1.
%
\bibitem{EM} G.A. Edgar and C. Miller.  Borel subrings of the reals. \emph{Proc. Amer. Math. Soc.} 131: 1121--1129, 2003. 
%

\bibitem{El} G. Elekes. Circle grids and bipartite graphs of distances. \emph{Combinatorica} 15: 167--174, 1995.
%
%
\bibitem{ER}G. Elekes and L. R\'onyai. A combinatorial problem on polynomials and rational functions. \emph{J. Combin. Theory Ser. A.}89: 1--20, 2000.
%
\bibitem{ES}G. Elekes and E. Szab\'o. How to find groups? (And how to use them in Erd\H{o}s geometry?). \emph{Combinatorica} 32: 537--571, 2012.
%
\bibitem{ErdSze}P. Erd\H{o}s and E. Szemer\'edi. On sums and products of integers. {\it Studies in Pure Mathematics} 213--218, 1983.

\bibitem{EV} P. Erd\H{o}s and B. Volkmann. Additive gruppen mit vorgegebener Hausdorffscher dimension. \emph{J. Reine Angew. Math.} 221: 203--208, 1966.
%
\bibitem{GKZ} L. Guth, N. Katz, and J. Zahl. On the discretized sum-product problem. \emph{Int. Math. Res. Not.} Vol. 2021, Issue 13, 9769--9785, 2021.
%
%
\bibitem{H} W. He. Discretized sum-product estimates in matrix algebras. \emph{J. Anal. Math.} 139: 637--676, 2019.
%
\bibitem{HS} W. He and N. de Saxc\'e. Sum-product for real Lie groups. \emph{J. Eur. Math. Soc.} 23(6): 2127--2151, 2021. 
%
\bibitem{Izo} A.~Izosimov. Curvature of Poisson pencils in dimension three. \emph{Differential Geom. Appl.} 31: 557--567, 2013. 
%
\bibitem{JRT}
Y.~Jing, S.~Roy, C-M. Tran. Semialgebraic methods and generalized sum-product phenomena. \emph{Disc. Anal.} 2022(18): 1--22, 2022.

%
\bibitem{Ka} R. Kaufman. On Hausdorff dimension of projections, \emph{Mathematika} 15: 153--155, 1968.
%
\bibitem{KM} R. Kaufman and P. Mattila. Hausdorff dimension and exceptional sets of linear transformations. \emph{Ann. Acad. Sci. Fenn. A Math.} 1: 387--392, 1975.
%
\bibitem{KT}N. Katz and T. Tao. Some connections between Falconer's distance set conjecture and sets of Furstenburg type. \emph{New York J. Math.} 7: 149--187. 2001.
%
\bibitem{KS} T. Keleti and P. Shmerkin.  New Bounds on the Dimensions of Planar Distance Sets. \emph{Geom. Funct. Anal.} 29: 1886--1948, 2019.
%
%
%

\bibitem{Lo} S. \L{}ojasiewicz. Sur le probl\'eme de la division. \emph{Studia Math}. 18: 87--136, 1959.

\bibitem{Lo2}  S. \L{}ojasiewicz. \emph{Ensembles semianalytiques}.  Cours Facult\'e des Sciences d'Orsay,
Inst. Hautes \'Etudes Sei. Bures-sur-Yvette, 1965.

\bibitem{Mar} J.M.~Marstrand. Some fundamental geometrical properties of plane sets of fractional dimensions. \emph{Proc. Lond. Math. Soc.} 3: 257--302, 1954.


\bibitem{MRSW} M. Makhul, O. Roche-Newton, S. Stevens, and A. Warren. The Elekes-Szab\'o problem and the uniformity conjecture. \emph{Israel J. Math.} 248: 39--66, 2022.

\bibitem{Mat} P.~ Mattila. Geometry of sets and measures in Euclidean spaces. Cambridge University Press, Cambridge, 1995. 
%

%
\bibitem{Ob12} D. Oberlin. Restricted Radon transforms and projections of planar sets. \emph{Canad. Math. Bull.}  55: 815--820, 2012. 
%
%
\bibitem{Or} T.~Orponen. On the distance sets of Ahlfors-David regular sets. \emph{Adv. Math.} 307: 1029--1045, 2017.


\bibitem{OS} T.~Orponen and P.~Shmerkin. On the Hausdorff dimension of Furstenberg sets and orthogonal projections in the plane. To appear, \emph{Duke Math. J.} {\tt arXiv:2106.03338}, 2021.

\bibitem{OS2} T.~Orponen and P.~Shmerkin. Projections, Furstenberg sets, and the ABC sum-product problem.   {\tt arXiv:2301.10199}, 2023.

%
%
\bibitem{PS} Y. Peres and W. Schlag. Smoothness of projections, Bernoulli convolutions, and the dimension of exceptions. \emph{Duke Math. J.} 102: 193--251, 2000.

\bibitem{RaSh} O. E. Raz, M. Sharir. The number of unit-area triangles in the plane: Theme and variations. \emph{Combinatorica}. 37: 1221--1240, 2017.
%
%
\bibitem{RSS2}O. E. Raz, M. Sharir, Solymosi. On triple intersections of three families of unit circles. \emph{Discrete Comput. Geom.} 54: 930--953, 2015.
%
%
\bibitem{RSS}O. E. Raz, M. Sharir, Solymosi. Polynomials vanishing on grids: The Elekes-R\'onyai problem revisited. \emph{Amer. J. Math.}, 138: 1029--1065, 2016.
%
\bibitem{RSZ}O. E. Raz, M. Sharir, and F. de Zeeuw. Polynomials vanishing on Cartesian products: The Elekes-Szab\'o Theorem revisited. \emph{Duke Math. J.} 165(18): 3517--3566, 2016. 

 \bibitem{RS}
 O. E. Raz and Z. Shem-Tov.
 Expanding polynomials: A generalization of the Elekes-R\'onyai theorem to $d$ variables. \emph{Combinatorica}, 40(5): 721--748, 2020. 

%
\bibitem{S} N. de Saxc\'e. A product theorem in simple Lie groups. \emph{Geom. Funct. Anal.} 25(3): 915--941, 2015.
%
\bibitem{Shm2} P. Shmerkin. On the Hausdorff dimension of pinned distance sets. \emph{Israel J. Math.} 230: 949--972, 2019.

\bibitem{ShaSoly}
M. Sharir and J. Solymosi.
Distinct distances from three points.
\emph{Combinat. Probab. Comput.} 25: 623--632, 2016.

\bibitem{Shm}
P. Shmerkin. A nonlinear version of Bourgain's projection theorem. \emph{J. Eur. Math. Soc.} 25(10): 4155--4204, 2023.
%
%

%
\bibitem{W}H. Wang. Exposition of Elekes Szab\'o paper. {\tt arXiv:2003.01636}, 2015. 

\end{thebibliography}
\end{document}